\documentclass[12pt]{amsart}
\usepackage{SSdefn}
\usepackage{eucal}
\usepackage{extarrows}
\usepackage[shortlabels]{enumitem}

\newcommand{\defi}[1]{\textit{#1}}
\newcommand{\bone}{\mathbf{1}}
\newcommand{\lpp}{(\!(}
\newcommand{\rpp}{)\!)}
\renewcommand{\top}{\mathrm{top}}
\newcommand{\fin}{\mathrm{fin}}
\newcommand{\MOD}{\mathrm{MOD}}

\newcommand{\stacks}[1]{\cite[\href{http://stacks.math.columbia.edu/tag/#1}{Tag~#1}]{stacks}}

\title{Symmetric ideals of the infinite polynomial ring}

\author{Rohit Nagpal}
\address{Department of Mathematics, University of Michigan, Ann Arbor, MI}
\email{\href{mailto:rohitna@umich.edu}{rohitna@umich.edu}}
\urladdr{\url{http://www-personal.umich.edu/~rohitna/}}

\author{Andrew Snowden}
\address{Department of Mathematics, University of Michigan, Ann Arbor, MI}
\email{\href{mailto:asnowden@umich.edu}{asnowden@umich.edu}}
\urladdr{\url{http://www-personal.umich.edu/~asnowden/}}

\thanks{RN was partially supported by NSF DMS-1638352. AS was supported by NSF DMS-1453893.}

\date{\today}

\begin{document}

\begin{abstract}
Let $R=\bC[\xi_1,\xi_2,\ldots]$ be the infinite variable polynomial ring, equipped with the natural action of the infinite symmetric group $\fS$. We classify the $\fS$-primes of $R$, determine the containments among these ideals, and describe the equivariant spectrum of $R$. We emphasize that $\fS$-prime ideals need not be radical, which is a primary source of difficulty. Our results yield a classification of $\fS$-ideals of $R$ up to copotency. Our work is motivated by the interest and applications of $\fS$-ideals seen in recent years.
\end{abstract}

\maketitle

\tableofcontents

\section{Introduction}

Cohen \cite{cohen, cohen2} showed that ideals in the infinite variable polynomial ring $R=\bC[\xi_1, \xi_2, \ldots]$ that are stable under the infinite symmetric group $\fS$ satisfy the ascending chain condition; in other words, $R$ is $\fS$-noetherian. There has been a surge of interest surrounding this theorem in recent years: it has found applications in algebraic statistics \cite{AschenbrennerHillar, kfactor, HillarSullivant}, algebraic geometry \cite{draismaeggermont, draismakuttler}, and the theory of configuration spaces \cite{ramos}; additionally, a number of authors have worked to understand aspects of $\fS$-ideals in more detail \cite{def, GunturkunNagel, KLS, LNNR, LNNR2, raicu, NR, NR2}. In this paper, we undertake a systematic study of $\fS$-ideals of $R$. We give a complete description of the equivariant spectrum of $R$, which yields a classification of all $\fS$-ideals up to copotency. This builds on the work of our previous paper \cite{svar} where we classified the radical $\fS$-ideals of $R$.

\subsection{The main problem}

We now formulate the main problem under consideration in this paper, and provide some motivation for it. For this discussion, it will be useful to introduce a piece of terminology: we say that two ideals $\fa$ and $\fb$ in a commutative ring are \defi{copotent} if there exist positive integers $n$ and $m$ such that $\fa^n \subset \fb^m \subset \fa$.

Let $T$ be a noetherian commutative ring. Typically, one cannot hope to give a complete classification of ideals in $T$. However, one does have the following picture:
\begin{enumerate}
\item Every ideal of $T$ is copotent with its radical; in fact, each copotency class contains a unique radical ideal.
\item If $\fa$ is a radical ideal of $T$ then there exist prime ideals $\fp_1, \ldots, \fp_n$ of $T$ such that $\fa=\fp_1 \cap \cdots \cap \fp_n$. Moreover, this expression is unique up to permutation if we assume that there are no containments among the $\fp_i$'s.
\item The set of prime ideals of $T$ forms a reasonable geometric object $\Spec(T)$.
\end{enumerate}
Thus, with a sufficient understanding of $\Spec(T)$, one can hope to answer many questions about ideals of $T$, or even $T$-modules.\footnote{We note that understanding ideals up to copotency is sufficient for many applications. For example, if $M$ is a finitely generated $T$-module then (a) and (b) show that $T$ admits a finite filtration whose quotients have prime annihilator; this is a common tool used to study modules.}

The goal of this paper, broadly speaking, is to understand the $\fS$-ideals of $R$. Thus we would like a picture similar to the above. Unfortunately, the most direct analog does not hold: indeed, the following example shows that (a) fails.

\begin{example} \label{ex:not-copotent}
Let $\fa=\langle \xi_i^2 \rangle_{i \ge 1}$ and $\fb=\langle \xi_i \rangle_{i \ge 1}$. These are $\fS$-ideals of $R$, and one easily sees that $\fb$ is the radical of $\fa$. However, $\fa$ and $\fb$ are not copotent. Indeed, $\fb^n$ contains the squarefree monomial $\xi_1 \cdots \xi_n$, and so no power of $\fb$ is contained in $\fa$.
\end{example}

Cohen's theorem suggests that $\fS$-ideals of $R$ might be well-behaved. However, the above example seems to indicate that they suffer from pathological behavior. This apparent conflict is resolved as follows. First, there is a full-fledged equivariant analog of commutative algebra. And second, the proper interpretation of Cohen's theorem is that $\fS$-ideals should be well-behaved from the equivariant perspective, not necessarily the ordinary one. To explain what this means in more detail, we introduce the two most important concepts from equivariant commutative algebra:

\begin{definition}
A non-unital $\fS$-ideal $\fp$ of $R$ is \defi{$\fS$-prime} if $\fa \fb \subset \fp$ implies $\fa \subset \fp$ or $\fb \subset \fp$ for $\fS$-ideals $\fa$ and $\fb$.
\end{definition}

\begin{definition}
Let $\fa$ be an $\fS$-ideal of $R$. The \defi{$\fS$-radical} of $\fa$, denoted $\rad_{\fS}(\fa)$, is the sum of all $\fS$-ideals $\fb$ of $R$ such that $\fb^n \subset \fa$ for some $n$. We say that $\fa$ is \defi{$\fS$-radical} if $\fa=\rad_{\fS}(\fa)$.
\end{definition}

\begin{example}
The ideal $\fa=\langle \xi_i^2 \rangle_{i \ge 1}$ from Example~\ref{ex:not-copotent} is $\fS$-prime. This is not difficult to prove directly (see Proposition~\ref{prop:powers-prime}). It is also $\fS$-radical, as is every $\fS$-prime. This example is important, as it shows that $\fS$-prime ideals need not be radical.
\end{example}

We said that Cohen's theorem suggests $\fS$-ideals should be well-behaved from the equivariant point of view. A concrete instance of this principle is the following set of statements, which are the equivariant analogs of (a)--(c):
\begin{enumerate}[label=(\alph*')]
\item Every $\fS$-ideal of $R$ is copotent with its $\fS$-radical, and each copotency class of $\fS$-ideals contains a unique $\fS$-radical ideal.
\item An $\fS$-radical ideal of $R$ is a finite intersection of $\fS$-prime ideals, with the usual uniqueness.
\item The $\fS$-prime ideals of $R$ naturally form a topological space $\Spec_{\fS}(R)$, called the \defi{$\fS$-spectrum} of $R$.
\end{enumerate}
These statements, along with many other basic facts about equivariant commutative algebra, are established in \S \ref{s:eqca}. We thus see that knowing $\Spec_{\fS}(R)$ allows us to classify $\fS$-ideals up to copotency. We are thus led to our main problem:

\begin{problem} \label{mainproblem}
Describe the $\fS$-spectrum of $R$.
\end{problem}

Our two main theorems, Theorems~\ref{mainthm} and~\ref{mainthm2}, provide a complete solution to this problem.

\begin{remark}
The main theorem of \cite{svar} classifies $\fS$-primes that are radical (in the usual sense). These  turn out to constitute a rather small piece of $\Spec_{\fS}(R)$, and so this theorem does not come close to solving Problem~\ref{mainproblem}. The results from \cite{svar} will play an important role in this paper though.
\end{remark}

\subsection{Review of the radical case}

Before describing the results of this paper, we recall some ideas and results from \cite{svar}. Let $\fX=\Spec(R)$ be infinite affine space. We say that a $\bC$-point $x=(x_i)_{i \ge 1}$ of $\fX$ is \defi{finitary} if the coordinates of $x$ assume only finitely many values, i.e., the set $\{x_i \mid i \ge 1\}$ is finite. Given a finite partition $\cU=\{U_1, \ldots, U_r\}$ of the set $[\infty]$ of positive integers, we let $\fX_{\cU}$ be  the subset of $\fX$ consisting of points $x$ such that $x_i=x_j$ whenever $i$ and $j$ belong to the same part of $\cU$. The space $\fX_{\cU}$ is isomorphic to affine space $\bA^r$, and the union $\bigcup \fX_{\cU}$ is exactly the finitary locus. It is not difficult to show that any non-zero $\fS$-ideal of $R$ contains a discriminant (see Proposition~\ref{prop:disc}). From this it follows that any proper $\fS$-stable Zariski closed subset of $\fX$ is contained in the finitary locus. For this reason, finitary points, partitions of $[\infty]$, and discriminants are all major players in our discussion.

Let $\cU=\{U_1,\ldots,U_r\}$ be as above, and put $\lambda_i=\# U_i$; reordering if necessary, we assume that $\lambda_i \ge \lambda_{i+1}$. Thus $\lambda$ is what we call a \defi{partition of $\infty$}, and we say that $\cU$ has \defi{type} $\lambda$. Let $Z$ be an irreducible Zariski closed subset of $\bU^r$, the open subset of $\bA^r$ where all coordinates are distinct. Let $\fX_{[\cU]}(Z)$ be the image of $Z$ under the isomorphism $\bA^r \cong \fX_{\cU}$, let $\fX_{[\lambda]}(Z)$ be the union of the $\fS$-translates of $\fX_{[\cU]}(Z)$, and let $\fX_{\lambda}(Z)$ be the Zariski closure of $\fX_{[\lambda]}(Z)$. Let $\fP(\lambda; Z)$ be the $\fS$-ideal of $R$ consisting of all functions that vanish on $\fX_{\lambda}(Z)$. The main theorem of \cite{svar} states that the ideals $\fP(\lambda; Z)$ are exactly the non-zero radical $\fS$-primes of $R$. (The parametrization is nearly unique: $\fP(\lambda;Z)=\fP(\lambda';Z')$ if and only if $\lambda=\lambda'$ and $Z'=\sigma Z$ for some $\sigma \in \Aut(\lambda) \subset \fS_r$.) When $Z=\bU^r$, we often drop $Z$ from the notation, e.g., we write $\fX_{\lambda}$ in place of $\fX_{\lambda}(Z)$.

The above concepts are central to the remainder of the paper. We give a number of examples to illuminate their nature. The non-trivial assertions below can be deduced from results in \cite{svar}, but are not essential to this paper.


\begin{example} \label{ex:rad}
In what follows, $x=(x_i)$ denotes a $\bC$-point of $\fX$.
\begin{enumerate}
\item Let $\lambda=(\infty)$ and $Z=\{0\} \subset \bA^1$. Then $\fX_{[\lambda]}(Z)=\fX_{\lambda}(Z)$ is just the origin, and so $\fP(\lambda; Z)=\langle \xi_i \rangle_{i \ge 1}$.
\item Let $\lambda=(\infty)$ and $Z=\bA^1$. Then $\fX_{[\lambda]}=\fX_{\lambda}$ consists of those points $x$ having all coordinates equal, i.e., $x_i=x_j$ for all $i$ and $j$. Thus $\fP(\lambda)$ is generated by the elements $\xi_i-\xi_j$ for $i,j \ge 1$.
\item Let $\lambda=(\infty,\infty)$ and $Z=\bU^2$. Then $\fX_{[\lambda]}$ consists of all points $x$ such that the coordinates of $x$ assume exactly two values, and each value is attained infinitely often. The closure $\fX_{\lambda}$ consists of those points $x$ whose coordinates take at most two values. It is not difficult to see that $\fP(\lambda)$ is the radical of the ideal generated by the $\fS$-orbit of the discriminant $\Delta_3=(\xi_1-\xi_2)(\xi_1-\xi_3)(\xi_2-\xi_3)$; in fact, this ideal is already radical by \cite[Theorem~2.4]{lovasz}.
\item Let $\lambda=(\infty,n)$ and $Z=\bU^2$. Then $\fX_{[\lambda]}$ consists of all points $x$ such that the coordinates of $x$ assume exactly two values, one of which occurs exactly $n$ times. The closure $\fX_{\lambda}$ consists of those points $x$ whose coordinates take at most two values, one of which is attained at most $n$ times. The ideal $\fP(\lambda)$ is the radical of the ideal generated by the $\fS$-orbits of $\Delta_3$ and $\prod_{i=1}^{n+1}(\xi_{2i-1}-\xi_{2i})$.
\item Let $\lambda=(\infty,\infty)$ and let $Z \subset \bU^2$ be the intersection of the unit circle in $\bA^2$ with $\bU^2$. Then $\fX_{[\lambda]}(Z)$ consists of those points $x$ in $\fX_{[\lambda]}$ such that $a^2+b^2=1$, where $\{x_i\}_{i \ge 1}=\{a,b\}$. The set $\fX_{\lambda}(Z)$ consists of those points $x$ for which there exist $a,b \in \bC$ with $a^2+b^2=1$ such that $x_i \in \{a,b\}$ for all $i$. The ideal $\fP(\lambda; Z)$ is the radical of the ideal generated by the $\fS$-orbits of $\Delta_3$ and $(\xi_1-\xi_2)(\xi_1^2+\xi_2^2-1)$. \qedhere
\end{enumerate}
\end{example}

\subsection{The ideals $\fP(\lambda,e;Z)$} \label{ss:intro-P}

Let $\lambda=(\lambda_1, \ldots, \lambda_r)$ be an $\infty$-partition and let $Z$ be an irreducible closed subset of $\bU^r$, as in the previous section. Additionally, let $e=(e_1, \ldots, e_r)$ be a tuple of positive integers; we refer to $e$ as a \defi{weight vector} and the $e_i$'s as \defi{weights}. We assume that $e$ is \defi{reduced}, meaning that $e_i=1$ whenever $\lambda_i$ is finite. We now define an important $\fS$-ideal $\fP(\lambda,e;Z)$ associated to this data. It consists of those elements $f \in R$ that satisfy the following condition: given a partition $\cU$ of $[\infty]$ of type $\lambda$ and a point $x \in \fX_{[\cU]}(Z)$, we have
\begin{displaymath}
(\partial_{i_1}^{k_1} \cdots \partial_{i_r}^{k_r} f)(x) = 0
\end{displaymath}
whenever $i_1, \ldots, i_r$ are distinct indices and $0 \le k_j < e_{\pi(i_j)}$, where $\pi(i) \in [r]$ is such that $i \in U_{\pi(i)}$ and $\partial_i$ denotes partial derivative with respect to $\xi_i$. Once again, we often omit $Z$ from the notation when $Z=\bU^r$.

The ideals $\fP(\lambda, e; Z)$ are a focal point of this paper. We define them in somewhat greater generality, and in a different manner, in \S \ref{ss:P}. There we also prove the fundamental results that $\fP(\lambda,e; Z)$ is $\fS$-prime and has radical $\fP(\lambda;Z)$. We give two simple examples of these ideals here; see \S \ref{ss:intro-gen} for some more complicated examples.

\begin{example} \label{ex:P-inf-n-zero}
Let $\lambda=(\infty)$, $e=(n)$, and $Z=\{0\}$. Then $\fP(\lambda, e; Z)$ consists of all functions $f$ such that $(\partial_{i_1}^{k_1} \cdots \partial_{i_r}^{k_r} f)(0)=0$ whenever $i_1, \ldots, i_r$ are distinct indices and $0 \le k_j < n$. It follows that $\fP(\lambda, e; Z)=\langle \xi_i^n \rangle_{i \ge 1}$. It is easy to verifiy directly that this ideal is $\fS$-prime (see Proposition~\ref{prop:powers-prime}). It is also not difficult to show that these are the only $\fS$-primes with radical $\langle \xi_i \rangle_{i \ge 1}$ (see Lemma~\ref{lem:class-3-1}).
\end{example}

\begin{example} \label{ex:P-inf-n}
Let $\lambda=(\infty)$, $e=(n)$, and $Z=\bA^1$. Then $\fP(\lambda, e)$ consists of all functions $f$ such that $(\partial_{i_1}^{k_1} \cdots \partial_{i_r}^{k_r} f)(x)=0$ whenever $i_1, \ldots, i_r$ are distinct indices, $0 \le k_j < n$ for all $j$, and $x$ satisfies $x_i=x_j$ for all $i$ and $j$. One easily sees that $\fP(\lambda, e)$ contains $(\xi_i-\xi_j)^{2n-1}$ for all $i$ and $j$. In fact, these elements generate: see Theorem~\ref{thm:contract}.
%
\end{example}

\subsection{The classification theorem}

The following is one of the two primary theorems of this paper:

\begin{theorem}[Theorem~\ref{thm:class}] \label{mainthm}
Every non-zero $\fS$-prime of $R$ has the form $\fP(\lambda,e;Z)$.
\end{theorem}

We briefly summarize the main ideas of the proof. Let $\fp$ be a non-zero $\fS$-prime of $R$. Then $\rad(\fp)$ is a radical $\fS$-prime (see Proposition~\ref{prop:rad-prime}). By the main theorem of \cite{svar}, it follows that $\rad(\fp)$ has the form $\fP(\lambda; Z)$ for some $\infty$-partition $\lambda=(\lambda_1, \ldots, \lambda_r)$ and some closed subset $Z$ of $\bU^r$. Let $\cU$ be a partition of $[\infty]$ of type $\lambda$. Recall that $\fX_{[\lambda]}(Z)$ is the union of the $\fS$-translates of $\fX_{[\cU]}(Z)$, and that $\fX_{\lambda}(Z)=V(\fP(\lambda;Z))$ is the Zariski closure of $\fX_{[\lambda]}(Z)$. The closed subscheme $V(\fp)$ is an infinitesimal thickening of $\fX_{\lambda}(Z)$. The core geometric idea in the proof of Theorem~\ref{mainthm} is that $V(\fp)$ should be completely determined by its behavior in the formal neighborhood of $\fX_{[\lambda]}(Z)$; in fact, since $V(\fp)$ is $\fS$-equivariant and $\fX_{[\lambda]}(Z)$ is the \emph{disjoint} union of the translates of $\fX_{[\cU]}(Z)$, it should be determined by its behavior in the formal neighborhood of $\fX_{[\cU]}(Z)$. This suggests that we should be able to recover $\fp$ after localizing and completing along $\fX_{[\cU]}(Z)$. We show that this is indeed the case: see Theorems~\ref{thm:class1} and~\ref{thm:class2} for precise statements. Given these results, it is an elementary matter to complete the proof of Theorem~\ref{thm:class}. See \S \ref{s:class-overview} for more details.

While the above proof strategy is certainly plausible, there are a number of non-trivial results needed to implement it, including: \cite[Theorem~9.6]{svar}, which provides information about the support of $\fS$-equivariant $R$-modules; Theorem~\ref{thm:contract}, which computes certain ideal contractions; and a number of abstract results on equivariant commutative algebra proved in \S \ref{s:eqca}.

\begin{remark}
For $n \ge 1$, let $\fa_n$ be the ideal of $R$ generated by $(\xi_i-\xi_j)^n$ for all $i,j \ge 1$. Let $\lambda=(\infty)$ and $Z=\bA^1$. We have seen in Example~\ref{ex:P-inf-n} that $\fP(\lambda,(n);Z)=\fa_{2n-1}$. Theorem~\ref{mainthm} implies that the $\fP(\lambda,e;Z)$ are the only $\fS$-primes with radical $\fP(\lambda;Z)$. It therefore follows that $\fa_{2n}$ is \emph{not} $\fS$-prime.

It is not difficult to see directly that $\fa_2$ is not $\fS$-prime. Indeed, let $f=\xi_1-\xi_2$. We claim that $f \cdot \sigma(f) \in \fa_2$ for all $\sigma \in \fS$. The claim is clear for $\sigma=\id$. If $\sigma$ is $(2\;3)$ or $(1\;3)(2\;4)$ then it follows from the identities
\begin{align*}
2(\xi_1-\xi_2)(\xi_1-\xi_3) &= (\xi_1-\xi_2)^2+(\xi_1-\xi_3)^2 - (\xi_2-\xi_3)^2 \\
2(\xi_1-\xi_2)(\xi_3-\xi_4) &= (\xi_1-\xi_4)^2+(\xi_2-\xi_3)^2-(\xi_1-\xi_3)^2-(\xi_2-\xi_4)^2
\end{align*}
All other cases are similar, and so the claim follows. Since $f \not\in \fa_2$, as $\fa_2$ is generated by homogeneous polynomials of degree two and $f$ is linear, it follows that $\fa_2$ is not $\fS$-prime.

For $n>1$, we do not have a direct proof of the non-$\fS$-primality of $\fa_{2n}$, though one can verify this by computer for small values of $n$. After completing this work, K.~Jun \cite{jun} succeeded in finding a direct (but rather involved) argument. A good conceptual explanation for why even and odd powers of $\xi_i-\xi_j$ behave so differently is still lacking.
\end{remark}

%
%

\subsection{The containment theorem}

Theorem~\ref{mainthm} raises a natural question: how does one see the poset structure on the set of $\fS$-primes in terms of the $(\lambda,e;Z)$ parametrization? We give a complete solution in the following theorem, which is the second primary theorem of this paper.

Before stating the theorem, we must introduce some notation. For a finite set $\cI$, we let $\bA^{\cI}$ be the affine space with coordinates indexed by $\cI$. For a function $\phi \colon \cI \to \cE$ there is an induced morphism $\phi^* \colon \bA^{\cE} \to \bA^{\cI}$; if $\phi$ is injective then $\phi^*$ is a projection map, while if $\phi$ is surjective then $\phi^*$ is a multi-diagonal map (and thus a closed immersion).

\begin{theorem} \label{mainthm2}
Let $\fp=\fP(\lambda,e;Z)$ and $\fq=\fP(\mu,d;Y)$ be $\fS$-primes, where $\lambda=(\lambda_1, \ldots, \lambda_r)$ and $\mu=(\mu_1, \ldots, \mu_s)$. Then $\fp \subset \fq$ if and only if there exists a subset $\cE \subset \{1, \ldots, r\}$ and a function $\phi \colon \cE \to \{1, \ldots, s\}$ such that the following conditions hold:
\begin{itemize}
\item For $1 \le \beta \le s$ we have $\mu_{\beta} \le \sum_{\alpha \in \cE_{\beta}} \lambda_{\alpha}$, where $\cE_{\beta}=\{\alpha \in \cE \mid \phi(\alpha)=\beta\}$. (Note that this implies that $\cE_{\beta} \ne \emptyset$, and so $\phi \vert_{\cE} \colon \cE \to \{1, \ldots, s\}$ is surjective.)
\item For each $1 \le \beta \le s$ with $\mu_{\beta}=\infty$, we have $d_{\beta} \le \sum_{\alpha \in \cE_{\beta}, \lambda_{\alpha}=\infty} e_{\alpha}$.
\item Letting $i \colon \bA^s \to \bA^{\cE}$ be the multi-diagonal map $\phi^*$ and $p \colon \bA^r \to \bA^{\cE}$ the projection map associated to the inclusion $\cE \subset \{1, \ldots, s\}$, we have $Y \subset i^{-1}(\ol{p(Z)})$.
\end{itemize}
\end{theorem}

The criterion given in Theorem~\ref{mainthm2} is admittedly somewhat complicated looking. However, we emphasize that it boils down to testing finitely many containments between finite dimensional varieties. In particular, it is completely algorithmic to determine if the containment $\fp \subset \fq$ holds. In fact, the criterion has a geometric interpretation that is more intuitive, as we discuss in \S \ref{ss:intro-spec} below.

The proof of Theorem~\ref{mainthm2} is long and difficult. We briefly describe some of what goes into it. Let (C) denote the condition that there exists $(\cE, \phi)$ satisfying the requirements in the theorem. Thus Theorem~\ref{mainthm2} states that $\fp \subset \fq$ if and only if condition (C) holds. To show that $\fp \subset \fq$ implies (C), we prove the contrapositive: assuming (C) does not hold, we explicitly write down an element of $\fp$ not belonging to $\fq$. Now assume (C) holds. First consider the case where $\mu=(\infty)$, $d=(n)$, and $Y=\{0\}$; thus $\fq=\langle \xi_i^n \rangle_{i \ge 1}$ (Example~\ref{ex:P-inf-n-zero}). Suppose $\fp \not \subset \fq$. Then there exists an element $f \in \fp$ that does not belong to $\fq$, meaning there is some monomial appearing in $f$ in which all exponents are $<n$. Using this, we build an element of $\fp$ that has a specific form, and show that this contradicts (C); see \S \ref{sss:simple} for the details in this case. To handle the general case, we essentially reduce to the special case just considered; see \S \ref{ss:contain-statement} for a more detailed summary.

\begin{example}
Take $\lambda=(\infty,\infty)$ and let $Z$ be an irreducible curve in $\bU^2$. Let $\fp$ be the radical $\fS$-prime $\fP(\lambda;Z)$, and let $\fq_n$ be the $\fS$-prime $\fP((\infty),(n);\{0\})$ from Example~\ref{ex:P-inf-n-zero}. Then, according to Theorem~\ref{mainthm2}:
\begin{itemize}
\item The containment $\fp \subset \fq_1$ holds.
\item We have $\fp \subset \fq_2$ if and only if the Zariski closure of $Z$ in $\bA^2$ contains the origin.
\item The containment $\fp \subset \fq_n$ never holds for $n \ge 3$. \qedhere
\end{itemize}
\end{example}

\begin{remark}
Theorems~\ref{mainthm} and~\ref{mainthm2} yield a classification of $\fS$-ideals of $R$ up to copotency. Indeed, each copotency class contains a unique $\fS$-radical ideal, and $\fS$-radical ideals correspond to finite collections of $\fS$-having no containments. Such collections can be classified using the two theorems.
\end{remark}

\subsection{Generators} \label{ss:intro-gen}

Suppose that $\fp$ is a prime ideal of the ring $\bC[x_1, \ldots, x_n]$ corresponding to an irreducible subvariety $X=V(\fp)$ of $\bA^n$. It can often be a subtle matter to find generators for $\fp$. Typically, it is much easier to find equations defining $X$ set-theoretically, as this is a more geometric problem. In terms of algebra, this amounts to finding polynomials $f_1, \ldots, f_r$ such that $\fp$ is the radical of the ideal generated by $f_1, \ldots, f_r$.

Now suppose that $\fp$ is an $\fS$-prime of $R$. Once again, it can be quite difficult to find generators for $\fp$. The more tractable (and, in a sense, geometric) problem, similar to the previous paragraph, is to find polynomials $f_1, \ldots, f_r \in R$ such that $\fp$ is the $\fS$-radical of the $\fS$-ideal generated by $f_1, \ldots, f_r$.

We solve this problem for all $\fS$-primes of $R$. That is, given data $(\lambda, e; Z)$, we explicitly define polynomials $f_1, \ldots, f_r$ of $R$ such that $\fP(\lambda,e;Z)$ is the $\fS$-radical of the $\fS$-ideal generated by the $f_i$'s. The construction is somewhat complicated, so we do not describe it here: see \S \ref{s:gen} for details. The key step in proving that the construction works is to show that if $\fq$ is an $\fS$-prime containing the $f_i$'s then $\fq$ contains $\fP(\lambda,e;Z)$. We accomplish this by appealing to Theorems~\ref{mainthm} and~\ref{mainthm2}.

\begin{example} \label{ex:gen1}
Take $\lambda=(\infty,\infty)$ and $Z=\bU^2$. Recall from Example~\ref{ex:rad}(c) that $\fP(\lambda)$ is generated by the $\fS$-orbit of the discriminant $\Delta_3=(\xi_1-\xi_2)(\xi_1-\xi_3)(\xi_2-\xi_3)$. Let $e=(2,2)$. We show that $\fP(\lambda,e)$ is the $\fS$-radical of the ideal generated by the orbits of
\begin{displaymath}
\Delta_5, \qquad \Delta_3 \cdot \prod_{i=1}^3 (\xi_1-\xi_4)^3, \qquad \Delta_3^3.
\end{displaymath}
Here $\Delta_5$ is the discriminant on $\xi_1, \ldots, \xi_5$. See Example~\ref{ex:2-2} for details.
\end{example}

\begin{example}
Let $\lambda=(\infty,\infty)$ and let $Z$ be the the intersection of the unit circle in $\bA^2$ with $\bU^2$. We described the radical $\fS$-prime $\fP(\lambda;Z)$ in Example~\ref{ex:rad}(e). Now let $e=(2,2)$. We show that $\fP(\lambda,e;Z)$ is the $\fS$-radical of the $\fS$-ideal generated by the three generators from Example~\ref{ex:gen1} together with the two additional elements
\begin{displaymath}
(\xi_1-\xi_2)^3 \cdot (\xi_1^2+\xi_2^2-1)^4, \qquad
\Delta_3 \cdot \prod_{1 \le i<j \le 3} (\xi_i^2+\xi_j^2-1)^{4}.
\end{displaymath}
This example is discussed in detail in \S \ref{ss:circle}.
\end{example}

\subsection{The equivariant spectrum} \label{ss:intro-spec}

The \defi{equivariant spectrum} of $R$, denoted $\Spec_{\fS}(R)$, is the set of $\fS$-prime ideals of $R$ equipped with a version of the Zariski topology (see \S \ref{ss:Gspec}). Our two main theorems describe this space: the classification theorem (Theorem~\ref{mainthm}) describes its point-set, while the containment theorem (Theorem~\ref{mainthm2}) gives (or leads to) a description of the topology. We now explain how this space, and our theorems, can be pictured more geometrically.

Consider an $\fS$-prime $\fp=\fP(\lambda,e;Z)$ where $\lambda=(\lambda_1,\ldots,\lambda_r)$ and $Z$ is a closed subvariety of $\bU^r$. First suppose that $Z=\{x\}$ is a single point. We picture $x$ as a collection of $r$ distinct points $x_1, \ldots, x_r$ in $\bA^1$, i.e., a configuration. We regard $x_i$ as having multiplicity $\lambda_i$ and weight $e_i$. In this way, we represent $\fp$ geometrically as a configuration in $\bA^1$ equipped with multiplicity and weight data. We note that a point of $\bU^r$ corresponds to a labeled configuration, but since $\fp$ only depends on $(\lambda,e;Z)$ up to isomorphism, $\fp$ corresponds to an unlabeled configuration. When $Z$ is positive dimensional, we picture $\fp$ as the generic point of a family of such configurations. We thus see that $\Spec_{\fS}(R)$ can be pictured as the space of unordered configurations in $\bA^1$ (in a scheme-theoretic sense) with multiplicity and weight data, together with an additional generic point (corresponding to the zero ideal of $R$). This is essentially a reformulation of Theorem~\ref{mainthm}.

We now describe the topology on $\Spec_{\fS}(R)$ from the configuration point of view. There are really two important properties:
\begin{itemize}
\item Let $x$ be a configuration with multiplicity and weight data. Let $y$ be obtained from $x$ by deleting points, reducing multiplicities, and/or reducing weights. Then $y$ is a specialization of $x$ in $\Spec_{\fS}(R)$, i.e., $y$ belongs to the closure of $\{x\}$.
\item Let $x(t)$ be a 1-parameter family of configurations, defined for $t \ne 0$, with multiplicity and weight data. For simplicity, we choose a labeling of the points $x(t)=(x_1(t), \ldots, x_r(t))$. Suppose that each $x_i(t)$ converges as $t \to 0$. Let $y_1, \ldots, y_s \in \bA^1$ be the distinct points obtained as such limits. Define the multiplicity of $y_i$ to be the sum of the multiplicities of the $x_i$'s that converge to $y_i$. If the multiplicitiy of $y_i$ is infinite, define its weight to be the sum of the weights of the $x_i$'s having infinite multiplicity that converge to $y_i$; otherwise, give $y_i$ weight~1. Then $x(t)$ converges to $y$ in $\Spec_{\fS}(R)$ as $t \to 0$.
\end{itemize}
The above two properties are essentially equivalent to Theorem~\ref{mainthm2}; see Remark~\ref{rmk:theta} for more details. These properties do not completely describe the topology on $\Spec_{\fS}(R)$, but do capture the most interesting aspects of it.

The above discussion has been somewhat imprecise by design. A precise and rigorous discussion along these lines is carried out in \S \ref{s:spec}.

\begin{remark}
Here is a direct description of how to go from the configuration point of view to $\fS$-primes. Let $x=(x_1, \ldots, x_r)$ be a configuration with multiplicity data $\lambda$ and weight data $e$ corresponding to the $\fS$-prime $\fp$. Define a $T$-point $\tilde{x}$ of $\Spec(R)=(\bA^1)^{\infty}$ by letting $\lambda_i$ coordinates be a jet of order $e_i$ in $\bA^1$ at $x_i$; here $T$ is a tensor product of algebras of the form $K[\epsilon]/\langle \epsilon^n \rangle$, where $K$ is a field of definition of $x$. Then $V(\fp)$ is the scheme-theoretic closure of $\fS \tilde{x}$.
\end{remark}

\begin{remark}
In the above discussion, we see that weights and multiplicities both add when points collide. For multiplicities, this is trivial; for weights, it is one of the deepest results of this paper. Here is a simple example that demonstrates why weights are more difficult to handle. Let $x=(a,b)$ be a point of $\bA^2$ with multiplicity data $(\infty,\infty)$ and weight data $(e,e)$, and define $\tilde{x}$ as in the previous remark. If $a \ne b$ then $x$ belongs to $\bU^2$, and the closure of the $\fS$-orbit of $\tilde{x}$ is scheme defined by the $\fS$-prime $\fP((\infty,\infty),(e,e); \{x\})$. If $a=b$ then $\tilde{x}$ coincides with $\tilde{y}$, where $y \in \bA^1$ is the point $a$ with multiplicity $\infty$ and weight $e$; the closure of the $\fS$-orbit of $\tilde{x}$ is the scheme defined by the $\fS$-prime $\fP((\infty),(e); \{a\})$. This shows that in a ``stationary collision'' weights do not add. Multiplicities still do add, however. Thus the behavior of weights is much more subtle.
\end{remark}

\subsection{Additional comments}

We are currently working on two projects related to this paper:
\begin{itemize}
\item In \cite{smod} we build on the results of this paper and study $\fS$-equivariant modules. We establish a number of results, including a version of primary decomposition and a description of the Grothendieck group.
\item We are also working on generalizing our results to the situation where $R$ is replaced by $A^{\otimes \infty}$, where $A$ is a finitely generated $\bC$-algebra; the ring $R$ considered in this paper is obtained by taking $A=\bC[\xi]$.
\end{itemize}
We mention some other work related to this paper:
\begin{itemize}
\item Equivariant commutative algebra with respect to the general linear group, and other algebraic (super)groups, is studied in the papers \cite{ganapathy, tcaspec, eqprimes}.
\item The locus $\fX_{\lambda}$ is a subspace arrangment, i.e., a union of linear subspaces. Ideals of subspace arrangements are studied in \cite{bps, loera, lovasz, schenck}. Some of the results in these works are relevant to the problem of determing generators for the radical $\fS$-primes $\fP(\lambda)$. It would be interesting if any of this work could be adapted to the study of the $\fS$-ideals $\fP(\lambda,e)$, which are closely related.
\item The results of \cite{raicu} give the resolutions of monomial $\fS$-ideals of $R$. It would be interesting to extend these results to other $\fS$-ideals, especially the $\fS$-primes.
\end{itemize}
We also mention an interesting problem stemming from this paper:
\begin{itemize}
\item Give an algorithm that takes as input $f_1, \ldots, f_r \in R$ and returns a description of the $\fS$-radical of the $\fS$-ideal they generate as an intersection of $\fS$-primes, specified by our parametrization.
\end{itemize}

\subsection{Outline}

In \S \ref{s:eqca}, we develop equivariant commutative algebra in the abstract. In \S \ref{s:ring}, we introduce the infinite variable polynomial ring and the $\fS$-prime ideals $\fP(\lambda,e;Z)$, and prove a few elementary results about them. In \S \ref{s:contract}, we prove the contraction result needed in the proof of the classification theorem. The heart of the paper is \S \ref{s:class} and \S \ref{s:contain}, where we prove the classification theorem (Theorem~\ref{mainthm}) and the containment theorem (Theorem~\ref{mainthm2}), respectively. In \S \ref{s:gen}, we determine generators for $\fS$-primes (up to $\fS$-radical), and finally, in \S \ref{s:spec}, we describe the equivariant spectrum.

\subsection{Notation and conventions}

All rings in this paper are commutative with unit. The following is the most important notation:
\begin{description}[align=right,labelwidth=2.5cm,leftmargin=!]
\item[ ${[\infty]}$ ] the set of positive integers
\item[ $\fS$ ] the infinite symmetric group
\item[ $\fY_{\cU}$ ] the Young subgroup of $\fS$ corresponding to a partition $\cU$ of $[\infty]$
\item[ $A$ ] the coefficient ring
\item[ $R$ ] the polynomial ring $A[\xi_i]_{i \ge 1}$ (except in \S \ref{s:eqca})
\item[ $\bA^{\cI}$ ] affine space with coordinates indexed by the finite set $\cI$
\item[ $\bU^{\cI}$ ] the open subspace of $\bA^{\cI}$ where all coordinates are distinct
\item[ $\fP(\lambda,e;Z)$ ] the $\fS$-ideal of $R$ defined in \S \ref{ss:P}
\end{description}
Some additional important notation is introduced in \S \ref{ss:setup}.

\section{Equivariant commutative algebra} \label{s:eqca}

In this section we develop what we need from equivariant commutative algebra. We work in the setting of smooth representations of what we call ``admissible'' topological groups. This is far from the most general setup: one can work with larger classes of representations, or, even more generally, in the setting of a tensor category; see \cite[\S 2]{eqprimes} for more on this perspective. Some of the results we prove here are also contained in \cite[\S 2]{eqprimes} or \cite[\S 2]{tcaspec}, but many are new.

\subsection{Smooth representations}

We begin by introducing the kind of groups and actions we will work with.

\begin{definition}
An \defi{admissible group} is a Hausdorff topological group $G$ satisfying the following two conditions:
\begin{enumerate}
\item The open subgroups of $G$ form a neighborhood basis of the identity element.
\item $G$ is Roelcke precompact: if $U$ and $V$ are open subgroups of $G$ then the double coset space $U \backslash G /V$ is finite. \qedhere
\end{enumerate}
\end{definition}

\begin{definition}
An action of an admissible group $G$ on a set $X$ is \defi{smooth} if the stabilizer of each element of $X$ is an open subgroup of $G$.
\end{definition}

\begin{example}
We give some examples of admissible groups.
\begin{enumerate}
\item Any finite group (with the discrete topology).
\item The infinite symmetric group $\fS=\bigcup_{n \ge 1} \fS_n$, equipped with the coarsest topology for which the action on $[\infty]$ is smooth. Precisley, let $\fS_{>n}$ be the subgroup fixing each of $1, \ldots, n$. Then these subgroups form a neighborhood basis of the identity. This is the main example of interest in this paper.
\item Any finite product of copies of $\fS$, or any subgroup containing such a product as a finite index subgroup (with a similar topology as above). In particular, the Young subgroups of $\fS$ and their normalizers are admissible. These will also be important examples in this paper.
\item The group $\GL_{\infty}(\bF)=\bigcup_{n \ge 1} \GL_n(\bF)$, where $\bF$ is a finite field, equipped with the coarsest topology for which the action on $\bF^{\infty}=\bigcup_{n \ge 1} \bF^n$ is smooth. \qedhere
\end{enumerate}
\end{example}

We fix an admissible group $G$ for the remainder of \S \ref{s:eqca}. We say that a $G$-set is finitely generated if it has finitely many orbits. We have the following important consequence of Roelcke precompactness:

\begin{proposition} \label{prop:prodfg}
If $X$ and $Y$ are finitely generated smooth $G$-sets then so is $X \times Y$.
\end{proposition}

\begin{proof}
It is clear that $X \times Y$ is smooth. To prove finite generation, it suffices to treat the case where $X$ and $Y$ are transitive; thus $X=G/U$ and $Y=G/V$ for open subgroups $U$ and $V$. We have $G \backslash (X \times Y) \cong U \backslash G / V$, which is finite by assumption. Thus the result follows.
\end{proof}

Let $\MOD_G$ be the category of all left $\bZ[G]$-modules and $\Mod_G$ the full subcategory spanned by the smooth modules (i.e., $\bZ[G]$-modules for which the action of $G$ is smooth). One easily sees the following:
\begin{itemize}
\item Any subquotient of a smooth $\bZ[G]$-module is smooth. In particular, $\Mod_G$ is an abelian subcategory of $\MOD_G$.
\item Any direct sum of smooth $\bZ[G]$-modules is smooth. In particular, $\Mod_G$ is closed in $\MOD_G$ under arbitrary colimits, and filtered colimits are exact.
\item The direct sum of all smooth cyclic $\bZ[G]$-modules (taken up to isomorphism) is a generator for $\Mod_G$. Thus $\Mod_G$ is a Grothendieck abelian category.
\item If $V$ and $W$ are smooth $\bZ[G]$-modules then so is $V \otimes W$. Thus $\Mod_G$ carries a natural tensor structure.
\end{itemize}
We also have the following property that will be important to our discussion:

\begin{proposition} \label{prop:tenfg}
If $V$ and $W$ are finitely generated smooth $\bZ[G]$-modules then so is $V \otimes W$.
\end{proposition}

\begin{proof}
Since $V$ is finitely generated as an $\bZ[G]$-module, it is generated as an $\bZ$-module by the orbits of finitely many elements; in other words, there is a finitely generated $G$-subset $X \subset V$ that generates $V$ as an $\bZ$-module. Similarly, there is a finitely generated $G$-subset $Y \subset W$ that generates $W$ as an $\bZ$-module. The set $Z=\{x \otimes y \mid x \in X, y \in Y\}$ thus generates $V \otimes W$ as an $\bZ$-module. Since $Z$ is a quotient of $X \times Y$, it is finitely generated as a $G$-set (Proposition~\ref{prop:prodfg}), and so the result follows.
\end{proof}

\begin{example}
Let $V_r$ be the free $\bZ$-module with basis $e_{i_1,\ldots,i_r}$ where $i_1,\ldots,i_r$ are distinct elements of $[\infty]$. Then $V_r$ is a smooth $\bZ[\fS]$-module, and every smooth $\bZ[\fS]$-module is a quotient of a direct sum of $V_r$'s. One easily sees that $V_r \otimes V_s$ decomposes as a finite direct sum of $V_k$'s with $k \le r+s$. This gives another proof of Proposition~\ref{prop:tenfg} in the case $G=\fS$.
\end{example}

\begin{remark}
We refer the reader to \cite{rovinsky,rovinsky2,rovinsky3}, and especially \cite[\S 2]{rovinsky}, for further discussion of admissible groups and their representations.
\end{remark}

\subsection{Equivariant rings, ideals, and modules}

A \defi{$G$-ring} is a commutative algebra object in the tensor category $\Mod_G$; in other words, it is a commutative ring $R$ equipped with a smooth action of $G$ by ring automorphisms. Fix a $G$-ring $R$. A \defi{$G$-ideal} of $R$ is a $G$-stable ideal of $R$. An \defi{$(R,G)$-module} is a module object for $R$ in $\Mod_G$; that is, it is a $G$-equivariant $R$-module $M$ such that the action of $G$ on $M$ is smooth.

When working in the equivariant setting, $G$-submodules will often take the role of elements. We introduce a notation to make this analogy more prominent: for a $G$-module $M$, we let $[M]$ denote the set of all $G$-submodules of $M$, and $[M]^{\rf}$ the set of all finitely generated $G$-submodules of $M$. Suppose $M$ is an $(R,G)$-module. Given $X \in [R]$ and $Y \in [M]$, we let $XY \in M$ be the set of all sums $\sum_{i=1}^n x_i y_i$ with $x_i \in X$ and $y_i \in Y$. We note that if $X \in [R]^{\rf}$ and $Y \in [M]^{\rf}$ then $XY \in [M]^{\rf}$ by Proposition~\ref{prop:tenfg}. We let $\bone_R$ be the $\bZ$-submodule of $R$ generated by~1. This satisfies $\bone_R Y=Y$ for any $Y \in [M]$.

\subsection{Primes}

Fix a $G$-ring $R$. We now introduce the most important definition for this paper:

\begin{definition}
A $G$-ideal $\fp$ of $R$ is \defi{$G$-prime} if  $\fp \ne R$ and $XY \in [\fp]$ implies $X \in [\fp]$ or $Y \in [\fp]$ for $X,Y \in [R]$. We say that $R$ is \defi{$G$-integral} or a \defi{$G$-domain} if the zero ideal is $G$-prime.
\end{definition}

If $\fp$ is a $G$-stable prime ideal then $\fp$ is $G$-prime: indeed, if $XY \in [\fp]$ and $X \not\in [\fp]$ then there exists $x \in X \setminus \fp$ and so for any $y \in Y$ we have $xy \in \fp$, whence $y \in \fp$; thus $Y \in [\fp]$. The converse to this is not true: in fact, a $G$-prime need not be radical; see Proposition~\ref{prop:powers-prime}. There are several other characterizations of $G$-primes (we leave the proof to the reader):

\begin{proposition}
Let $\fp$ be a $G$-ideal of $R$. The following are equivalent:
\begin{enumerate}
\item $\fp$ is $G$-prime.
\item $XY \in [\fp]$ implies $X \in [\fp]$ or $Y \in [\fp]$ for all $X,Y \in [R]^{\rf}$.
\item $\fa \fb \subset \fp$ implies $\fa \subset \fp$ or $\fb \subset \fp$ for $G$-ideals $\fa,\fb \subset R$.
\item Given $x,y \in R$ such that $x \cdot \sigma y \in \fp$ for all $\sigma \in G$ we have $x \in \fp$ or $y \in \fp$.
\end{enumerate}
\end{proposition}

We now establish some elementary results about $G$-primes; for the most part, these are analogs of standard facts about prime ideals.

\begin{proposition} \label{prop:min-prime}
Every $G$-prime of $R$ contains a minimal $G$-prime.
\end{proposition}

\begin{proof}
Let $\fp$ be a $G$-prime of $R$ and let $\Sigma$ be the set of all $G$-primes of $R$ contained in $\fp$. Of course, $\Sigma$ is non-empty since $\fp \in \Sigma$. If $\{\fq_i\}_{i \in I}$ is a descending chain in $\Sigma$ then $\bigcap_{i \in I} \fq_i$ is easily seen to be a $G$-prime, and clearly belongs to $\Sigma$. Thus the result follows by Zorn's lemma.
\end{proof}

A \defi{multiplicative $G$-system} in $R$ is a subset $S$ of $[R]^{\rf}$ that contains $\bone_R$ and is closed under products (i.e., $X,Y \in S$ implies $XY \in S$). We say that a $G$-submodule $Y \subset R$ is \defi{disjoint} from $S$ if $[Y]$ and $S$ are disjoint subsets of $[R]$, i.e., there is no subrepresentation of $Y$ that belongs to $S$.

\begin{proposition} \label{prop:max-prime}
Let $S$ be a multiplicative $G$-system in $R$ such that $0 \not\in S$. Let $\Sigma$ be the set of all $G$-ideals of $R$ that are disjoint from $S$. Then
\begin{enumerate}
\item $\Sigma$ is non-empty.
\item Every element of $\Sigma$ is contained in a maximal element.
\item Every maximal element of $\Sigma$ is $G$-prime.
\end{enumerate}
In particular, there exists a $G$-prime of $R$ that is disjoint from $S$.
\end{proposition}

\begin{proof}
Since $0 \not\in S$ we have $0 \in \Sigma$, and so $\Sigma$ is non-empty. Suppose that $\{\fa_i\}_{i \in I}$ is an ascending chain in $\Sigma$, and let $\fa=\bigcup_{i \in I} \fa_i$. Then $\fa$ is a $G$-ideal of $R$, and belongs to $\Sigma$: indeed, if we had $X \subset \fa$ for some $X \in S$ then we would have $X \subset \fa_i$ for some $i$ since $X$ is finitely generated, a contradiction. It follows from Zorn's lemma that every element of $\Sigma$ is contained in a maximal element.

Now let $\fp$ be a maximal element of $\Sigma$. Suppose that $\fp$ is not prime. Let $\fa,\fb$ be $G$-ideals such that $\fa \fb \subset \fp$ but $\fa,\fb \not\subset \fp$. Then $\fp+\fa$ and $\fp+\fb$ are strictly larger than $\fp$, and thus do not belong to $\Sigma$. Thus we have $X \subset \fp+\fa$ and $Y \subset \fp+\fb$ for some $X,Y, \in \Sigma$. It follows that $XY \subset (\fp+\fa)(\fp+\fb) \subset \fp$, which is a contradiction since $XY \in S$. Thus $\fp$ is prime.
\end{proof}

A $G$-ideal is \defi{maximal} if it is a proper ideal and maximal among proper $G$-ideals.

\begin{proposition} \label{prop:max-ideal}
Let $R$ be a $G$-ring. Then every proper $G$-ideal is contained in a maximal $G$-ideal, and every maximal $G$-ideal is $G$-prime.
\end{proposition}

\begin{proof}
If $R=0$ the proposition is vacuous. Otherwise it follows from Proposition~\ref{prop:max-prime} with $S=\{\bone_R\}$.
\end{proof}

\begin{proposition} \label{prop:primes-exist}
Let $R$ be a non-zero $G$-ring. Then $R$ contains a maximal $G$-ideal and a minimal $G$-prime.
\end{proposition}

\begin{proof}
Since the zero ideal is proper, it is contained in a maximal $G$-ideal $\fp$ by Proposition~\ref{prop:max-ideal}, which is $G$-prime. Thus $\fp$ contains a minimal $G$-prime by Proposition~\ref{prop:min-prime}.
\end{proof}

\begin{proposition}
Let $\fp$ be a $G$-prime of $R$. Let $S(\fp)$ be the set of all $X \in [R]^{\rf}$ such that $X \not\subset \fp$. Then $S(\fp)$ is a multiplicative $G$-system. Moreover, if $Y$ is a $G$-submodule of $R$ then $Y$ is disjoint from $S(\fp)$ if and only if $Y \subset \fp$.
\end{proposition}

\begin{proof}
If $X,Y \in S(\fp)$ then $X,Y \not\in [\fp]$ and so $XY \not\in \fp$; thus $XY \in S(\fp)$. This shows that $S(\fp)$ is a multiplicative set. Now let $Y$ be a $G$-submodule of $R$. The following are equivalent:
\begin{itemize}
\item $Y$ is disjoint from $S(\fp)$.
\item There is no $X \in [Y]^{\rf}$ that belongs to $S(\fp)$; in other words, for every finitely generated $G$-submodule $X$ of $Y$ we have $X \subset \fp$.
\item We have $Y \subset \fp$.
\end{itemize}
The first two statements are equivalent by definition, while the second and third are clearly equivalent. The result follows.
\end{proof}

We say that $X \in [R]$ is a \defi{$G$-zerodivisor} if there exists $Y \in [R]$ non-zero such that $XY=0$; otherwise we say that $X$ is a \defi{$G$-non-zerodivisor}. We say that $x \in R$ is a $G$-(non)-zerodivisor if the $G$-submodule it generates is. Note that 0 is a $G$-zerodivisor (assuming $R \ne 0$).

\begin{proposition} \label{prop:zd}
Let $\fp$ be a minimal $G$-prime. Then every element of $\fp$ is a $G$-zerodivisor.
\end{proposition}

\begin{proof}
Let $S_1=S(\fp)$, let $S_2$ be the set of all $X \in [R]^{\rf}$ such that $X$ is a $G$-non-zerodivisor, and let $S=S_1S_2$, i.e., $S$ consists of all elements of $[R]^{\rf}$ of the form $XY$ with $X \in S_1$ and $Y \in S_2$. We have already seen that $S_1$ is a multiplicative $G$-system, and it is easily seen that $S_2$ is one as well; thus $S$ is too. Since every element of $S_2$ is a $G$-non-zerodivisor and every element of $S_1$ is non-zero, it follows that $0 \not\in S$. By Proposition~\ref{prop:max-prime}, we see that there is a $G$-prime $\fq$ of $R$ that is disjoint from $S$. In particular, $\fq$ is disjoint from both $S_1$ and $S_2$. Since $\fq$ is disjoint from $S_2$, every element of $\fq$ is a $G$-zerodivisor; indeed, if $\fq$ contained a $G$-non-zerodivisor $x$ then the $G$-submodule generated by $x$ would be contained in $S_2$, contradicting the disjointness of $\fq$ from $S_2$. Since $\fq$ is disjoint from $S_1$, we have $\fq \subset \fp$; hence, by the minimality of $\fp$ we have $\fq=\fp$. The result follows.
\end{proof}

\begin{proposition} \label{prop:avoid}
Let $\fp_1, \ldots, \fp_n$ be $G$-primes and let $\fa$ be a $G$-ideal such that $\fa \subset \bigcup_{i=1}^n \fp_i$. Then $\fa \subset \fp_i$ for some $i$.
\end{proposition}

\begin{proof}
Assume $\fa$ is not contained in any of the $\fp_i$'s. If $\fa$ were contained in a union of $n-1$ of the $\fp_i$'s then, by induction, it would be contained in one of the $\fp_i$'s, so this is also not the case. For each $i$, pick $x_i \in \fa$ that is not contained in $\bigcup_{j \ne i} \fp_i$; necessarily, $x_i \in \fp_i$. Now consider $\sigma_1(x_1) \cdots \sigma_{n-1}(x_{n-1})+x_n$ for $\sigma_i\in G$. This is obviously an element of $\fa$. It does not belong to any of $\fp_1, \ldots, \fp_{n-1}$: indeed, the first term does, so if the sum did then $x_n$ would belong to one of the $\fp_1, \ldots, \fp_{n-1}$, which it does not. It also does not belong to $\fp_n$ for some choice of $\sigma_1, \ldots, \sigma_{n-1}$: indeed for all choices of $\sigma_1, \ldots, \sigma_{n-1}$, the second term does belong to $\fp_n$, so if the sum did then the first would; but then we would have $x_i \in \fp_n$ for some $i$ (by $G$-primality), a contradiction. We have thus produced an element of $\fa$ that is not contained in any $\fp_i$, which is a contradiction. Hence $\fa$ is contained in some $\fp_i$.
\end{proof}

\begin{proposition} \label{prop:rad-prime}
Let $\fp$ be a $G$-prime of $R$. Then $\rad(\fp)$ is also a $G$-prime of $R$.
\end{proposition}

\begin{proof}
Suppose $x \cdot \sigma y \in \rad(\fp)$ for all $\sigma \in G$. Let $n(\sigma)$ be the minimal positive integer such that $(x \cdot \sigma y)^{n(\sigma)} \in \fp$, let $U \subset G$ be the stabilizer of $x$, and let $V \subset G$ be the stabilizer of $y$. Then $n(\alpha \sigma \beta)=n(\sigma)$ for $\alpha \in U$ and $\beta \in V$. It follows that $n$ factors through the space $U \backslash G /V$ of double cosets. Since this set is finite (by Roelcke precompactness), it follows that the function $n$ is bounded. Thus there exists one integer $n$ such that $x^n \cdot \sigma y^n \in \fp$ for all $\sigma \in G$. Since $\fp$ is $G$-prime, it follows that $x^n \in \fp$ or $y^n \in \fp$, and so $x \in \rad(\fp)$ or $y \in \rad(\fp)$, which completes the proof.
\end{proof}

\begin{proposition} \label{prop:prime-contract}
Let $\phi \colon R \to S$ be a homomorphism of $G$-rings and let $\fp$ be a $G$-prime of $S$. Then the contraction $\fp^c$ is a $G$-prime of $R$.
\end{proposition}

\begin{proof}
Suppose $X,Y \in [R]$ and $XY \in [\fp^c]$. Then $\phi(X)\phi(Y) = \phi(XY) \in [\fp]$. Since $\fp$ is $G$-prime, it follows that $\phi(X) \in [\fp]$ or $\phi(Y) \in \fp$, and so $X \in [\fp^c]$ or $Y \in [\fp^c]$. This completes the proof.
\end{proof}

\begin{proposition} \label{prop:prime-loc}
Let $A$ be a subring of the invariant ring $R^G$ and let $S$ be a multiplicative subset of $A$. Then extension and contraction induce mutually inverse bijections
\begin{displaymath}
\{ \text{$G$-primes $\fp$ of $R$ with $\fp \cap S=\emptyset$} \} \leftrightarrow \{ \text{$G$-primes of $S^{-1} R$} \}
\end{displaymath}
\end{proposition}

\begin{proof}
Let $\fp$ be a $G$-prime of $R$ with $\fp \cap S=\emptyset$, and let $\fq$ be its extension to $S^{-1} R$. Suppose $x \cdot \sigma y \in \fq$ for all $\sigma \in G$, where $x,y \in K \otimes_A R$. Clearing denominators, we can assume $x,y \in R$ and that $sx(\sigma y) \in \fp$ for all $\sigma \in G$, where $s \in S$. Since $\fp$ is $G$-prime, we find $sx \in \fp$ or $y \in \fp$. In the former case, we have $x (\sigma s) \in \fp$ for all $\sigma$, since $s$ is invariant, and so $x \in \fp$ (since $s \not\in \fp$). We thus see that $\fq$ is $G$-prime.

It is clear that $\fp \subset \fq^c$. Suppose that $x \in \fq^c$. Then $sx \in \fp$ for some non-zero $s \in A$. As in the previous paragraph, we find that $x \in \fp$, and so $\fp=\fq^c$.

In the other direction, suppose we start with a $G$-prime $\fq$ of $S^{-1} R$. Then $\fq^c$ is a $G$-prime of $R$ by Proposition~\ref{prop:prime-contract}. It is clear that $\fq=(\fq^c)^e$ and $\fq^c \cap S=\emptyset$. Thus the result follows.
\end{proof}

Suppose $A \subset R^G$; we think of $A$ as the ``coefficient ring'' of $R$. Combining the previous two propositions, we can reduce the task of classifying $G$-primes of $R$ to the case where the coefficient ring is a field. Indeed, if $\fp$ is a $G$-prime of $R$ then $\fc=A \cap \fp$ is a prime of $A$ by Proposition~\ref{prop:prime-contract}. Thus $\fp$ corresponds to a $G$-prime of $R/\fc R$, which, in turn, corresponds to a prime of $\Frac(A/\fc) \otimes_A R$ by Proposition~\ref{prop:prime-loc}.

\subsection{Radicals}

We now examine the equivariant analog of radicals.

\begin{definition}
We make the following definitions:
\begin{enumerate}
\item The \defi{$G$-radical} of a $G$-ideal $\fa$, denoted $\rad_G(\fa)$, is the sum of all $G$-submodules $X \subset R$ for which $X^n \subset \fa$ for some $n$.
\item The \defi{$G$-(nil)radical} of $R$, denoted $\rad_G(R)$, is the $G$-radical of the zero ideal.
\item The $G$-ideal $\fa$ is said to be \defi{$G$-radical} if $\fa=\rad_G(\fa)$. \qedhere
\end{enumerate}
\end{definition}

One easily sees that $\rad_G(\fa)$ is a $G$-ideal of $R$. It is clear that any $G$-prime is $G$-radical. Since $G$-primes need not be radical, we thus see that radical and $G$-radical need not coincide.

\begin{proposition} \label{prop:fin-nilp}
If $X \in [\rad_G(\fa)]^{\rf}$ then $X^n \in [\fa]^{\rf}$ for some $n$.
\end{proposition}

\begin{proof}
Let $S$ be the set of all $Y \in [R]$ such that $Y^n \subset \fa$ for some $n$. Thus $\rad_G(\fa)=\sum_{Y \in S} Y$. Since $X$ is finitely generated and contained in this sum, it follows that $X \subset Y_1+\cdots+Y_k$ for some $Y_1,\ldots,Y_k \in S$. If $n$ is such that $Y_i^n \subset \fa$ for each $i$ then we have $X^{nk} \subset \fa$, as required.
\end{proof}

The following proposition gives an elemental characterization of $G$-radical.

\begin{proposition}
Let $\fa$ be a $G$-ideal of $R$. Then $x \in \rad_G(\fa)$ if and only if there exists $n \ge 1$ such that $(\sigma_1 x) \cdots (\sigma_n x) \in \fa$ for all $\sigma_1, \ldots, \sigma_n \in G$.
\end{proposition}

\begin{proof}
Let $X$ be the $G$-submodule generated by $x$. Then $x \in \rad_G(\fa)$ if and only if $X \subset \rad_G(\fa)$ if and only if $X^n \subset \fa$ for some $n$ (we have appealed to Proposition~\ref{prop:fin-nilp} for the second equivalence). Finally, simply note that $X^n$ is generated, as a $\bZ$-module, by the products $(\sigma_1 x) \cdots (\sigma_n x)$ for $\sigma_1, \ldots, \sigma_n \in G$.
\end{proof}

\begin{proposition} \label{prop:rad-primes}
The $G$-radical of a $G$-ideal $\fa$ is the intersection of all $G$-primes containing $\fa$.
\end{proposition}

\begin{proof}
Passing to $R/\fa$, it suffices to treat the case where $\fa=0$. Thus, letting $\fn$ be the intersection of all $G$-primes, we aim to show $\rad(R)=\fn$. If $X \in [\rad(R)]^{\rf}$ then $X^n=0$ for some $n$, and so $X$ is contained in every $G$-prime, and thus $X \subset \fn$. It follows that $\rad(R) \subset \fn$.

Now let $X \in [R]^{\rf}$ be non-nilpotent. Let $S$ be the multiplicative $G$-system $\{X^n \mid n \ge 0\}$, where $X^0=\bone_R$. Since $X$ is not nilpotent, $0 \not\in S$. Hence by Proposition~\ref{prop:max-prime}, there exists some $G$-prime $\fp$ that is disjoint from $S$; in particular, $X \not\subset \fp$, and so $X \not\subset \fn$. We thus see that if $X \in [\fn]^{\rf}$ then $X$ is nilpotent, i.e., $X \subset \rad(R)$. It follows that $\fn \subset \rad(R)$, which completes the proof.
\end{proof}

\begin{proposition} \label{prop:rad-contract}
Formation of $G$-radical commutes with contraction. That is, if $\phi \colon R \to S$ is a homomorphism of $G$-rings and $\fa$ is a $G$-ideal of $S$ then $\rad_G(\fa^c)=\rad_G(\fa)^c$. In particular, the contraction of a $G$-radical ideal is $G$-radical.
\end{proposition}

\begin{proof}
Suppose $X \in [R]^{\rf}$. Then $X \in [\rad_G(\fa^c)]$ if and only if $X^n \in [\fa^c]$ for some $n$ (Proposition~\ref{prop:fin-nilp}), which is equivalent to $\phi(X)^n \in [\fa]$. Similarly, $X \in [\rad_G(\fa)^c]$ if and only if $\phi(X) \in [\rad_G(\fa)]$, which is equivalent to $\phi(X)^n \in [\fa]$ for some $n$ (Proposition~\ref{prop:fin-nilp}). The result follows.
\end{proof}

\subsection{The spectrum} \label{ss:Gspec}

Let $R$ be a $G$-ring. We define the \defi{$G$-spectrum} of $R$, denoted $\Spec_G(R)$, to be the set of all $G$-primes of $R$. Given a $G$-ideal $\fa$, we let $V_G(\fa) \subset \Spec_G(R)$ be the set of all $G$-primes containing $\fa$. We define the \defi{Zariski topology} on $\Spec_G(R)$ to be the topology in which the closed sets are the $V_G(\fa)$ (one verifies that this defines a topology just as in the ordinary case). By Proposition~\ref{prop:rad-primes}, we see that $V_G(\fa) \subset V_G(\fb)$ if and only if $\rad_G(\fb) \subset \rad_G(\fa)$. In particular, the closed subsets of $\Spec_G(R)$ are in bijective correspondence with the $G$-radical ideals of $R$, and the irreducible components of $\Spec_G(R)$ are in bijective correspondence with the minimal $G$-primes.

\subsection{Nilpotent extensions} \label{ss:nilp}

Let $A \subset B$ be commutative rings. We say that $B$ is a \defi{nilpotent extension} of $A$ if there is a nilpotent ideal $\fn \subset B$ such that $B=A+\fn$; by ``nilpotent'' here, we mean $\fn^k=0$ for some $k$. We now study how $G$-primes relate in nilpotent extenions; this will be used in the proof of Theorem~\ref{thm:class2} in \S \ref{ss:class2}. We begin with a result that does not involve group actions:

\begin{proposition} \label{prop:nilp}
Let $A \subset B$ be a nilpotent extension, and write $B=A+\fn$ with $\fn^k=0$.
\begin{enumerate}
\item If $\fa$ is an ideal of $A$ then $(\fa^{ec})^k \subset \fa \subset \fa^{ec}$.
\item If $\fb$ is an ideal of $B$ then $\fb^{2k-1} \subset \fb^{ce} \subset \fb$.
\end{enumerate}
\end{proposition}

\begin{proof}
(a) Let $f_1, \ldots, f_k \in \fa^{ec}$. We have $\fa^e=\fa B = \fa+\fa\fn \subset \fa+\fn$, and so we can write $f_i=a_i+b_i$ with $a_i \in \fa$ and $b_i \in \fn$. Since $b_i=f_i-a_i$, we have $b_i \in A$. Now consider the product $f_1 \cdots f_k = \prod_{i=1}^k (a_i+b_i)$. Expanding the product, the final term is $b_1 \cdots b_k$, which vanishes since $\fn^k=0$. The other terms have at least one $a_i$, and thus belong to $\fa$, since each $a_i$ belongs to $\fa$ and each $b_i$ belongs to $A$. Thus $f_1 \cdots f_k \in \fa$, and so $(\fa^{ec})^k \subset \fa$. The inclusion $\fa \subset \fa^{ec}$ is obvious.

(b) Let $f_1, \ldots, f_{2k-1} \in \fb$, and write $f_i=a_i+b_i$ with $a_i \in A$ and $b_i \in \fn$. Then $a_1 \cdots a_k = \prod_{i=1}^k (f_i-b_i)$. Expanding the product, the final term is $b_1 \cdots b_k$, which vanishes since $\fn^k=0$, while the other terms all have at least one $f$, and thus belong to $\fb$. We thus see that $a_1 \cdots a_k \in \fb^c$; of course, the same holds for any product of $k$ of the $a_i$'s. Now, we have $f_1 \cdots f_{2k-1} = \prod_{i=1}^{2k-1} (a_i+b_i)$. Expanding the product, each term either has at least $k$ $a$'s or at least $k$ $b$'s. In the first case, the term belongs to $\fb^{ce}$, while in the second, it vanishes. We thus see that $f_1 \cdots f_{2k-1} \in \fb^{ce}$, and so $\fb^{2k-1} \subset \fb^{ce}$. The inclusion $\fb^{ce} \subset \fb$ is obvious.
\end{proof}

The following is the main result we require on nilpotent extensions:

\begin{proposition} \label{prop:nilp2}
Let $A \subset B$ be a nilpotent extension of $G$-rings. (The nilpotent ideal is not required to be $G$-stable.) Then we have mutually inverse bijections
\begin{displaymath}
\xymatrix{
\{ \text{$G$-primes of $A$} \} \ar@<2pt>[r]^{\Phi} & \{ \text{$G$-primes of $B$} \} \ar@<2pt>[l]^{\Psi} }
\end{displaymath}
given by $\Phi(\fp)=\rad_G(\fp^e)$ and $\Psi(\fq)=\fq^c$.
\end{proposition}

\begin{proof}
Write $B=A+\fn$ where $\fn$ is an ideal of $B$ with $\fn^k=0$. Let $\fp$ be a $G$-prime of $A$. We first show that $\rad_G(\fp^e)$ is $G$-prime. Thus suppose $\fb \fc \subset \rad_G(\fp^e)$, where $\fb$ and $\fc$ are finitely $G$-generated $G$-ideals of $B$. Then $\fb^n \fc^n \subset \fp^e$ for some $n$. Thus $(\fb^c)^n(\fc^c)^n \subset \fp^{ec}$. By Proposition~\ref{prop:nilp}(a), we have $(\fp^{ec})^k \subset \fp \subset \fp^{ec}$, and so $\fp^{ec}=\fp$ since $\fp$ is $G$-prime. We thus see that $(\fb^c)^n (\fc^c)^n \subset \fp$ and so $\fb^c \subset \fp$ or $\fc^c \subset \fp$; without loss of generality, suppose $\fb^c \subset \fp$. Then $\fb^{ce} \subset \fp^e$, and so $\fb^{2k-1} \subset \fp^e$ by Proposition~\ref{prop:nilp}(b). Thus $\fb \subset \rad_G(\fp^e)$, which proves that $\rad_G(\fp^e)$ is $G$-prime. Thus $\Phi$ is well-defined. It is clear that $\Psi$ is well-defined.

Let $\fp$ be a $G$-prime of $A$. Then
\begin{displaymath}
\fp=\rad_G(\fp)=\rad_G(\fp^{ec})=\rad_G(\fp^e)^c=\Psi(\Phi(\fp)).
\end{displaymath}
In the first step, we used that every $G$-prime is equal to its own $G$-radical; in the second, that $\fp=\fp^{ec}$, which we proved in the previous paragraph; and in the third, that formation of $G$-radicals commutes with contraction (Proposition~\ref{prop:rad-contract}). We thus see that $\Psi \circ \Phi$ is the identity.

Now let $\fq$ be a $G$-prime of $B$. By Proposition~\ref{prop:nilp}(b), we have $\fq^{2k-1} \subset \fq^{ce} \subset \fq$, and so $\fq=\rad_G(\fq^{ce})=\Phi(\Psi(\fq))$ since $\fq$ is its own $G$-radical. Thus $\Phi \circ \Psi$ is the identity as well, which completes the proof.
\end{proof}

\subsection{Subgroups}

When studying $G$-equivariant commutative algebra, it is at times helpful to consider subgroups of $G$. We say that a subgroup $H$ of $G$ is \defi{admissible} if it is so when given the subspace topology. We fix an admissible subgroup $H$ of $G$. In the remainder of this section, we investigate how $G$- and $H$-equivariant concepts compare.

The following proposition isolates an important property will sometimes be needed.

\begin{proposition}
The following are equivalent:
\begin{enumerate}
\item For every open subgroup $U$ of $G$, the set $H \backslash G / U$ is finite.
\item Every finitely generated smooth $G$-set is also finitely generated as an $H$-set.
\item Every finitely generated smooth $\bZ[G]$-module is finitely generated as a $\bZ[H]$-module.
\end{enumerate}
\end{proposition}

\begin{proof}
Let $U$ be an open subgroup. Then the $G$-set $G/U$ is finitely generated as an $H$-set if and only if $H \backslash G/U$ is finite. Since every finitely generated $G$-set is a finite union of sets of the form $G/U$, we see that (a) and (b) are equivalent. Since every finitely generated smooth $G$-module is a quotient of a finite sum of modules of the form $\bZ[G/U]$, we see that (b) and (c) are equivalent.
\end{proof}

\begin{definition} \label{defn:big}
The subgroup $H$ of $G$ is \defi{big} if the conditions of the proposition hold.
\end{definition}

\begin{example}
We give some examples of big subgroups:
\begin{enumerate}
\item Any finite index subgroup is big.
\item Any open subgroup is big: this is exactly Roelcke precompactness.
\item We will see that the Young subgroups of $\fS$ are big (Proposition~\ref{prop:big-young}).
\item Similarly, parabolic subgroups of $\GL_{\infty}(\bF)$ are big, where $\bF$ is a finite field. \qedhere
\end{enumerate}
\end{example}

\subsection{The $I_G$ operation}

The following construction will play a primary role in relating $G$- and $H$-equivariant concepts.

\begin{definition}
Let $R$ be a $G$-ring and let $\fa$ be an ideal of $R$. We put $I_G(\fa)=\bigcap_{\sigma \in G} \sigma \fa$, which is a $G$-ideal of $R$.
\end{definition}

It is clear that the $I_G$ construction commutes with arbitrary intersections, i.e., we have $I_G(\bigcap_{i \in I} \fa_i)=\bigcap_{i \in I} I_G(\fa_i)$. We will mostly consider $I_G(\fa)$ when $\fa$ is an $H$-ideal. The following proposition gives the basic properties of the construction in this context:

\begin{proposition} \label{prop:int-prime}
Let $R$ be a $G$-ring and let $\fa$ be an $H$-ideal of $R$.
\begin{enumerate}
\item If $\fa$ is $H$-prime then $I_G(\fa)$ is $G$-prime.
\item If $\fa$ is $H$-radical then $I_G(\fa)$ is $G$-radical.
\item If $H$ is big then $\rad_G(I_G(\fa))=I_G(\rad_H(\fa))$.
\item We have $\rad(I_G(\fa))=I_G(\rad(\fa))$; in particular, if $\fa$ is radical then so is $I_G(\fa)$.
\end{enumerate}
\end{proposition}

\begin{proof}
(a) Suppose $x \cdot (\sigma y) \in I_G(\fa)$ for all $\sigma \in G$. Then $(\tau x)(\sigma y) \in \fa$ for all $\sigma,\tau \in G$, and so $(\tau x)(\pi \sigma y) \in \fa$ for all $\sigma,\tau \in G$ and $\pi \in H$. Since $\fa$ is $H$-prime, it follows that for all $\sigma,\tau \in G$ we have $\tau x \in \fa$ or $\sigma y \in \fa$. If $\tau x \in \fa$ for all $\tau \in G$ then $x \in I_G(\fa)$. Otherwise, there is some $\tau \in G$ such that $\tau x \not\in \fa$, and then $\sigma y \in \fa$ for all $\sigma \in G$, and so $y \in I_G(\fa)$. Thus $I_G(\fa)$ is $G$-prime.

(b) Write $\fa=\bigcap_{i \in I} \fq_i$ with $\fq_i$ an $H$-prime, per Proposition~\ref{prop:rad-primes}. Then $I_G(\fa)=\bigcap_{i \in I} I_G(\fq_i)$. Since $I_G(\fq_i)$ is $G$-prime by (a), it follows that $I_G(\fa)$ is $G$-radical.

(c) Since $\rad_H(\fa)$ is $H$-radical, it follows from (b) that $I_G(\rad_H(\fa))$ is $G$-radical. Since it contains $I_G(\fa)$, therefore it contains $\rad_G(I_G(\fa))$. We now prove the reverse inclusion. Thus suppose $X$ is a finitely generated $\bZ[G]$-submodule of $I_G(\rad_H(\fa))$. Since $X \subset \rad_H(\fa)$ and $X$ is finitely generated as a $\bZ[H]$-module (by bigness), there is some $n$ such that $X^n \subset \fa$. Since $X^n$ is $G$-stable, it follows that $X^n \subset I_G(\fa)$, and so $X \subset \rad_G(I_G(\fa))$.

(d) This follows from usual ring theory.
\end{proof}

\begin{proposition} \label{prop:IG-ext}
Let $A \to A'$ be a flat ring homomorphism, let $R$ be a $G$-ring equipped with a homorphism $A \to R^G$, let $R'=A' \otimes_A R$, let $H$ be a big subgroup of $G$, and let $\fa$ be an $H$-ideal of $R$. Then $I_G(\fa)^e=I_G(\fa^e)$, where $(-)^e$ denotes extension to $R'$.
\end{proposition}

\begin{proof}
It is clear that $I_G(\fa)^e \subset I_G(\fa^e)$; we prove the reverse inclusion. For a subset $S$ of $G$ and an ideal $\fb$ in a $G$-ring, put $I_S(\fb)=\bigcap_{\sigma \in S} \sigma \fb$. Let $U$ be an open subgroup of $G$, write $G=\bigsqcup_{i=1}^n H \sigma_i U$ (which is possible since $H$ is big), and let $S=\{\sigma_i^{-1}\}_{1 \le i \le n}$. Now, we have the following three general facts:
\begin{enumerate}
\item Extension along flat ring maps commutes with finite intersections.
\item If $\fb$ is an ideal of $R$ then $\fb^e=A' \otimes_A \fb$.
\item If $M$ is an $A[G]$-module then $(A' \otimes_A M)^U=A' \otimes_A M^U$. (This can be proved using Lazard's theorem: since $A'$ is flat over $A$, it is a filtered colimit of finite rank free $A$-modules.)
\end{enumerate}
Applying these, we find
\begin{displaymath}
(I_S(\fa^e))^U=(I_S(\fa)^e)^U=(A' \otimes_A I_S(\fa))^U=A' \otimes_A I_S(\fa)^U.
\end{displaymath}
Now, consider an element $f \in I_S(\fa)^U$. Then we have $\sigma_i f \in \fa$ for each $1 \le i \le n$. Since $f$ is $U$-invariant and $\fa$ is $H$-stable, it follows that $\tau \sigma_i \pi f \in \fa$ for all $\tau \in H$ and $\pi \in U$. We thus see that $\sigma f \in \fa$ for all $\sigma \in G$, and so $f \in I_G(\fa)$. Thus $I_S(\fa)^U \subset I_G(\fa)$, and so $(I_S(\fa^e))^U \subset I_G(\fa)^e$ for all $U$. It follows that $(I_G(\fa^e))^U \subset I_G(\fa)^e$ for any open subgroup $U$. Taking the union over all $U$ gives the desired containment as $R'$ is $G$-smooth.
\end{proof}

\subsection{$G$-contraction}

The following concept will be used frequently in this paper:

\begin{definition} \label{defn:G-contract}
Suppose $R$ is a $G$-ring, $S$ is an $H$-ring, and $\phi \colon R \to S$ is a homomorphism of $H$-rings. Given an $H$-ideal $\fa$ of $S$ we define its \defi{$G$-contraction} to $R$, denoted $\fa^{Gc}$, to be $I_G(\fa^c)$, where $\fa^c=\phi^{-1}(\fa)$ denotes the usual contraction.
\end{definition}

The following proposition summarizes some important properties of this construction:

\begin{proposition} \label{prop:G-contract}
Let $\phi \colon R \to S$ be as in the Definition~\ref{defn:G-contract} and let $\fa$ be an $H$-ideal of $S$.
\begin{enumerate}
\item If $\fa$ is $H$-prime then $\fa^{Gc}$ is $G$-prime.
\item If $\fa$ is $H$-radical then $\fa^{Gc}$ is $G$-radical.
\item If $H$ is big then $\rad_G(\fa^{Gc})=\rad_H(\fa)^{Gc}$.
\item If $\fa$ is radical then $\fa^{Gc}$ is radical; more generally, $\rad(\fa^{Gc})=\rad(\fa)^{Gc}$.
\end{enumerate}
\end{proposition}

\begin{proof}
(a) Since $\fa^c$ is $H$-prime (Proposition~\ref{prop:prime-contract}), $\fa^{Gc}$ is $G$-prime (Proposition~\ref{prop:int-prime}).

(b) Since $\fa^c$ is $H$-radical (Proposition~\ref{prop:rad-contract}), $\fa^{Gc}$ is $G$-radical (Proposition~\ref{prop:int-prime}).

(c) We have
\begin{displaymath}
\rad_G(\fa^{Gc})=\rad_G(I_G(\fa^c))=I_G(\rad_H(\fa^c)) = I_G(\rad_H(\fa)^c) = \rad_H(\fa)^{Gc}.
\end{displaymath}
The first step is the definition of $G$-contraction, the second follows from Proposition~\ref{prop:int-prime}(c), the third from Proposition~\ref{prop:rad-contract}, and the fourth is again the definition of $G$-contraction.

(d) This follows from usual ring theory.
\end{proof}

The next proposition is a kind of base change result relating extension and $G$-contraction. It will be used to prove an important base change property for the main ideals considered in this paper (see Propositions~ \ref{prop:P-ext} and~\ref{prop:flat-bc}). We first recall a lemma from ordinary commutative algebra.

\begin{lemma} \label{lemma:flat-ec}
Let $R \to R'$ be a flat ring homomorphism, let $R \to S$ be an arbitrary ring homomorphism, and let $S'=R' \otimes_R S$. Let $\fa$ be an ideal of $S$. Then we have $(\fa^c)^e=(\fa^e)^c$ as ideals of $R'$.
\end{lemma}

\begin{proof}
We have an exact sequence
\begin{displaymath}
0 \to \fa^c \to R \to S/\fa
\end{displaymath}
Applying $R' \otimes_R -$, which is exact since $R'$ is flat, we obtain an exact sequence
\begin{displaymath}
0 \to (\fa^c)^e \to R' \to S'/\fa^e
\end{displaymath}
which shows that $(\fa^c)^e=(\fa^e)^c$.
\end{proof}

\begin{proposition} \label{prop:Gc-bc}
Let $H$ be a big subgroup of $G$. Consider the following situation:
\begin{itemize}
\item $R$ is a $G$-ring equipped with a ring homomorphism $A \to R^G$.
\item $S$ is an $H$-ring
\item $R \to S$ is a homomorphism of $H$-rings
\item $A \to A'$ is a flat ring homomorphism
\item $R'=A' \otimes_A R$ and $S'=A' \otimes_A S$
\end{itemize}
We thus have the following co-cartesian diagram
\begin{displaymath}
\xymatrix{
R \ar[r] \ar[d] & S \ar[d] \\
R' \ar[r] & S' }
\end{displaymath}
Let $\fa$ be an $H$-ideal of $S$. Then $(\fa^e)^{Gc}=(\fa^{Gc})^e$.
\end{proposition}

\begin{proof}
We have
\begin{displaymath}
(\fa^e)^{Gc}=I_G((\fa^e)^c)=I_G((\fa^c)^e)=(I_G(\fa^c))^e=(\fa^{Gc})^e.
\end{displaymath}
In the first step we used the definition of $G$-contraction; in the second, Lemma~\ref{lemma:flat-ec}; in the third Proposition~\ref{prop:IG-ext}; and in the fourth the definiton of $G$-contraction.
\end{proof}

We next show that $G$-contraction commutes with direct limits in certain cases. This will be used in a limiting argument in the proof of Proposition~\ref{prop:corr-2}.

\begin{proposition} \label{prop:limit-contract}
Let $H$ be a big subgroup of $G$. Consider the following situation:
\begin{itemize}
\item $R$ is a $G$-ring that is a directed union of $G$-subrings $R=\bigcup_{i \in I} R_i$.
\item $S$ is an $H$-ring that is a directed union of $H$-subsring $S=\bigcup_{i \in I} S_i$.
\item $\phi \colon R \to S$ is a map of $H$-rings such that $\phi(R_i) \subset S_i$ for all $i \in I$.
\item $\{\fa_i\}_{i \in I}$ is a directed system of $H$-ideals in $\{S_i\}$, i.e., $\fa_i$ is an $H$-ideal of $S_i$ for each $i$, and for $i \le j$ we have $\fa_i \subset \fa_j$.
\end{itemize}
Then
\begin{displaymath}
\big( \bigcup_{i \in I} \fa_i \big)^{Gc} = \bigcup_{i \in I} (\fa_i^{Gc}).
\end{displaymath}
\end{proposition}

\begin{proof}
It is clear that the right side is contained in the left; we prove the reverse inclusion. Let $f \in \big( \bigcup_{i \in I} \fa_i \big)^{Gc}$ be given. Thus for each $\sigma \in G$, we have $\sigma f \in \big( \bigcup_{i \in I} \fa_i \big)^c=\bigcup_{i \in I} \fa_i^c$. Let $U \subset G$ be the stabilizer of $f$ and write $G=\bigsqcup_{j=1}^r H\sigma_iU$. Let $i \in I$ be such that $\sigma_j f \in \fa^c_i$ for each $1 \le j \le r$. Given $\sigma \in G$, write $\sigma=\rho \sigma_j \tau$ with $\rho \in H$ and $\tau \in U$. Since $\tau$ fixes $f$ and $\rho$ fixes $\fa^c_i$, we have $\sigma f \in \fa^c_i$. As this holds for all $\sigma$, we have $f \in \fa_i^{Gc}$, as required.
\end{proof}

\subsection{Finite index subgroups}

We now examine how $G$-primes and $H$-primes compare when $H$ has finite index in $G$.

\begin{proposition} \label{prop:prime-subgroup}
Let $R$ be a $G$-ring and let $H$ be a finite index normal subgroup of $G$.
\begin{enumerate}
\item If $\fq$ is an $H$-prime then $\fp=I_G(\fq)$ is a $G$-prime, and $\fq$ is a minimal $H$-prime above $\fp$.
\item If $\fp$ is a $G$-prime then the minimal $H$-primes above $\fp$ form a single $G$-orbit, and $\fp$ is their intersection. In particular, $\fp$ is $H$-radical.
\item We have mutually inverse bijections
\begin{displaymath}
\xymatrix{
\{ \text{$G$-primes of $R$} \} \ar@<2pt>[r]^-{\Phi} & \{ \text{$H$-primes of $R$} \}/G \ar@<2pt>[l]^-{\Psi} }
\end{displaymath}
where $\Phi$ takes a $G$-prime to any minimal $H$-prime above it, and $\Psi$ takes an $H$-prime $\fq$ to $I_G(\fq)$.
\end{enumerate}
\end{proposition}

\begin{proof}
(a) We have already seen (Proposition~\ref{prop:int-prime}) that $\fp$ is $G$-prime. We now show that $\fq$ is minimal among the $H$-primes above $\fp$. Thus suppose that we have $\fp \subset \fq' \subset \fq$ where $\fq'$ is another $H$-prime. Let $x \in \fq$. Let $\sigma_1, \ldots, \sigma_n$ be coset representatives of $G/H$. For any $\tau_1, \ldots, \tau_n \in H$, we have $(\tau_1 \sigma_1 x) \cdots (\tau_n \sigma_n x) \in \prod_{\rho \in G/H} \rho \fq \subset \fp \subset \fq'$. Since $\fq'$ is $H$-prime, it follows that $\sigma_i x \in \fq'$ for some $i$, or, in other words, $x \in \sigma_i^{-1} \fq'$ for some $i$. Thus $\fq \subset \bigcup_{i=1}^n \sigma_i \fq'$, and so $\fq \subset \sigma_i \fq'$ for some $i$ by equivariant prime avoidance (Proposition~\ref{prop:avoid}). In particular, $\fq' \subset \sigma_i \fq'$. Multiplying each side by $\sigma_i$ gives $\sigma_i \fq' \subset \sigma_i^2 \fq'$. We thus get a chain
\begin{displaymath}
\fq' \subset \sigma_i \fq' \subset \sigma_i^2 \fq' \subset \cdots \subset \sigma_i^n \fq',
\end{displaymath}
where $n=[G:H]$. Since the group $G/H$ has order $n$, we have $\sigma_i^n \in H$, and so $\sigma_i^n \fq'=\fq'$ since $\fq'$ is $H$-stable. Thus $\sigma_i \fq'=\fq'$, and so $\fq \subset \fq'$, as desired.

(b) Let $\fq$ be a minimal $H$-prime above $\fp$, which exists by Proposition~\ref{prop:primes-exist} applied to $R/\fp$. By the Proposition~\ref{prop:zd}, every element of $\fq$ is an $H$-zerodivisor in $R/\fp$. In other words, for any $x \in \fq$ there exists $y \not\in \fp$ such that $x \cdot \tau y \in \fp$ for all $\tau \in H$. Let $\sigma_1, \ldots, \sigma_n$ be coset representatives for $G/H$. We then see that $(\sigma_1 x \cdots \sigma_n x) \cdot \rho y \in \fp$ for all $\rho \in G$. Since $y \not\in \fp$ and $\fp$ is $G$-prime, it follows that $\sigma_1 x \cdots \sigma_n x \in \fp$. Thus $\prod_{i=1}^n \sigma_i \fq \subset \fp$. Consider $\fp' = \bigcap_{i=1}^n \sigma_i \fq$. Then $\fp \subset \fp'$ and, by what we have just shown, $(\fp')^n \subset \fp$. Since $\fp'$ is $G$-prime by (a), we have $\fp=\fp'$.

To complete the proof, suppose that $\fq'$ is a second $H$-prime minimal above $\fp$. The same reasoning as above applies, and so $\bigcap_{i=1}^n \sigma_i \fq = \bigcap_{i=1}^n \sigma_i \fq'$. We thus find that $\fq \subset \sigma_i \fq'$ and $\fq' \subset \sigma_j \fq$ for some $i$ and $j$. We thus see that $\fq \subset \sigma_i \fq' \subset \sigma_i \sigma_j \fq$. Arguing as in (a), we find $\sigma_i \sigma_j \fq=\fq$, and so $\fq=\sigma_i \fq'$. This completes the proof of (b).

(c) This follows from (a) and (b).
\end{proof}

\section{The infinite polynomial ring} \label{s:ring}

In this section, we introduce the primary objects of study of this paper: the infinite polynomial ring $R$ and the $\fS$-ideals $\fP(\lambda,e;Z)$. We also introduce a number of axuiliary objects, and establish some basic results. Finally, we recall the classification of radical $\fS$-primes from \cite{svar}.

\subsection{The ring $R$}

Fix a commutative ring $A$. We let $R=A[\xi_i]_{i \ge 1}$ be the polynomial ring over $A$ in the variables $\xi_i$ with $i \ge 1$. The infinite symmetric group $\fS$ acts on $R$ by permuting the $\xi$ variables, and this action is smooth. This paper is a study of the $\fS$-equivariant commutative algebra of $R$, especially the $\fS$-primes and the $\fS$-spectrum.

\subsection{Cohen's theorem}

We recall Cohen's fundamental theorem \cite{cohen, cohen2} on the ring $R$:

\begin{theorem} \label{thm:cohen}
Suppose that $A$ is noetherian. Then $R$ is $\fS$-noetherian, that is, any ascending chain of $\fS$-ideals stabilizes. In particular, any $\fS$-ideal is generated by the $\fS$-orbits of finitely many elements.
\end{theorem}

In fact, there is a stronger result concern finitely generated $(R, \fS)$-modules, but we will not need that in this paper. We give a few consequences of this theorem:

\begin{corollary}
Suppose that $A$ is noetherian.
\begin{enumerate}
\item The $\fS$-spectrum $\Spec_{\fS}(R)$ of $R$ is noetherian.
\end{enumerate}
Now let $\fa$ be an $\fS$-ideal of $R$.
\begin{enumerate}[resume]
\item The ring $R/\fa$ has finitely many minimal $\fS$-primes.
\item The $\fS$-radical of $\fa$ is a finite intersection of $\fS$-primes.
\item We have $(\rad_{\fS}(\fa))^n \subset \fa$ for some $n \ge 1$.
\end{enumerate}
\end{corollary}

\begin{proof}
(a) Let $Z_{\bullet}$ be a descending chain of closed subsets of $\Spec_{\fS}(R)$. Write $Z_i=V_{\fS}(\fb_i)$ where $\fb_i$ is an $\fS$-radical ideal of $R$. Then $\fb_{\bullet}$ is an ascending chain of $\fS$-ideals of $R$, and thus stabilizes by Cohen's theorem. Hence the original chain $Z_{\bullet}$ stabilizes, and so $\Spec_{\fS}(R)$ is noetherian.

(b) The minimal $\fS$-primes of $R/\fa$ correspond to the irreducible components of $V_{\fS}(\fa) \subset \Spec_{\fS}(R)$. There are finitely many of these since $\Spec_{\fS}(\fa)$ is a noetherian space. The result follows.

(c) The minimal $\fS$-primes over $\fa$ correspond to the minimal primes of $R/\fa$, and so there are finitely many by (b). The result now follows from Proposition~\ref{prop:rad-primes}

(d) Let $X$ be a finitely generated $\fS$-submodule of $\rad_{\fS}(\fa)$ that generates it as an ideal; this exists by Cohen's theorem. By Proposition~\ref{prop:fin-nilp} we have $X^n \subset \fa$ for some $n$, from which the result follows.
\end{proof}

\subsection{Discriminants} \label{ss:disc}

Let $T$ be a commutative ring. We define the \defi{discriminant} of elements $x_1, \ldots, x_n \in T$ by
\begin{displaymath}
\Delta(x_1, \ldots, x_n) = \prod_{1 \le i<j \le n} (x_i-x_j).
\end{displaymath}
We also have the following additive expression for the discriminant:
\begin{displaymath}
\Delta(x_1, \ldots, x_n) = \sum_{\sigma \in \fS_n} \sgn(\sigma) x_{\sigma(1)}^{n-1} x_{\sigma(2)}^{n-2} \cdots x_{\sigma(n)}^0.
\end{displaymath}
We define $\Delta_n \in R$ to be the discriminant of $\xi_1, \ldots, \xi_n$. More generally, for a finite subset $S$ of $[\infty]$, we let $\Delta_S \in R$ be the discriminant of $\{\xi_i\}_{i \in S}$, ordered in the usual manner. Discriminants will play an important role in our study of $\fS$-ideals. The following simple proposition already hints that this might be the case:

\begin{proposition} \label{prop:disc}
Let $\fa$ be an $\fS$-ideal of $R$ and let $f \in \fa$ be non-zero. Then $\fa$ contains $c \cdot \Delta_n$ for some $n$, where $c \in A$ is some non-zero coefficient appearing in $f$.
\end{proposition}

\begin{proof}
See \cite[Proposition~2.6]{svar}.
\end{proof}

\subsection{Partitions}

We now introduce some basic terminology related to partitions that will be used throughout the paper. An \defi{$\infty$-partition} is a finite tuple $(\lambda_1, \ldots, \lambda_r)$ where $1 \le \lambda_{\alpha} \le \infty$, $\lambda_{\alpha} \ge \lambda_{\alpha+1}$, and $\lambda_1=\infty$.  More generally, \defi{$\infty$-composition} is a pair $(\lambda,\cI)$ consisting of a finite set $\cI$ and a function $\lambda \colon \cI \to \bZ_{\ge 1} \cup \{\infty\}$ such that $\lambda_{\alpha}=\infty$ for some $\alpha \in \cI$. We typically just write $\lambda$ for an $\infty$-composition, and refer to $\cI$ as the index set when needed. We regard an $\infty$-partition $\lambda=(\lambda_1, \ldots, \lambda_r)$ as an $\infty$-composition on the index set $[r]$. For an $\infty$-composition $\lambda$ on index set $\cI$, we write $\cI^{\infty}$ (resp.\ $\cI^{\rf}$) for the subset of $\cI$ consisting of indices $\alpha$ for which $\lambda_{\alpha}=\infty$ (resp.\ $\lambda_{\alpha}$ is finite), assuming there is no danger of ambiguity.

An \defi{isomorphism} of $\infty$-compositions $(\lambda,\cI) \to (\mu, \cK)$ is a bijection $\phi \colon \cI \to \cK$ of index sets such that $\lambda_{\alpha}=\mu_{\phi(\alpha)}$ for all $\alpha \in \cI$. Every $\infty$-composition is isomorphic to a unique $\infty$-partition. For an $\infty$-composition $\lambda$, we let $\rN_{\lambda}$ be the automorphism group of $\lambda$, i.e., the group of all isomorphisms from $\lambda$ to itself. This is a subgroup of the symmetric group on the index set, and is thus a finite group.

Let $\lambda$ be an $\infty$-composition with index set $\cI$. A \defi{weight function} on $\cI$ is a function $e \colon \cI \to \bZ_{\ge 1}$. We say that $e$ is \defi{($\lambda$-)reduced} if $e_{\alpha}=1$ whenever $\lambda_{\alpha}$ is finite, and define the \defi{($\lambda$-)reduction} of $e$, denoted $\red_{\lambda}(e)$ to be the weight function that coincides with $e$ on $\cI^{\infty}$ and is~1 on $\cI^{\rf}$. We write $e_{\max}$ for the maximum value of $e$.


Let $[\infty]=\{1,2,\ldots\}$ be the set of positive integers. A \defi{partition of $[\infty]$} is a finite family $\cU=\{U_{\alpha}\}_{\alpha \in \cI}$ where each $U_{\alpha}$ is a non-empty subset of $[\infty]$ such that the $U_{\alpha}$'s are disjoint and cover $[\infty]$. Let $\cU$ be a partition of $[\infty]$. We let $\pi \colon [\infty] \to \cI$ be the function assigning to each $i \in [\infty]$ the unique index $\alpha \in \cI$ such that $i \in U_{\alpha}$. For $\alpha \in \cI$, let $\lambda_{\alpha}=\# U_{\alpha}$. Then $\lambda$ is an $\infty$-composition with index set $\cI$. We call this the $\infty$-composition \defi{associated} to $\cU$, and say that $\cU$ has \defi{type} $\lambda$.

Let $\fY_{\cU}$ be the subgroup of $\fS$ consisting of all permutations $\sigma$ such that $\sigma(U_{\alpha}) \subset U_{\alpha}$ for all $\alpha \in \cI$. Then $\fY_{\cU}$ is the product of symmetric groups on the sets $U_{\alpha}$. We refer to $\fY_{\cU}$ as the \defi{Young subgroup} of $\fS$ associated to $\cU$. Let $\rN \fY_{\cU}$ be the normalizer of $\fY_{\cU}$ in $\fS$. This consists of all $\sigma \in \fS$ such that $\pi(\sigma(i))=\pi(\sigma(j))$ whenever $\pi(i)=\pi(j)$. There is a canonical short exact sequence
\begin{displaymath}
1 \to \fY_{\cU} \to \rN \fY_{\cU} \to \rN_{\lambda} \to 1
\end{displaymath}
that non-canonically splits. In particular, we see that $\fY_{\cU}$ is a finite index normal subgroup of $\rN\fY_{\cU}$.

\begin{proposition} \label{prop:big-young}
The Young subgroup $\fY_{\cU}$ is a big subgroup of $\fS$ in the sense of Definition~\ref{defn:big}.
\end{proposition}

\begin{proof}
Recall that $\fS_{>n}$ is the subgroup of $\fS$ fixing the numbers $1, \ldots, n$. The quotient $\fS/\fS_{>n}$ is isomorphic to a $\fS$-subset of $X^n$, where $X=[\infty]$ is the usual permutation set on which $\fS$ acts. Since $X$ is finitely generated as a $\fY_{\cU}$-set and $\fY_{\cU}$ is admissible, it follows that $X^n$ is finitely generated as a $\fY_{\cU}$-set. Thus the $\fY_{\cU}$-subset $\fS/\fS_{>n}$ is also finitely generated, which completes the proof.
\end{proof}

\subsection{Auxiliary rings} \label{ss:setup}

We now define some auxiliary rings that will be used throughout the paper. Let $\cU=\{U_{\alpha}\}_{\alpha \in \cI}$ be a partition of $[\infty]$, let $\pi \colon [\infty] \to \cI$ be the associated map, and let $\lambda$ be the associated $\infty$-composition.
\begin{itemize}
\item We put $R_{\cU}=R$. We use the $R_{\cU}$ notation to indicate context, e.g., if $\rN \fY_{\cU}$ is the relevant symmetry group.
\item We let $S_{\cU,0}$ be the ring $A[t_{\alpha}]_{\alpha \in \cI}$. This is a finite variable polynomial ring over $A$. The group $\rN_{\lambda}$ naturally acts by permuting the $t$'s.
\item We let $S_{\cU}$ be the ring $S_{\cU,0}[\epsilon_i]_{i \ge 1}$. The group $\rN \fY_{\cU}$ acts on $S_{\cU}$, by acting on both the $\epsilon_i$'s and on the $t_{\alpha}$'s.
\end{itemize}
We also introduce some localizations of these rings:
\begin{itemize}
\item $\tilde{R}_{\cU}$ is the localization of $R_{\cU}$ where we invert $\xi_i-\xi_j$ whenever $\pi(i) \ne \pi(j)$.
\item $\tilde{S}_{\cU,0}$ is the localization of $S_{\cU,0}$ where we invert $t_{\alpha}-t_{\beta}$ for all $\alpha \ne \beta$.
\item $\tilde{S}_{\cU}$ is the ring $\tilde{S}_{\cU,0}[\epsilon_i]_{i \ge 1}$.
\end{itemize}
A few more definitions:
\begin{itemize}
\item We let $\bA^{\cI}$ (or $\bA^{\cI}_A$) be the spectrum of $S_{\cU,0}$; this is an affine space over $\Spec(A)$ with coordinates indexed by $\cI$.
\item We let $\bU^{\cI}$ be the spectrum of $\tilde{S}_{\cU,0}$; this is the open subscheme of $\bA^{\cI}$ where the coordinates are distinct.
\item For a subset $Z$ of $\bU^{\cI}$ (i.e., a set of scheme-theoretic points), we let $\fI_0(Z) \subset \tilde{S}_{\cU,0}$ be its radical ideal, and we let $\fI(Z)$ be the extension of $\fI_0(Z)$ to $\tilde{S}_{\cU}$, which is also radical.
\item We let $I_{\cU,n} \subset R_{\cU}$ be the ideal generated by $(\xi_i-\xi_j)^n$ for all $i,j \in [\infty]$ with $\pi(i)=\pi(j)$. We let $\tilde{I}_{\cU,n}$ be its extension to $\tilde{R}_{\cU}$. We just write $I_{\cU}$ or $\tilde{I}_{\cU}$ when $n=1$. If $\fa$ is a $\fY_{\cU}$-ideal of $\tilde{R}_{\cU}$ then $\rad(\fa)$ contains $\tilde{I}_{\cU}$ if and only if $\fa$ contains $\tilde{I}_{\cU,n}$ for some $n$. 
\item We let $J_{\cU,n} \subset S_{\cU}$ be the ideal generated by $\epsilon_i^n$ for all $i \in [\infty]$. We let $\tilde{J}_{\cU,n}$ be its extension to $\tilde{S}_{\cU}$. We write $J_{\cU}$ or $\tilde{J}_{\cU}$ when $n=1$. As above, if $\fa$ is a $\fY_{\cU}$-ideal of $\tilde{S}_{\cU}$ then $\rad(\fa)$ contains $\tilde{J}_{\cU}$ if and only if $\fa$ contains $\tilde{J}_{\cU,n}$ for some $n$.
\item We let $\iota \colon R_{\cU} \to S_{\cU}$ be the $A$-algebra homomorphism defined by $\iota(\xi_i)=t_{\pi(i)}+\epsilon_i$.
\end{itemize}
We note that $\iota$ does \emph{not} induce a homomorphism $\tilde{R}_{\cU} \to \tilde{S}_{\cU}$. However, we do have the following observation, which will be adequate for our purposes:

\begin{proposition} \label{prop:iota} 
The map $\iota$ induces a ring homomorphism $\tilde{R}_{\cU} \to \tilde{S}_{\cU}/\tilde{J}_{\cU,n}$ for any $n$. The kernel of this homomorphism contains $\tilde{I}_{\cU,2n-1}$.
\end{proposition}

\begin{proof}
To prove the first statement, we must show that $\xi_i-\xi_j$ maps to a unit of $\tilde{S}_{\cU}/\tilde{J}_{\cU,n}$ when $\pi(i) \ne \pi(j)$. Put $\alpha=\pi(i)$ and $\beta=\pi(j)$. Then
\begin{displaymath}
\iota(\xi_i-\xi_j)=(t_{\alpha}-t_{\beta})+(\epsilon_i-\epsilon_j).
\end{displaymath}
Since $t_{\alpha}-t_{\beta}$ is a unit of $\tilde{S}_{\cU}$ and $\epsilon_i$ and $\epsilon_j$ are nilpotent in $\tilde{S}_{\cU}/\tilde{J}_{\cU,n}$, it follows that $\iota(\xi_i-\xi_j)$ is a unit of $\tilde{S}_{\cU}/\tilde{J}_{\cU,n}$, as required.

We now show that the kernel of $\iota \colon \tilde{R}_{\cU} \to \tilde{S}_{\cU}/\tilde{J}_{\cU,n}$ contains $\tilde{I}_{\cU,2n-1}$. Suppose that $\pi(i)=\pi(j)=\alpha$. Then
\begin{displaymath}
\iota(\xi_i-\xi_j)=(t_{\alpha}+\epsilon_i)-(t_{\alpha}+\epsilon_j)=\epsilon_i-\epsilon_j,
\end{displaymath}
and so
\begin{displaymath}
\iota((\xi_i-\xi_j)^{2n-1}) = (\epsilon_i-\epsilon_j)^{2n-1} = \sum_{k=0}^{2n-1} \binom{2n-1}{k} \epsilon_i^k \epsilon_j^{2n-1-k}.
\end{displaymath}
In each term in the sum, we have either $k \ge n$ or $2n-1 - k \ge n$, and so the expression belongs to $\tilde{J}_{\cU,n}$. The result follows.
\end{proof}

\begin{remark}
See Lemma~\ref{lem:class-2-3} for more information on the kernel of $\tilde{R}_{\cU} \to \tilde{S}_{\cU}/\tilde{J}_{\cU,n}$.
\end{remark}

\begin{corollary}
Let $\fa$ be an ideal of $\tilde{S}_{\cU}$ that contains $\tilde{J}_{\cU,n}$ for some $n$. Then we can form the contraction $\fa^c$ of $\fa$ to $\tilde{R}_{\cU}$ along $\iota$, and $\fa^c$ contains $\tilde{I}_{\cU,2n-1}$.
\end{corollary}

Our basic strategy for studying $\fS$-primes of $R$ is to study extension and contraction along $\iota$, for appropriate choices of $\cU$. This is carried out in detail in \S \ref{s:class}.

\subsection{Some $\fS$-primes}

The following result gives some simple examples of $\fS$-primes:

\begin{proposition} \label{prop:powers-prime}
Suppose $A$ is a domain. Then the ideal $\langle \xi_i^n \rangle_{i \ge 1}$ is an $\fS$-prime of $R$ for any $n \ge 1$.
\end{proposition}

\begin{proof}
Call this ideal $\fp$. Suppose we have $f,g \in R$ such that $f \cdot (\sigma g) \in \fp$ for all $\sigma \in \fS$. Choose $\sigma$ so that $f$ and $\sigma g$ are in disjoint sets of variables. Since $f \cdot (\sigma g) \in \fp$, every monomial appearing in it contains an $n$th (or higher) power of some variable. It follows that the same must be true for $f$ or $\sigma g$ on its own, and so $f \in \fp$ or $g \in \fp$, as required.
\end{proof}

This example is important for two reasons. First, it shows that $\fS$-primes of $R$ need not be radical. And second, we can extend this example by combining it with the construction of radical $\fS$-primes from \cite{svar} to obtain many more $\fS$-primes. This leads to the $\fQ$ and $\fP$ ideals defined below.

\subsection{The $\fQ$ ideals}

Let $\lambda$ be an $\infty$-composition with index set $\cI$, let $\cU=\{U_{\alpha}\}_{\alpha \in \cI}$ be a partition of $[\infty]$ of type $\lambda$, and use other notation as in \S \ref{ss:setup}. Suppose that $Z$ is a subset of $\bU^{\cI}$ and $e$ is a weight function on $\cI$.

\begin{definition} \label{defn:Q-ideal}
We let $\fQ_{\cU}(\lambda,e;Z)$ be the ideal of $\tilde{S}_{\cU}$ generated by $\fI_0(Z)$ and $\epsilon_i^{e(\alpha)}$ for $\alpha \in \cI$ and $i \in U_{\alpha}$.
\end{definition}

It is clear that $\fQ_{\cU}(\lambda,e;Z)$ is a $\fY_{\cU}$-ideal. When $Z=\bU^{\cI}$, we omit $Z$ from the notation, and just write $\fQ_{\cU}(\lambda,e)$. Similarly, when $e=1$ identically, we omit it from the notation, and just write $\fQ_{\cU}(\lambda;Z)$. We note that $\fI_0(Z)$, and therefore $\fQ_{\cU}(\lambda, e; Z)$, only depends on the Zariski closure of $Z$, so we will typically assume $Z$ to be a closed subset of $\bU^{\cI}$. The $\fQ$ ideals play an important, though auxiliary, role in this paper. The following proposition gives the most notable property of the $\fQ$ ideals, and generalizes Proposition~\ref{prop:powers-prime}:

\begin{proposition} \label{prop:Qprime}
If $e$ is reduced and $Z$ is irreducible then $\fQ_{\cU}(\lambda, e; Z)$ is $\fY_{\cU}$-prime.
\end{proposition}

\begin{proof}
Put $\fq=\fQ_{\cU}(\lambda,e;Z)$. For simplicity, identify $\cI$ with $[r]$, and do so in such a way that $\lambda_1, \ldots, \lambda_s$ are infinite and $\lambda_{s+1}, \ldots, \lambda_r$ are finite. We can then identify $\tilde{S}_{\cU}$ with a polynomial ring over $\tilde{S}_{\cU,0}$ in variables $\xi_{i,j}$ where $1 \le i \le r$ and, when $\lambda_i$ is finite, $1 \le j \le \lambda_i$. Under this identification, $\fY_{\cU}$ is identified with the product of the $\fS_{\lambda_i}$'s. Let $B=\tilde{S}_{\cU,0}/\fI_0(Z)$, which is a domain. Since $\fq$ contains $\fI_0(Z)$ and the $\xi_{i,j}$ with $\lambda_i$ finite, it suffices to show that $\fq$ extends to a $\fY_{\cU}$-prime when we quotient by them. The quotient ring is $B[\xi_{i,j}]_{1 \le i \le s}$, the extension of $\fq$ is the ideal generated by $\xi_{i,j}^{e(i)}$ for all $i$ and $j$, and the group $\fY_{\cU}$ is identified with $\fS^s$ (the finite parts are now irrelevant since we have killed the associated variables). This ideal is $\fS^s$-prime by using exactly the same argument as from Proposition~\ref{prop:powers-prime}.
\end{proof}

We now establish a few more properties of the $\fQ$-ideals.

\begin{proposition} \label{prop:Qrad}
The radical of $\fQ_{\cU}(\lambda, e; Z)$ is $\fQ_{\cU}(\lambda; Z)$.
\end{proposition}

\begin{proof}
It is clear from the definitions that $\rad(\fQ_{\cU}(\lambda,e;Z)) = \rad(\fQ_{\fU}(\lambda;Z))$. We have $\tilde{S}_{\cU}/\fQ_{\cU}(\lambda;Z)=\tilde{S}_{\cU,0}/\fI_0(Z)$, and so $\fQ_{\cU}(\lambda;Z)$ is radical since $\fI_0(Z)$ is.
\end{proof}

\begin{proposition} \label{prop:Qunion}
Let $\{Z_i\}_{i \in I}$ be a family of subsets of $\bU^{\cI}$, and put $Z=\bigcup_{i \in I} Z_i$. Then $\fQ_{\cU}(\lambda, e; Z)=\bigcap_{i \in I} \fQ_{\cU}(\lambda, e; Z_i)$.
\end{proposition}

\begin{proof}
Let $\fe$ be the ideal of $\tilde{S}_{\cU}$ generated by $\epsilon_i^{e(\alpha)}$ with $i \in U_{\alpha}$. Let $\cM$ be the set of monomials in the $\epsilon_i$'s that do not belong to $\fe$. If $\fI_0$ is an ideal of $\tilde{S}_{\cU,0}$ and $\fI$ is its extension to $\tilde{S}_{\cU}$, then $\tilde{S}_{\cU}/(\fI+\fe)$ is free as a $\tilde{S}_{\cU,0}/\fI_0$-module with basis $\cM$. Now, we must show
\begin{displaymath}
\fI(Z)+\fe=\bigcap_{i \in I} (\fI(Z_i)+\fe)
\end{displaymath}
Call the left ideal $\fa$ and the right ideal $\fb$. It is clear that $\fa \subset \fb$. We now prove the reverse. Suppose $f$ belongs to $\fb$. Since $\fa$ and $\fb$ both contain $\fe$, we may as well suppose that $f$ is a $\tilde{S}_{\cU,0}$-linear combination of the monomials in $\cM$; write $f=\sum_{m \in \cM} c_m m$ with $c_m \in \tilde{S}_{\cU,0}$. Since $f$ maps to~0 in $\tilde{S}_{\cU}/(\fI(Z_i)+J)$, we see that $c_m \in \fI_0(Z_i)$. Since this holds for all $i$ and $\bigcap_{i \in I} \fI_0(Z_i)=\fI_0(Z)$, we see that $c_m \in \fI_0(Z)$. Thus $f \in \fa$, which completes the proof.
\end{proof}

\begin{proposition} \label{prop:Qrad2}
The $\fY_{\cU}$-radical of $\fQ_{\cU}(\lambda, e; Z)$ is $\fQ_{\cU}(\lambda, \red_{\lambda}(e); Z)$.
\end{proposition}

\begin{proof}
Let $\fa=\fQ_{\cU}(\lambda, e; Z)$ and $\fb=\fQ_{\cU}(\lambda, \red_{\lambda}(e); Z)$. Write $Z=\bigcup_{i \in I} Z_i$ where each $Z_i$ is irreducible (e.g., one can take the $Z_i$'s to be the 1-element subsets of $Z$). By Proposition~\ref{prop:Qunion}, $\fb$ is the intersection of the ideals $\fQ_{\cU}(\lambda, \red_{\lambda}(e); Z_i)$, each of which is $\fY_{\cU}$-prime by Proposition~\ref{prop:Qprime}. Thus $\fb$ is $\fY_{\cU}$-radical.

We have $\fa \subset \fb$, and so $\rad_{\fY_{\cU}}(\fa) \subset \fb$. We now prove the reverse inclusion. It suffices to show that $\epsilon_i \in \rad_{\fY_{\cU}}(\fa)$ whenever $i \in U_{\alpha}$ with $U_{\alpha}$ finite. Thus let $i \in U_{\alpha}$ be given with $U_{\alpha}$ finite. Let $N=e_{\alpha} \cdot \lambda_{\alpha}$. Given $\sigma_1, \ldots, \sigma_N \in \fY_{\cU}$, we have $N$ elements $\sigma_1(i), \ldots, \sigma_N(i)$ of $U_{\alpha}$, and so one of them, say $j$, must appear at least $e_{\alpha}$ times. It follows that $\sigma_1(\epsilon_i) \cdots \sigma_N(\epsilon_i)$ contains $\epsilon_j^{e_{\alpha}}$ as a factor, which belongs to $\fa$. Thus $\sigma_1(\epsilon_i) \cdots \sigma_N(\epsilon_i)$ belongs to $\fa$ for all choices of $\sigma_1, \ldots, \sigma_N$ in $\fY_{\cU}$, and so $\epsilon_i \in \rad_{\fY_{\cU}}(\fa)$.
\end{proof}

\begin{proposition} \label{prop:Q-ext} 
Let $A \to A'$ be a flat ring map, let $\tilde{S}'_{\cU}=A' \otimes_A \tilde{S}_{\cU}$, let $Z \subset \bU^{\cI}_A$ be a closed subset, let $Z' \subset \bU^{\cI}_{A'}$ be the inverse image of $Z$, and let $e$ be a $\lambda$-reduced weighting. Let $\fq=\fQ_{\cU}(\lambda,e;Z)$ and $\fq'=\fQ_{\cU}(\lambda,e;Z')$. Then:
\begin{enumerate}
\item We have $\rad_{\fY_{\cU}}(\fq^e)=\fq'$.
\item If $\fI_0(Z)^e \subset \tilde{S}'_{\cU,0}$ is radical then $\fq^e=\fq'$.
\end{enumerate}
We note that if $A \to A'$ is a localization then (b) holds. 
\end{proposition}

\begin{proof}
Let $\fe \subset \tilde{S}_{\cU}$ be the ideal generated by $\epsilon_i^{e(\alpha)}$ with $i \in U_{\alpha}$, and let $\fe' \subset \tilde{S}'_{\cU}$ be defined analogously. Thus $\fq=\fI(Z)+\fe$ and $\fq'=\fI(Z')+\fe'$. As $\fe^e=\fe'$, we find $\fq^e=\fI(Z)^e+\fe'$. Now, $\rad(\fI(Z)^e)=\fI(Z')$. If $\fI_0(Z)^e$ is radical then so is $\fI(Z)^e$, and so $\fq^e=\fq'$; this proves (b). More generally, we find $\fq^e \subset \fq'$, and so $\rad_{\fY_{\cU}}(\fq^e) \subset \fq'$ since $\fq'$ is $\fY_{\cU}$-radical (Proposition~\ref{prop:Qrad2}). On the other hand, note that $\rad_{\fY_{\cU}}(\fq^e)$ contains $\rad_{\fY_{\cU}}(\fI(Z)^e)+\rad_{\fY_{\cU}}(\fe')$. Since the action of $\fY_{\cU}$ on $\fI_0(Z)^e$ is trivial, we see that $\rad_{\fY_{\cU}}(\fI(Z)^e)$ contains the radical of $\fI(Z)^e$. It follows that $\rad_{\fY_{\cU}}(\fq^e)$  contains $\fq'$. Thus (a) holds.
\end{proof}

\subsection{The $\fP$ ideals} \label{ss:P}

Let $\lambda$ be an $\infty$-composition on the index set $\cI$, let $e$ be a weighting on $\cI$, and let $Z$ be a subset of $\bU^{\cI}$. Choose a partition $\cU$ of $[\infty]$ of type $\lambda$. We now come to the most important objects of this paper:

\begin{definition}
We define $\fP(\lambda,e;Z)$ to be the $\fS$-contraction of $\fQ_{\cU}(\lambda,e;Z)$ to $R$.
\end{definition}

We make a number of remarks concerning this definition:
\begin{itemize}
\item See Definition~\ref{defn:G-contract} for the definition of $\fS$-contraction and Definition~\ref{defn:Q-ideal} for the definition of $\fQ_{\cU}(\lambda,e;Z)$.
\item The $\fS$-contraction is formed with respect to the composition $R \to \tilde{R}_{\cU} \stackrel{\iota\,}{\to} \tilde{S}_{\cU}/\tilde{J}_{\cU,n}$, where $n$ is large enough so that $\fQ_{\cU}(\lambda,e;Z)$ contains $\tilde{J}_{\cU,n}$; in fact, one can take $n=e_{\max}$.
\item One easily sees that $\fP(\lambda,e;Z)$ is independent of the choice of $\cU$.
\item As with the $\fQ$ ideals, we omit $Z$ from the notation when $Z=\bU^{\cI}$, and we omit $e$ from the notation when $e=1$.
\item Suppose that $(\lambda',e';Z')$ is a second triple, and let $\cI'$ be the index set of $\lambda'$. We say that $(\lambda,e;Z)$ is \defi{isomorphic} to $(\lambda',e';Z')$ if there is a bijection $\phi \colon \cI \to \cI'$ under which $\lambda$, $e$, and $Z$ correspond to $\lambda'$, $e'$, and $Z'$. In this case, one easily sees that $\fP(\lambda,e;Z)$ is equal to $\fP(\lambda',e';Z')$.  We will eventually prove that the converse holds as well.
\item In \S \ref{ss:intro-P}, we gave a different description of $\fP(\lambda,e;Z)$ in characteristic~0. One can show that it is equivalent; see Lemma~\ref{lem:contract-a-3} for the idea.
\end{itemize}
The most important property of the $\fP$ ideals is the following:

\begin{proposition} \label{prop:Pprime}
If $e$ is reduced and $Z$ is irreducible then $\fP(\lambda, e; Z)$ is $\fS$-prime.
\end{proposition}

\begin{proof}
This follows from Proposition~\ref{prop:G-contract} and~\ref{prop:Qprime}.
\end{proof}

The following propositions give a few more simple properites of these ideals:

\begin{proposition} \label{prop:Pprop}
We have the following:
\begin{enumerate}
\item The radical of $\fP(\lambda, e; Z)$ is $\fP(\lambda; Z)$.
\item The $\fY_{\cU}$-radical of $\fP(\lambda, e; Z)$ is $\fP(\lambda, \red_{\lambda}(e); Z)$.
\item If $Z=\bigcup_{i \in I} Z_i$ then $\fP(\lambda, e; Z)=\bigcap_{i \in I} \fP(\lambda, e; Z_i)$.
\end{enumerate}
\end{proposition}

\begin{proof}
These statements follow from Propositions~\ref{prop:Qrad}, \ref{prop:Qunion}, and~\ref{prop:Qrad2}, and properties of $\fS$-contraction (Proposition~\ref{prop:G-contract}).
\end{proof}

\begin{proposition} \label{prop:P-ext} 
Let $A \to A'$ be a flat ring map, let $R'=A' \otimes_A R$, let $Z \subset \bU^{\cI}_A$ be a closed subset, let $Z' \subset \bU^{\cI}_{A'}$ be the inverse image of $Z$, and let $e$ be a $\lambda$-reduced weighting. Put $\fp=\fP(\lambda,e;Z)$ and $\fp'=\fP(\lambda,e;Z')$. Then:
\begin{enumerate}
\item We have $\rad_{\fS}(\fp^e)=\fp'$.
\item If $\fI_0(Z)^e \subset \tilde{S}'_{\cU,0}$ is radical then $\fp^e=\fp'$.
\end{enumerate}
We note that if $A \to A'$ is a localization then (b) holds.
\end{proposition}

\begin{proof}
Let $\fq=\fQ_{\cU}(\lambda,e;Z)$ and $\fq'=\fQ_{\cU}(\lambda,e;Z')$. Consider the diagram
\begin{displaymath}
\xymatrix{
\fp \ar@{}[r]|*=0[@]{\subset} & R \ar[d] \ar[r] & \tilde{S}_{\cU} \ar[d] & \ar@{}[l]|*=0[@]{\subset} \fq  \\
\fp' \ar@{}[r]|*=0[@]{\subset} & R' \ar[r] & \tilde{S}'_{\cU} & \ar@{}[l]|*=0[@]{\subset} \fq' }
\end{displaymath}
We have
\begin{equation} \label{eq:P-ext}
\fp^e = (\fq^{\fS c})^e = (\fq^e)^{\fS c}
\end{equation}
where in the first step we used the definition of $\fp$, and in the second Proposition~\ref{prop:Gc-bc}.

(a) Taking $\fS$-radicals in \eqref{eq:P-ext}, we find
\begin{displaymath}
\rad_{\fS}(\fp^e) = \rad_{\fS}((\fq^e)^{\fS c})=(\rad_{\fY_{\cU}}(\fq^e))^{\fS c} = (\fq')^{\fS c} = \fp'
\end{displaymath}
where in the second step we used Proposition~\ref{prop:G-contract}(c), in the third Proposition~\ref{prop:Q-ext}(a), and in the fourth the definition of $\fp'$.

(b) By Proposition~\ref{prop:Q-ext}(b), we have $\fq^e=\fq'$, so \eqref{eq:P-ext} gives $\fp^e=\fp'$.
\end{proof}

We now give a slightly different way to construct $\fP(\lambda,e;Z)$ which is sometimes convenient. Let $\psi \colon R \to T$ be a ring homomorphism, where $T$ is some commutative ring. We say that $\psi$ has \defi{class} $(\lambda,e;Z)$ if there exists a partition $\cU$ of $[\infty]$ of type $\lambda$ and a ring homomorphism $\psi_0 \colon \tilde{S}_{\cU,0} \to T$ such that the following conditions hold:
\begin{enumerate}
\item We have $\psi \vert_A = \psi_0 \vert_A$.
\item We have $\fI_0(Z) \subset \ker(\psi_0)$.
\item We have $\psi(\xi_i)=\psi_0(t_{\pi(i)})+\delta_i$ where $\delta_i \in T$ satisfies $\delta_i^{e(\pi(i))}=0$ and $\pi \colon [\infty] \to \cI$ is the function associated with $\cU$.
\end{enumerate}

\begin{proposition} \label{prop:ker-psi}
The following ideals of $R$ coincide:
\begin{enumerate}
\item $\fP(\lambda,e;Z)$
\item The intersection of $\ker(\psi)$ over all homomorphisms $\psi \colon R \to T$ of class $(\lambda,e;Z)$.
\item The intersection of $\ker(\psi)$ over all homomorphisms $\psi \colon R \to T$ of class $(\lambda,e;Z)$ where the nilradical of $T$ is maximal.
\end{enumerate}
\end{proposition}

\begin{proof}
Let $\fp=\fP(\lambda,e;Z)$, let $\fb$ be the ideal in (b), and let $\fc$ be the ideal in (c). Then $\fb \subset \fc$. We first claim that $\fp \subset \fb$. To see this, let $\psi \colon R \to T$ be a homomorphism of class $(\lambda,e;Z)$. We must show that $\fp \subset \ker(\psi)$. Write $\psi(\xi_i)=\psi_0(t_{\pi(i)})+\delta_i$ as above, where $\pi \colon [\infty] \to \cI$ is associated to a partition $\cU$ of $[\infty]$ of type $\lambda$. Let $\tilde{\psi} \colon \tilde{S}_{\cU} \to T$ be the homomorphism defined by $\tilde{\psi} \vert_{\tilde{S}_{\cU,0}} =\psi_0$ and $\tilde{\psi}(\epsilon_i)=\delta_i$. Then $\fQ_{\cU}(\lambda,e;Z) \subset \ker(\tilde{\psi})$ and $\psi=\tilde{\psi} \circ \iota$. Since $\iota$ carries $\fp$ into $\fQ_{\cU}(\lambda,e;Z)$, we see that $\fp \subset \ker(\psi)$, as required.

We now show $\fc \subset \fp$. Fix a partition $\cU$ of $[\infty]$ of type $\lambda$. Let $\Sigma$ be the set of all homomorphisms $\psi \colon R \to T$ as in (c) defined with respect to $\cU$. Let $\tilde{\Sigma}$ be the set of associated homorphisms $\tilde{\psi} \colon \tilde{S}_{\cU} \to T$ as in the previous paragraph. Let $\fd$ be the intersection of $\ker(\tilde{\psi})$ over $\tilde{\psi} \in \tilde{\Sigma}$. From the relation $\psi=\tilde{\psi} \circ \iota$, we see $\fc \subset \fd^c$.

We claim that $\fd$ coincides with $\fq=\fQ_{\cU}(\lambda,e;Z)$. We have $\fq\subset \ker(\tilde{\psi})$ for each $\tilde{\psi} \in \tilde{\Sigma}$ by the first paragraph, and so $\fq \subset \fd$. Now suppose that $f$ is an element of $\tilde{S}_{\cU}$ that does not belong to $\fq$. Let $\cM$ be the set of monomials in the $\epsilon_i$'s that do not belong to $\fq$. Then, as in the proof of Proposition~\ref{prop:Qunion}, $\tilde{S}_{\cU}/\fq$ is a free $\tilde{S}_{\cU,0}/\fI_0(Z)$-module with basis $\cM$. Let $c$ be a non-zero coefficient of $f$. Since $\fI_0(Z)$ is radical, the intersection of all primes of $\tilde{S}_{\cU,0}/\fI_0(Z)$ is zero, and so we can find some prime $\fr$ that does not contain $c$. Let $\Omega$ be the residue field of $\fr$. Then homorphism $\tilde{S}_{\cU} \to \Omega[\epsilon_i]/\langle \epsilon_i^{e(\pi(i))} \rangle$ belongs to $\tilde{\Sigma}$ and does not contain $f$ in its kernel. Thus $f \not\in \fd$, and so $\fd=\fq$, as claimed. Since $\fc$ is $\fS$-stable and contained in the contraction of $\fq$, it is therefore contained in the $\fS$-contraction of $\fq$, which is $\fp$. This completes the proof.
\end{proof}

\subsection{Radical $\fS$-primes} \label{ss:radical}

We have seen that $\fP(\lambda; Z)$ is a radical $\fS$-prime when $Z$ is irreducible. In fact, the main result of \cite{svar}, which we now recall, asserts that these essentially account for all such ideals:

\begin{theorem} \label{thm:rad-prime}
Let $\fp$ be a radical $\fS$-prime of $R$. Then exactly one of the following two cases hold:
\begin{enumerate}
\item $\fp$ is the extension to $R$ of its contraction to $A$ (which is a prime ideal of $A$)
\item $\fp=\fP(\lambda;Z)$ for some $\infty$-composition $\lambda$ and irreducible closed subvariety $Z$ of $\bU^{\cI}$, where $\cI$ is the index set of $\lambda$. 
\end{enumerate}
Moreover, in case (b), the pair $(\lambda;Z)$ is unique up to isomorphism.
\end{theorem}

\begin{proof}
There are some differences between our setup and the one from \cite{svar}, so we explain how to bridge the gaps.

First suppose that $A=K$ is a field. Then the main theorem of \cite{svar} shows that the non-zero $\fS$-irreducible closed subsets of $\fX=\Spec(R)$ are in bijection with pairs $(\lambda;Z)$ where $\lambda$ is an $\infty$-partition on the index set $[r]$ and $Z$ is an $\rN_{\lambda}$-irreducible closed subset of $\bU^r_K$. Explicitly, this bijection works as follows. Let $(\lambda;Z)$ be given and let $\fZ \subset \fX$ be the corresponding $\fS$-irreducible closed subset. Let $\cU$ be a partition of $[\infty]$ of type $\lambda$ and let $\fX_{[\cU]}$ be the subset of $\fX$ consisting of points $(x_i)_{i \ge 1}$ such that $x_i=x_j$ if and only if $i$ and $j$ belong to the same part of $\cU$. Then $\fX_{[\cU]}$ is isomorphic to $\bU^r$. Via this identification, we can regard $Z$ as a subset of $\fX$. The set $\fZ$ is then the Zariski closure of the union of the $\fS$-translates of $Z$.

Now, let $(\lambda,Z)$ be given corresponding to $\fZ$. Let $Z'$ be an irreducible component of $Z$. Then the union of the $\fS$-translates of $Z$ coincides with the union of the $\fS$-translates of $Z'$. Thus we could have defined $\fZ$ using $Z'$ as well. The ideal of $\fZ$ is exactly the ideal $\fP(\lambda;Z')$ by Proposition~\ref{prop:ker-psi}: indeed, in this case the intersection of $\ker(\psi)$'s is exactly describing the ideal for the Zariski closure of the $\fS$-translates of $Z'$. This shows that every radical $\fS$-prime has the form $\fP(\lambda;Z')$. Moreover, the uniquess from \cite{svar} shows that $\fP(\lambda;Z')=\fP(\lambda;Z'')$ if and only if $Z''=\sigma Z'$ for some $\sigma \in \rN_{\lambda}$, which translates to the stated uniqueness in this theorem.

We now treat the case of general coefficient ring. (The results of \cite{svar} are only stated for noetherian coefficient rings.) Let $\fp$ be a radical $\fS$-prime of $R$. Let $\fc=A \cap \fp$, which is a prime ideal of $A$ by Proposition~\ref{prop:prime-contract} (since $\fS$ acts trivially on $A$). Let $K$ be the residue field of $\fc$ and let $R'=K \otimes_A R$. Then $\fp$ corresponds to a radical $\fS$-prime of $R/\fc R$ which (by Proposition~\ref{prop:prime-loc}) extends to a radical $\fS$-prime, say $\fp'$ of $R'$. If $\fp' = 0$ then $\fp$ is the extension of $\fc$ and we are in case (a). On the other hand, if $\fp'$ is non-zero, then by the above, $\fp'$ has the form $\fP(\lambda;Z')$ for some $\infty$-composition $\lambda$ on an index set $\cI$ and irreducible closed subset $Z' \subset \bU^{\cI}_K$. Let $Z$ be the closure of $Z'$ in $\bU^{\cI}_A$. Then $\fP(\lambda;Z)$ is easily seen to contain $\fc$, and can thus be regarded as an ideal of $R/\fc R$. Its extension to $R'$ coincides with $\fP(\lambda;Z')$ by Proposition~\ref{prop:P-ext}. Thus, by Proposition~\ref{prop:prime-loc}, we have $\fp=\fP(\lambda;Z)$, which completes the proof.
\end{proof}

Suppose now that $\fp$ is an arbitrary $\fS$-prime of $R$. Then $\rad(\fp)$ is a radical $\fS$-prime by Proposition~\ref{prop:rad-prime}, and so we can apply the above theorem to it. This leads to the following important definition:

\begin{definition}
Let $\fp$ be an $\fS$-prime of $R$. We say that $\fp$ is \defi{finitary} if $\rad(\fp)$ has the form $\fP(\lambda; Z)$. In this case, we define the \defi{type} of $\fp$ to be (the isomorphism class of) $\lambda$.
\end{definition}

For example, let $\fp$ be an $\fS$-prime of the form $\fP(\lambda,e;Z)$. Then $\rad(\fp)=\fP(\lambda;Z)$ (Proposition~\ref{prop:Pprop}), and so $\fp$ is finitary of type $\lambda$. The following proposition classifies the non-finitary primes:

\begin{proposition} \label{prop:non-finitary}
Let $\fp$ be an $\fS$-prime of $R$ that is not finitary. Then $\fp$ is radical, and thus the extension to $R$ of a prime of $A$.
\end{proposition}

\begin{proof}
Let $\fp_0$ be the contraction of $\fp$ to $A$. This is a prime ideal of $A$: indeed, it is $\fS$-prime, but $\fS$ acts trivially on $A$. Now, the contraction of $\rad(\fp)$ to $A$ is $\rad(\fp_0)=\fp_0$. By assumption, $\rad(\fp)=\fp_0 R$. Since $\fp$ contains $\fp_0$, it thus contains $\rad(\fp)$, and so $\fp=\rad(\fp)$. The result follows.
\end{proof}

We now investigate finitary $\fS$-primes a bit more. Let $\fX=\Spec(R)$ and let $x$ be a $K$-point of $\fX$, for some field $K$. We say that $x$ is \defi{finitary} if the coordinates of $x$ take on only finitely many values. This condition is invariant under field extension, and so can be defined for scheme-theoretic points of $\fX$. The following proposition shows the relevance of this property:

\begin{proposition} \label{prop:finitary-char}
Let $\fp$ be an $\fS$-prime of $R$. Then $\fp$ is finitary if and only if $V(\fp) \subset \fX$ consists of finitary points.
\end{proposition}

\begin{proof}
Replacing $A$ with $A/(A \cap \fp)$, we can assume that $A$ is a domain and $A \cap \fp=0$. If $\fp$ is not finitary then $\fp=0$, and so $V(\fp)=\fX$ contains non-finitary points. Suppose now that $\fp$ is finitary. Then $\fp \ne 0$ and so by Proposition~\ref{prop:disc}, $\fp$ contains $c \cdot \Delta_n$ for some $n$ and non-zero $c \in A$. Since $\fp$ is $\fS$-prime and $c$ is an $\fS$-invariant element not in $\fp$, we have $\Delta_n \in \fp$. Thus $V(\fp) \subset V(\fq)$, where $\fq$ is the $\fS$-ideal generated by $\Delta_n$. It is clear that every point in $V(\fq)$ is finitary, and so the result follows.
\end{proof}

We now investigate $V(\fp)$ more closely when $\fp$ is finitary. Let $x$ be a finitary $K$-point of $\fX$. Let $\{y_{\alpha}\}_{\alpha \in \cI}$ be the finite set of values that the coordinates of $x$ assume, and let $U_{\alpha} \subset [\infty]$ be the set of indices $i$ such that $x_i=y_{\alpha}$. Let $\lambda$ be the $\infty$-composition on $\cI$ defined by $\lambda_{\alpha}=\# U_{\alpha}$. We define the \defi{type} of $x$ to be the (isomorphism class of) $\lambda$. Again, this definition is invariant under field extension, and so can be defined for scheme-theoretic points of $\fX$.

Let $\lambda$ and $\mu$ be $\infty$-compositions on index sets $\cI$ and $\cK$. We define $\mu \preceq \lambda$ if there exists a function $\phi \colon \cI \to \cK$ such that $\mu_{\beta} \le \sum_{\alpha \in \phi^{-1}(\beta)} \lambda_{\alpha}$ for all $\beta \in \cK$. In words, $\mu \preceq \lambda$ means that $\mu$ is obtained from $\lambda$ by merging parts and reducing parts.

\begin{proposition} \label{prop:type-char}
Let $\fp=\fP(\lambda;Z)$ be a finitary radical $\fS$-prime of $R$ and let $\Lambda$ be the set of $\infty$-partitions such that $V(\fp)$ contains a point of type $\lambda$. Then $\lambda$ is the unique maximal element of $\Lambda$, that is $\lambda \in \Lambda$ and every $\mu \in \Lambda$ satisfies $\mu \preceq \lambda$.
\end{proposition}

\begin{proof}
It is clear from the construction of $\fp$ that $V(\fp)$ contains a point of type $\lambda$. The fact that $\lambda$ is the unique maximal element of $\Lambda$ follows from \cite[Proposition~8.4]{svar}.
\end{proof}

To close this section, we observe that type is well-behaved along certain contractions:

\begin{proposition} \label{prop:type-contract}
Let $A \to A'$ be a ring homomorphism, let $R'=A' \otimes_A R$, and let $\fp$ be a finitary $\fS$-prime of $R'$ of type $\lambda$. Then $\fp^c$ is a finitary $\fS$-prime of $R$ of type $\lambda$.
\end{proposition}

\begin{proof}
Set $\fX'=\Spec(R')$.  Let $\fX'_{[\lambda]} \subset \fX'$ be the set of points of type $\lambda$ and let $\fX'_{\lambda}$ be its Zariski closure. The equations defining $\fX'_{\lambda}$ are independent of the coefficient ring $A'$ and have integral coefficients \cite[Theorem~8.6]{svar}; suppose that the $\fS$-orbits of $f_1, \ldots, f_r \in \bZ[\xi_i]$ define $\fX'_{\lambda}$. Since $V(\fp) \subset \fX'_{\lambda}$, we see that $\fp$ contains a power of each $f_i$. Thus $\fp^c$ contains a power of each $f_i$ as well, and so $V(\fp^c) \subset \fX_{\lambda}$. Thus every point of $V(\fp^c)$ is finitary of type $\preceq \lambda$, which shows that $\fp^c$ is finitary. Any point of $V(\fp)$ of type $\lambda$ defines a point of $V(\fp^c)$ of type $\lambda$. Thus $V(\fp^c)$ has a point of type $\lambda$, and so the type of $\fp^c$ is $\lambda$.
\end{proof}

\section{Some contractions} \label{s:contract}

In this section we compute certain ideal contractions. These computations will play a key role in the proof of the classification theorem in the subsequent section.

\subsection{Statement of results}

We fix the following notation for \S \ref{s:contract}:
\begin{itemize}
\item $K$ is a field.
\item $R_n=K[\xi_1, \ldots, \xi_n]$ and $S_n=K[t,\epsilon_1, \ldots, \epsilon_n]$.
\item $\iota_n \colon R_n \to S_n$ is the $K$-algebra homomorphism defined by $\iota_n(\xi_i)=t+\epsilon_i$.
\item $q$ is a positive integer.
\item $\fb_{n,q}$ is the ideal of $S_n$ generated by $\epsilon_i^q$ for $1 \le i \le n$.
\end{itemize}
The following is the main result of this section:

\begin{theorem} \label{thm:contract}
We have the following:
\begin{enumerate}
\item If $K$ has characteristic~0 then $\fb_{n,q}^c$ is generated by $(\xi_i-\xi_j)^{2q-1}$ for $1 \le i,j \le n$.
\item If $K$ has characteristic $p>0$ and $q$ is a power of $p$ then $\fb_{n,q}^c$ is generated by $(\xi_i-\xi_j)^q$ for $1 \le i,j \le n$.
\end{enumerate}
\end{theorem}

\begin{remark} \label{rmk:contract}
The obvious analog of Theorem~\ref{thm:contract} for $n=\infty$ follows from the theorem, as $\fb_{\infty,q}=\bigcup_{n \ge 0} \fb_{n,q}$ and contraction commutes with directed unions. From this, see that $\fP((\infty),(q))$ is generated by $(\xi_i-\xi_j)^{2q-1}$ when $A$ is a field of characteristic~0, as stated in Example~\ref{ex:P-inf-n}.
\end{remark}

\begin{remark} 
In fact, the generators of $\fb_{n,q}^c$ given in Theorem~\ref{thm:contract} form a Gr\"obner basis. Our proof of this result is quite involved, and not included in this paper.
\end{remark}

\subsection{Proof of Theorem~\ref{thm:contract}(a)} \label{ss:contract-a}

We assume throughout \S \ref{ss:contract-a} that $K$ has characteristic~0. Let $q=(q_1, \ldots, q_n)$ be a tuple of positive integers. Define $\fb_n(q)$ to be the ideal of $S_n$ generated by $\epsilon_i^{q_i}$ for $1 \le i \le n$, and define $\fa_n(q)$ to be the ideal of $R_n$ generated by $(\xi_i-\xi_j)^{q_i+q_j-1}$ for $1 \le i,j \le n$. We prove:

\begin{proposition} \label{prop:contract-a}
The contraction of $\fb_n(q)$ to $R_n$ is $\fa_n(q)$.
\end{proposition}

Taking $q_1=\cdots=q_n=q$ yields Theorem~\ref{thm:contract}(a). We require several lemmas before proving the proposition. We begin with the following simple observation:

\begin{lemma} \label{lem:contract-a-1}
The ideal $\fa_n(q)$ is contained in the contraction of $\fb_n(q)$.
\end{lemma}

\begin{proof}
We have $\iota(\xi_i-\xi_j)=\epsilon_i-\epsilon_j$. Letting $N=q_i+q_j-1$, we thus find
\begin{displaymath}
\iota((\xi_i-\xi_j)^N) = \sum_{s=0}^N \binom{N}{s} \epsilon_i^s \epsilon_j^{N-s}.
\end{displaymath}
In each term in the sum, we have either $s \ge q_i$ or $N-s \ge q_j$, and so either $\epsilon_i^s$ or $\epsilon_j^{N-s}$ belongs to $\fb_n(q)$. We thus see that $(\xi_i-\xi_j)^N$ belongs to the contraction of $\fb_n(q)$, as required.
\end{proof}

Let $\fa_n(1)$ be the ideal $\fa_n(q)$ with $q_1=\cdots=q_n=1$; explicitly, $\fa_n(1)$ is generated by $\xi_i-\xi_j$ for $1 \le i, j \le n$. Similarly define $\fb_n(1)$; it is generated by the elements $\epsilon_i$ for $1 \le i \le n$.

\begin{lemma} \label{lem:contract-a-2}
The proposition is true for $q=1$: that is, the contraction of $\fb_n(1)$ is $\fa_n(1)$.
\end{lemma}

\begin{proof}
We have $\fa_n(1) \subset \fb_n(1)^c$ by Lemma~\ref{lem:contract-a-1} (or direct observation). We prove the reverse inclusion. Thus suppose that $f \in \fb_n(1)^c$. Write $f=g(\xi_1)+\sum_{i=2}^n h_i (\xi_i-\xi_1)$ where $g$ is a univariate polynomial and $h_i \in R_n$. The ideal $\fb_n(1)^c$ is the kernel of the composition
\begin{displaymath}
R_n \to S_n \to S_n/\fb_n(1) \cong K[t],
\end{displaymath}
where the final isomorphism maps each $\epsilon_i$ to~0. Since $f$ belongs to $\fb_n(1)^c$, it maps to~0 in $K[t]$. On the other hand, its image in $K[t]$ is $g(t)$. Thus $g=0$, and so $f \in \fa_n(1)$.
\end{proof}

Let $\partial_i \colon R_n \to R_n$ denote partial derivative with respect to $\xi_i$. We have the following useful characterization of the contraction of $\fb_n(q)$:

\begin{lemma} \label{lem:contract-a-3}
Let $f \in R_n$ be given. Then $f$ belongs to the contraction of $\fb_n(q)$ if and only if $\partial_1^{m_1} \cdots \partial_n^{m_n} f$ belongs to $\fa_n(1)$ for all $0 \le m_i < q_i$.
\end{lemma}

\begin{proof}
We have
\begin{displaymath}
\iota_n(f) = f(t+\epsilon_1, \ldots, t+\epsilon_n) = \sum_{m_1,\ldots,m_n \ge 0} \frac{1}{m_1! \cdots m_n!} (\partial_1^{m_1} \cdots \partial_n^{m_n} f)(t, \ldots, t) \epsilon_1^{m_1} \cdots \epsilon_n^{m_n},
\end{displaymath}
where in the second step we took the Taylor expansion of $f$. This belongs to $\fb_n(q)$ if and only if the coefficient of $\epsilon_1^{m_1} \cdots \epsilon_n^{m_n}$ vanishes whenever $0 \le m_i <q_i$ for all $i$. We thus require $(\partial_1^{m_1} \cdots \partial_n^{m_n} f)(t, \ldots, t)=0$ for such $m$'s, which is exactly the condition that $\partial_1^{m_1} \cdots \partial_n^{m_n} f$ belongs to $\fb_n(1)^c$, which equals $\fa_n(1)$ by Lemma~\ref{lem:contract-a-2}.
\end{proof}

\begin{proof}[Proof of Proposition~\ref{prop:contract-a}]
We proceed by induction on $\vert q \vert=\sum_{i=1}^n (q_i-1)$. If $\vert q \vert=0$ then $q=1$, and the result follows from Lemma~\ref{lem:contract-a-2}; this is the base case of the induction. Now suppose that $\vert q \vert>0$. Applying a permutation if necessary, we assume that $q_n \ge 2$. Since $\fa_n(q) \subset \fb_n(q)^c$ by Lemma~\ref{lem:contract-a-1}, it is enough to prove the reverse inclusion. Thus let $f \in \fb_n(q)^c$ be given. We show that $f \in \fa_n(q)$.

Let $q'=(q_1, \ldots, q_{n-1}, q_n-1)$. By Lemma~\ref{lem:contract-a-3}, we see that $\partial_n f$ belongs to $\fb_n(q')^c$. Since $\vert q' \vert<\vert q \vert$, we have $\fb_n(q')^c=\fa_n(q')$ by the inductive hypothesis. We therefore have an expression
\begin{displaymath}
\partial_n f = \sum_{1 \le i,j <n} a_{i,j} (\xi_i-\xi_j)^{q_i+q_j-1} + \sum_{1 \le i <n} b_i (\xi_i-\xi_n)^{q_i+q_n-2}
\end{displaymath}
where $a_i,b_i \in R_n$. Write $b_i=\sum_{j \ge 0} b_{i,j} (\xi_i-\xi_n)^j$ where $b_{i,j} \in R_{n-1}$. Thus
\begin{displaymath}
\partial_n f = \sum_{1 \le i,j <n} a_{i,j} (\xi_i-\xi_j)^{q_i+q_j-1} + \sum_{1 \le i <n} \sum_{0 \le j} b_{i,j} (\xi_i-\xi_n)^{q_i+q_n+j-2}.
\end{displaymath}
Let $\tilde{a}_{i,j} \in R_n$ be an anti-derivative of $a_{i,j}$ with respect to $\xi_n$, i.e., an element satisfying $\partial_n \tilde{a}_{i,j}=a_{i,j}$. Put
\begin{displaymath}
g = \sum_{1 \le i,j < n} \tilde{a}_{i,j} (\xi_i-\xi_j)^{q_i+q_j-1} + \sum_{1 \le i <n} \sum_{0 \le j} b_{i,j} \frac{(\xi_i-\xi_n)^{q_i+q_n+j-1}}{q_i+q_n+j-1}.
\end{displaymath}
Then $g$ belongs to $\fa_n(q)$ and $\partial_n g=\partial_n f$. We thus see that $h=f-g$ belongs to $\fb_n(q)^c$ (since $\fa_n(q) \subset \fb_n(q)^c$ by Lemma~\ref{lem:contract-a-1}) and $\partial_nh=0$, i.e., $h$ belongs to $R_{n-1}$.

Now, we have a commutative diagram
\begin{displaymath}
\xymatrix@C=4em{ R_{n-1} \ar[r]^{\iota_{n-1}} \ar[d] & S_{n-1} \ar[d] \\
R_n \ar[r]^{\iota_n} & S_n }
\end{displaymath}
where the vertical maps are the natural inclusions. We have seen that $h \in R_{n-1}$ belongs to the contraction of $\fb_n(q)$ to $R_{n-1}$. We observed this by first contracting to $R_n$ and then to $R_{n-1}$, but since the diagram commutes, we can first contract to $S_{n-1}$ and then to $R_{n-1}$. The contraction of $\fb_n(q)$ to $S_{n-1}$ is easily seen to be $\fb_{n-1}(q'')$ where $q''=(q_1, \ldots, q_{n-1})$. We thus see that $h$ belongs to $\fb_{n-1}(q'')^c$. Since $q_n \ge 2$, we have $\vert q'' \vert<\vert q \vert$, and so $\fb_{n-1}(q'')^c=\fa_{n-1}(q'')$ by the inductive hypothesis. Since $\fa_{n-1}(q'')$ is clearly contained in $\fa_n(q)$ as subsets of $R_n$, we see that $h\in \fa_n(q)$. Thus $f=g+h$ belongs to $\fa_n(q)$, which completes the proof.
\end{proof}

\subsection{Proof of Theorem~\ref{thm:contract}(b)} \label{ss:contract-b}

We now assume that $K$ has characteristic $p>0$ and that $q$ is a power of $p$. Let $\phi_n \colon R_n \to R_n$ be the $K$-algebra homomorphism that raises each variable to the $q$th power, and similarly define $\psi_n \colon S_n \to S_n$. We begin with a lemma:

\begin{lemma} \label{lem:contract-b}
Consider the commutative diagram
\begin{displaymath}
\xymatrix@C=4em{
R_n \ar[d]_{\phi_n} \ar[r]^{\iota_n} & S_n \ar[d]^{\psi_n} \\
R_n \ar[r]^{\iota_n} & S_n }
\end{displaymath}
and let $\fc$ be an ideal of $S_n$. Then $(\fc^e)^c=(\fc^c)^e$.
\end{lemma}

\begin{proof}
We first note that the above diagram does indeed commute, because $q$ is a power of $p$. Now, consider the following commutative diagram
\begin{displaymath}
\xymatrix@C=4em{
R_n \ar[d]^{\phi_n} \ar[r]^{\iota_n} & S_n \ar[d]^{\alpha} \ar@{=}[r] & S_n \ar[d]^{\psi_n} \\
R_n \ar[r]^-{\beta} & R_n \otimes_{R_n,\phi_n,\iota_n} S_n \ar[r]^-{\gamma} & S_n }
\end{displaymath}
where $\alpha=1 \otimes \id$, $\beta=\id \otimes 1$, and $\gamma(a \otimes b)=\iota_n(a) \psi_n(b)$. If we regard $S_n$ as an $R_n$-algebra via $\iota_n$, then the natural map $R_n[t] \to S_n$ is an isomorphism. With this identification, we have $R_n \otimes_{R_n} S_n \cong S_n$, and the map $\gamma \colon S_n \to S_n$ is the $R_n$-linear map taking $t$ to $t^p$. In particular, we see that $\gamma$ is faithfully flat. Now, we have
\begin{displaymath}
(\fc^{e,\psi_n})^{c,\gamma}=((\fc^{e,\alpha})^{e,\gamma})^{c,\gamma}=\fc^{e,\alpha}
\end{displaymath}
since extension along $\gamma$ followed by contraction along $\gamma$ is the identity, as $\gamma$ is faithfully flat. We also have
\begin{displaymath}
(\fc^{e,\alpha})^{c,\beta}=(\fc^{c,\iota_n})^{e,\phi_n}
\end{displaymath}
by Lemma~\ref{lemma:flat-ec}; note that $\phi_n$ is flat. Combining the above two identities and using that $\gamma \circ \beta=\iota_n$ yields the stated formula.
\end{proof}

\begin{proof}[Proof of Theorem~\ref{thm:contract}(b)]
Apply Lemma~\ref{lem:contract-b} with $\fc=\fb_{n,1}$. We have $\fb_{n,1}^e=\fb_{n,q}$ and $\fa_{n,1}^e=\fa_{n,q}$ by direct inspection (for the latter we use the identity $\xi_i^q-\xi_j^q=(\xi_i-\xi_j)^q$, which is valid because $q$ is a power of $p$), and $\fb_{n,1}^c=\fa_{n,1}$ by the analog of Lemma~\ref{lem:contract-a-2}. Combining these identities, we find $\fb_{n,q}^c=\fa_{n,q}$, as required.
\end{proof}

\section{Classification of $\fS$-primes} \label{s:class}

In this section, we prove the classification theorem (Theorem~\ref{mainthm}), which is the first principal theorem of this paper.

\subsection{Statement of results} \label{s:class-overview}

We begin by stating the classification theorem in its most general form:

\begin{theorem} \label{thm:class}
Let $\fp$ be a finitary $\fS$-prime of $R$. Then $\fp=\fP(\lambda,e;Z)$ for some data $(\lambda,e;Z)$ consisting of an $\infty$-composition $\lambda$ on an index set $\cI$, a $\lambda$-reduced weighting $e$ on $\cI$, and an irreducible closed subset $Z$ of $\bU^{\cI}$. Moreover, the data $(\lambda,e;Z)$ is unique up to isomorphism.
\end{theorem}

The proof of this theorem will take the entirety of \S \ref{s:class}. We deduce the theorem from three auxiliary theorems, which we now describe.

Fix an $\infty$-composition $\lambda$ on $\cI$ and a corresponding partition $\cU=\{\cU_{\alpha}\}_{\alpha \in \cI}$ of $[\infty]$. We use notation as in \S \ref{ss:setup}. We constructed the $\fS$-prime $\fP(\lambda,e;Z)$ as the $\fS$-contraction to $R$ of the $\fY_{\cU}$-prime $\fQ_{\cU}(\lambda,e;Z)$ of $\tilde{S}_{\cU}$ along the composition $R \to \tilde{R}_{\cU} \to \tilde{S}_{\cU}/\tilde{J}_{\cU,n}$ for $n \gg 0$. Of course, to form this $\fS$-contraction, we can first form the ordinary contraction to $\tilde{R}_{\cU}$, and then perform the $\fS$-contraction of the resulting ideal. We thus see that $\fP(\lambda,e;Z)$ is the $\fS$-contraction of a $\fY_{\cU}$-prime of $\tilde{R}_{\cU}$ (and, in fact, one containing $\tilde{I}_{\cU}$ in its radical). As a first step in proving Theorem~\ref{thm:class}, we show that every $\fS$-prime of type $\lambda$ can be realized in this manner:

\begin{theorem} \label{thm:class1}
Extension and $\fS$-contraction induce mutually inverse bijections
\begin{displaymath}
\{ \text{$\fS$-primes $\fp \subset R$ of type $\lambda$} \} \leftrightarrow \{ \text{$\rN \fY_{\cU}$-primes $\fq \subset \tilde{R}_{\cU}$ with $\tilde{I}_{\cU} \subset \rad(\fq)$} \}.
\end{displaymath}
\end{theorem}

We note that the difference between $\rN \fY_{\cU}$-prime and $\fY_{\cU}$-prime is not so significant due to Proposition~\ref{prop:prime-subgroup}. The proof of Theorem~\ref{thm:class1} relies on one piece of non-trivial input, namely, a theorem about the support of equivariant $R$-modules proven in \cite{svar}. Given the above theorem, our task is reduced to classifying the primes appearing on the right side of the correspondence. This is addressed by our second auxiliary theorem:

\begin{theorem} \label{thm:class2}
Contraction induces a bijection
\begin{displaymath}
\{ \text{$\fY_{\cU}$-primes $\fp \subset \tilde{R}_{\cU}$ with $\tilde{I}_{\cU} \subset \rad(\fp)$} \} \leftarrow
\{ \text{$\fY_{\cU}$-primes $\fq \subset \tilde{S}_{\cU}$ with $\tilde{J}_{\cU} \subset \rad(\fq)$} \}
\end{displaymath}
\end{theorem}

The inverse map can be described explicitly; see Remark~\ref{rmk:class-2}. The proof of this theorem also relies on one piece of non-trivial input, namely, the contraction calculations of \S \ref{s:contract}. Given Theorem~\ref{thm:class2}, our task is once again reduced to classifying the primes appearing on the right side of the correspondence. Our third auxiliary theorem accomplishes this:

\begin{theorem} \label{thm:class3}
Let $\fq$ be a $\fY_{\cU}$-prime of $\tilde{S}_{\cU}$ with $\tilde{J}_{\cU} \subset \rad(\fq)$. Then $\fq=\fQ_{\cU}(\lambda,e;Z)$ for some reduced weighting $e$ and irreducible closed subsets $Z \subset \bU^{\cI}$.
\end{theorem}

The proof of this theorem is elementary, and accomplished through a direct analysis. The proof of Theorem~\ref{thm:class} is divided into three main steps (corresponding to the above three theorems), and one final short subsection tying up the loose ends.

\subsection{Step 1a}

In this subsection, we show that $(\fq^{\fS c})^e=\fq$ for $\fq$ as in Theorem~\ref{thm:class1}. This should be thought of as the ``easy'' half of Theorem~\ref{thm:class1}: $\rN\fY_{\cU}$-primes $\fq$ as in the theorem are extremely constrained, and so, in principle, should be easier to understand than the primes $\fp$ on the other side of the correspondence.

Before getting into the details of the proof, let us illustrate the essential idea in a simple case. Suppose that $\fq$ is as in the theorem. The $\fS$-contraction of $\fq$ is defined as an infinite intersection, so the main concern is that it might be too small. Here is why this is not the case. Let $f$ belong to the ordinary contraction of $\fq$; note that $\tilde{R}_{\cU}$ is a localization of $R$, so the ordinary contraction has plenty of elements. Suppose for simplicity that $\cU$ has two parts $\cU_1$ and $\cU_2$, that $1 \in \cU_1$ and $2 \in \cU_2$, and that $f$ only uses the variables $\xi_1$ and $\xi_2$. Let $s=\xi_1-\xi_2$ (a unit of $\tilde{R}_{\cU}$), and let $k$ be such that $\tilde{I}_{\cU, k} \subset \fq$. We claim that $s^kf$ belongs to the $\fS$-contraction of $\fq$, i.e., $\sigma(s^kf) \in \fq^c$ for all $\sigma \in \fS$. There are two cases. First suppose that $\sigma(1)$ and $\sigma(2)$ belong to different parts of $\cU$. Then we can find an element $\tau$ of $\rN \fY_{\cU}$ such that $\sigma(1)=\tau(1)$ and $\sigma(2)=\tau(2)$. Since $\fq$ is stable by $\rN \fY_{\cU}$, we find that $\sigma(s^kf)=\tau(s^kf)$ belongs to $\fq^c$. Now suppose that $\sigma(1)$ and $\sigma(2)$ belong to the same part of $\cU$. Then $\sigma(s^k) \in I_{\cU, k}$, and so $\sigma(s^kf) \in \fq^c$. This establishes the claim. We thus have a good supply of elements in the $\fS$-contraction of $\fq$.

We now begin the work in earnest. We first introduce some notation and terminology. Let $N$ be a positive integer. We say that $\sigma \in \fS$ is \defi{$N$-good} if for every $1 \le i,j \le N$ with $\pi(i) \ne \pi(j)$ we have $\pi(\sigma(i)) \ne \pi(\sigma(j))$. We say that $N$ is \defi{sufficiently large relative to $\cU$}, denoted $N \gg \cU$, if $[N]$ contains every finite part of $\cU$, and for every infinite part $U_{\alpha}$ of $\cU$ the intersection $[N] \cap U_{\alpha}$ is non-empty and strictly larger (in cardinality) than every finite part of $\cU$. We put $s_N=\prod_{1 \le i,j \le N, \pi(i) \ne \pi(j)} (\xi_i-\xi_j)$, which is a unit of $R_{\cU}$. For $f \in R$, we write $N \gg f$ if every variable $\xi_i$ appearing in $f$ has $i<N$. The following is the key lemma:
 
\begin{lemma} \label{lem:Ngood}
Let $N \gg \cU$ and let $\sigma \in \fS$.
\begin{enumerate}
\item Suppose $\sigma$ is $N$-good. Then there exists $\tau \in \rN \fY_{\cU}$ such that $\sigma(i)=\tau(i)$ for all $1 \le i \le N$.
\item Suppose $\sigma$ is not $N$-good. Then $\sigma(s_N^k) \in I_{\cU, k}$.
\end{enumerate}
\end{lemma}

\begin{proof}
For $\alpha \in \cI$, let $X_{\alpha}$ be the set of indices $\beta \in \cI$ such that $\sigma$ maps some element of $[N] \cap \cU_{\alpha}$ into $\cU_{\beta}$. Since $[N] \cap \cU_{\alpha}$ is non-empty, it follows that $X_{\alpha}$ is non-empty. The $N$-goodness of $\sigma$ exactly means that $X_{\alpha}$ and $X_{\beta}$ are disjoint for $\alpha \ne \beta$. Since $\cI$ is finite, it follows that each $X_{\alpha}$ is a singleton. We thus have a unique permutation $\eta \colon \cI \to \cI$ such that $X_{\alpha}=\{\eta(\alpha)\}$.

Since $\sigma$ maps $[N] \cap \cU_{\alpha}$ injectively into $\cU_{\eta(\alpha)}$, we see that $\#([N] \cap \cU_{\alpha}) \le \# \cU_{\eta(\alpha)}$ for all $\alpha$. If $\cU_{\alpha}$ is finite then $[N]$ contains $\cU_{\alpha}$, and so we see $\# \cU_{\alpha} \le \# \cU_{\eta(\alpha)}$. If $\cU_{\alpha}$ is infinite then $[N] \cap \cU_{\alpha}$ is bigger than all the finite pieces of $\cU$, and so $\cU_{\eta(\alpha)}$ must also be infinite. We thus see that $\# \cU_{\alpha} \le \# \cU_{\eta(\alpha)}$ in all cases. Since $\cI$ is finite, it follows that $\# \cU_{\alpha}=\# \cU_{\eta(\alpha)}$ for all $\alpha$.

For $\alpha \in \cI$, define a bijection $\tau_{\alpha} \colon \cU_{\alpha} \to \cU_{\eta(\alpha)}$ as follows. If $\cU_{\alpha}$ is finite, let $\tau_{\alpha}$ be the restriction of $\sigma$ to $\cU_{\alpha}$, which is indeed such a bijection. If $\cU_{\alpha}$ is infinite, then $\sigma$ induces an injection $[N] \cap \cU_{\alpha} \to \cU_{\eta(\alpha)}$, and we let $\tau_{\alpha}$ be any bijective extension of this map. Let $\tau \in \fS$ be the permutation restricting to $\tau_{\alpha}$ on $\cU_{\alpha}$. Then $\tau \in \rN \fY_{\cU}$ and $\sigma(i)=\tau(i)$ for $1 \le i \le N$, as required.

(b) We have
\begin{displaymath}
\sigma(s^k_N) = \prod_{1 \le i,j \le N, \pi(i) \ne \pi(j)} (\xi_{\sigma(i)}-\xi_{\sigma(j)})^k.
\end{displaymath}
Since $\sigma$ is not $N$-good, there exists $1 \le i,j \le N$ with $\pi(i)\ne \pi(j)$ but $\pi(\sigma(i))=\pi(\sigma(j))$. The $(i,j)$ factor in the above product then belongs to $I_{\cU, k}$, which completes the proof.
\end{proof}

We now reach the main result of this section:

\begin{proposition} \label{prop:corr-1}
Let $\fq$ be an $\rN \fY_{\cU}$-ideal of $\tilde{R}_{\cU}$ such that $\tilde{I}_{\cU} \subset \rad(\fq)$.
\begin{enumerate}
\item We have $\fq=(\fq^{\fS c})^e$.
\item Suppose that $\fq^{\fS c}$ is $\fS$-prime. Then $\fq$ is $\rN \fY_{\cU}$-prime.
\end{enumerate}
\end{proposition}

\begin{proof}
Let $k$ be such that $\tilde{I}_{\cU, k} \subset \fq$.

(a) It is clear that $(\fq^{\fS c})^e \subset \fq$. Now suppose that $f \in \fq^c$. Let $N \gg f,\cU$. We claim that $s_N^k f \in \fq^{\fS c}$. We must show that $\sigma(s_N^k f) \in \fq^c$ for any $\sigma \in \fS$. Thus let $\sigma$ be given. If $\sigma$ is $N$-good then there exists $\tau \in \rN \fY_{\cU}$ such that $\sigma(i)=\tau(i)$ for $1 \le i \le N$, and so $\sigma(s_N^k f)=\tau(s_N^k f)$ belongs to $\fq^c$ since $\fq$ is $\rN \fY_{\cU}$-stable. Now suppose that $\sigma$ is not $N$-good. Then $\sigma(s_N^k) \in I_{\cU,k}$, and so $\sigma(s_N^k f) \in \fq^c$. This proves the claim. Now, since $s_N^k$ is a unit of $\tilde{R}_{\cU}$, it follows that $f \in (\fq^{\fS c})^e$. Thus $(\fq^{\fS c})^e$ contains $(\fq^c)^e=\fq$, and so we have $\fq=(\fq^{\fS c})^e$, as required.

(b) We first show that $\fq^c$ is $\rN \fY_{\cU}$-prime. Thus suppose that $f \cdot \sigma(g) \in \fq^c$ for all $\sigma \in \rN \fY_{\cU}$. Since $\fq^c$ is $\rN \fY_{\cU}$-stable, it follows that $\sigma(f) \cdot \tau(g) \in \fq^c$ for all $\sigma,\tau \in \rN \fY_{\cU}$. Let $N \gg \cU,f,g$. We claim that $\sigma(s_N^k f) \cdot \tau(s_N^k g) \in \fq^c$ for all $\sigma,\tau \in \fS$. If $\sigma$ is not $N$-good then $\sigma(s_N^k) \in \fq^c$, which gives the required containment; similarly if $\tau$ is not $N$-good. If both $\sigma$ and $\tau$ are $N$-good then we can find $\sigma',\tau' \in \rN \fY_{\cU}$ such that $\sigma(i)=\sigma'(i)$ and $\tau(i)=\tau'(i)$ for $1 \le i \le N$, and so $\sigma(s_N^k f) \tau(s_N^k g)=\sigma'(s_N^k f) \tau'(s_N^k g)$ belongs to $\fq^c$. This establishes the claim. It follows that $\sigma(s_N^k f) \cdot (s_N^k g) \in \fq^{\fS c}$ for all $\sigma \in \fS$. Since $\fq^{\fS c}$ is $\fS$-prime, it follows that $s_N^k f \in \fq^{\fS c}$ or $s_N^k g \in \fq^{\fS c}$; without loss of generality, assume the former. Thus $s_N^k f \in \fq^c$. Since $s_N^k$ is a unit of $\tilde{R}_{\cU}$, we see that $f \in \fq^c$. Thus $\fq^c$ is $\rN \fY_{\cU}$-prime.

We now show that $\fq$ itself is $\rN \fY_{\cU}$-prime. Thus suppose that $f \cdot \sigma(g) \in \fq$ for all $\sigma \in \rN \fY_{\cU}$, with $f,g \in \tilde{R}_{\cU}$. Write $f=f_0/s$ and $g=g_0/s'$ with $s,s'$ units of $R_{\cU}$ and $f_0,g_0 \in R$. We thus see that $f_0 \cdot \sigma(g_0) \in \fq^c$ for all $\sigma$. Thus, by the previous paragraph, $f_0 \in \fq^c$ or $g_0 \in \fq^c$, and so $f \in \fq$ or $g \in \fq$.
\end{proof}

\subsection{Step 1b} \label{ss:class-1b}

We now prove the other half of Theorem~\ref{thm:class1} (more or less); namely, that $\fp=(\fp^e)^{\fS c}$ for an $\fS$-prime $\fp$ of type $\lambda$. As indicated, this is the harder part of the theorem, since we have to handle arbitrary $\fS$-primes. To deal with this, we leverage our theory of radical $\fS$-primes developed in \cite{svar}. We review the relevant results now.

Recall that an $(R,\fS)$-module is a $\fS$-equivariant $R$-module $M$ such that the action of $\fS$ on $M$ is smooth. Let $M$ be such a module. Let $\supp(M)$ denote the support of $M$, i.e., the set of all (ordinary) primes $\fp$ of $R$ for which the localization $M_{\fp} \ne 0$. This is a subset of $\fX=\Spec(R)$. We have $\supp(M) \subset V(\ann{M})$. Ideally, these two sets would coincide under some reasonable finiteness hypothesis. This is not quite the case, but something close does hold.

Suppose that $A$ is noetherian. Let $\fZ$ be a closed $\fS$-subset of the finitary locus $\fX_{\fin}$. Then $\fZ=\fZ_1 \cup \cdots \cup \fZ_n$ where each $\fZ_i$ is $\fS$-irreducible (this relies on the fact that $A$ is noetherian). Suppose that $\fZ_i$ has type $\lambda_i$. We define $\fZ_i^{\top}$ to be the set of points of $\fZ_i$ of type $\lambda$. We define $\fZ^{\top}=\fZ_1^{\top} \cup \cdots \cup \fZ_n^{\top}$. We note that $\fZ^{\top}$ is Zariski dense in $\fZ$. The following result is \cite[Theorem~9.6]{svar}:

\begin{proposition} \label{prop:supp}
Suppose that $A$ is noetherian. Let $M$ be an $(R,\fS)$-module that is finitely $\fS$-generated (i.e., generated as an $R$-module by finitely many $\fS$-orbits). Suppose that the set $V(\ann(M))$ consists of finitary points. Then
\begin{displaymath}
V(\ann{M})^{\top} \subset \supp(M) \subset V(\ann{M}).
\end{displaymath}
\end{proposition}

With this tool in hand, we can now prove our main result:

\begin{proposition} \label{prop:corr-2}
Let $\fp$ be an $\fS$-prime of $R$ of type $\lambda$. Then $\fp=(\fp^e)^{\fS c}$.
\end{proposition}

\begin{proof}
First suppose that $A$ is noetherian. Let $\fp'=(\fp^e)^{\fS c}$. It is clear that $\fp \subset \fp' \subset \fp^e$. Let $M=\fp'/\fp$; this is an $(R,\fS)$-module and is finitely $\fS$-generated by Cohen's theorem (Theorem~\ref{thm:cohen}). Since $(\fp')^e=\fp^e$, we find $\tilde{R}_{\cU} \otimes_R M=0$. It follows that $\supp(M) \cap \Spec(\tilde{R}_{\cU})=\emptyset$. If $x$ is a point of $\fX$ of type $\lambda$ then $\fS^{\rm big} x$ meets $\Spec(\tilde{R}_{\cU})$, where $\fS^{\rm big}$ denotes the group of all permutations of $[\infty]$. Since $\supp(M)$ is $\fS^{\rm big}$-stable, it follows that it has no point of type $\lambda$.

Let $\fa=\ann(M)$, an $\fS$-ideal of $R$. We have $\fp \subset \fa$, and so $V(\fa) \subset V(\fp)$ consists of finitary points (see Proposition~\ref{prop:finitary-char}). Appealing to Proposition~\ref{prop:supp}, we find $V(\fa)^{\top} \subset \supp(M)$. In particular, $V(\fa)^{\top}$ does not contain any point of type $\lambda$.

We have $\fa \fp' \subset \fp$, so $\fa \subset \fp$ or $\fp' \subset \fp$ since $\fp$ is $\fS$-prime. As both opposite containments hold, we must have $\fa=\fp$ or $\fp'=\fp$. The former is impossible since $V(\fa)^{\top}$ has no point of type $\lambda$ but $V(\fp)^{\top}$ does, as $\fp$ has type $\lambda$ (see Proposition~\ref{prop:type-char}). Thus $\fp'=\fp$, which completes the proof in the noetherian case.

We now treat the general case. Write $A=\bigcup_{i \in I} A_i$, where $\{A_i\}_{i \in I}$ is a direct family of noetherian subrings of $A$ (e.g., one could take all finitely generated $\bZ$-subalgebras of $A$). Let $R_i=A_i[\xi_j]_{j \ge 1}$, so that $R$ is the directed union of the $R_i$. Let $\fp_i$ be the contraction of $\fp$ to $R_i$; this is an $\fS$-prime of $R_i$ of type $\lambda$ (Proposition~\ref{prop:type-contract}). We thus have
\begin{displaymath}
(\fp^e)^{\fS c}=\big( \bigcup_{i \in I} \fp_i^e \big)^{\fS c}=\bigcup_{i \in I} (\fp_i^e)^{\fS c} = \bigcup_{i \in I} \fp_i = \fp.
\end{displaymath}
In the first and last step, we used the equality $\fp=\bigcup_{i \in I} \fp_i$; in the second, we appealed to Proposition~\ref{prop:limit-contract}; and in the third, we used the noetherian case. The result follows.
\end{proof}

\subsection{Completion of Step~1}

We now finish the proof of Theorem~\ref{thm:class1}. Let $X$ be the set of $\fS$-primes $\fp$ of $R$ of type $\lambda$, and let $Y$ be the set of $\rN \fY_{\cU}$-primes $\fq$ of $\tilde{R}_{\cU}$ such that $\tilde{I}_{\cU} \subset \rad(\fq)$. We must show that extension and $\fS$-contraction induce mutually inverse bijections $X \leftrightarrow Y$.

Suppose $\fq \in Y$. Then $\fq^{\fS c}$ is an $\fS$-prime of $R$ (Proposition~\ref{prop:G-contract}). We claim that it has type $\lambda$. Let $\fq'=\rad(\fq)$. Then $\fq'$ is an $\rN \fY_{\cU}$-prime (Proposition~\ref{prop:rad-prime}) and radical, and contains $\tilde{I}_{\cU}$. Since $\tilde{R}_{\cU}/\tilde{I}_{\cU} \cong \tilde{S}_{\cU,0}$, we see that $\fq'$ corresponds to a $\Aut(\lambda)$-prime $\fc$ of $\tilde{S}_{\cU,0}$. The $\fS$-contraction of $\fq'$ to $R$ is the radical $\fS$-prime $\fP(\lambda; Z)$, where $Z$ is the $\Aut(\lambda)$-irreducible variety defined by $\fc$. Since this contraction coincides with $\rad(\fq^{\fS c})$, the claim follows. We thus see that $\fq^{\fS c} \in X$. Thus $\fS$-contraction does indeed define a map $\phi \colon Y \to X$.

Now suppose that $\fp \in X$. Then $\fp^e$ is certainly an $\rN \fY_{\cU}$-ideal of $\tilde{R}_{\cU}$. By Proposition~\ref{prop:corr-2}, we have $(\fp^e)^{\fS c}=\fp$, and so by Proposition~\ref{prop:corr-1}(b) we see that $\fp^e$ is $\rN \fY_{\cU}$-prime. Thus extension does indeed define a map $\psi \colon X \to Y$.

As we have just seen, Proposition~\ref{prop:corr-2} shows that $\phi \circ \psi=\id_X$. Proposition~\ref{prop:corr-1}(a) gives $\psi \circ \phi=\id_Y$. Thus $\phi$ and $\psi$ are mutually inverse, which completes the proof.

\subsection{Step 2} \label{ss:class2}

We require a few lemmas before proving the theorem.

%

\begin{lemma} \label{lem:class-2-2}
Suppose that $A$ contains a field.
\begin{enumerate}
\item If $\bQ \subset A$ then $\ker(R_{\cU} \to S_{\cU}/J_{\cU,n})=I_{\cU,2n-1}$.
\item If $\bF_p \subset A$ and $n$ is a power of $p$ then $\ker(R_{\cU} \to S_{\cU}/J_{\cU,n})=I_{\cU,n}$.
\end{enumerate}
\end{lemma}

\begin{proof}
First suppose that $A=K$ is a field. For $\alpha \in \cI$, let $R_{\alpha}=K[\xi_i]_{i \in \cU_{\alpha}}$ and $S_{\alpha}=K[t_{\alpha},\epsilon_i]_{i \in \cU_{\alpha}}$. Then $R_{\cU}=\bigotimes R_{\alpha}$ and $S_{\cU}=\bigotimes S_{\alpha}$. Let $I_{\alpha,n} \subset R_{\alpha}$ be the ideal generated by $(\xi_i-\xi_j)^n$ with $i,j \in \cU_{\alpha}$ and let $J_{\alpha,n} \subset S_{\alpha}$ be the ideal generated by $\epsilon_i^n$ with $i \in \cU_{\alpha}$. Then $I_{\cU,n}$ is generated by the $I_{\alpha,n}$'s and $J_{\cU,n}$ by the $J_{\alpha,n}$'s. Finally, $\iota$ decomposes as a tensor product of maps $\iota_{\alpha} \colon R_{\alpha} \to S_{\alpha}$. It thus suffices to show that $\iota_{\alpha}^{-1}(J_{\alpha,n})$ is either $I_{\alpha,2n-1}$ or $I_{\alpha,n}$ depending on which case we are in. This is exactly Theorem~\ref{thm:contract} (see also Remark~\ref{rmk:contract} for the case when $U_{\alpha}$ is infinite).

The general case follows from the field case by base change. Precisely, let $k$ be a subfield of $A$ and use primes to denote the analogous objects over $k$. Then, in case (a), we have shown that $I'_{\cU,2n-1}$ is the kernel of $R'_{\cU} \to S'_{\cU}/J'_{\cU,n}$. Applying the exact functor $A \otimes_k (-)$ yields the analogous statement without primes. Similarly for (b).
\end{proof}

\begin{lemma} \label{lem:class-2-3}
Suppose that $A$ contains a field.
\begin{enumerate}
\item If $\bQ \subset A$ then $\ker(\tilde{R}_{\cU} \to \tilde{S}_{\cU}/\tilde{J}_{\cU,n})=\tilde{I}_{\cU,2n-1}$.
\item If $\bF_p \subset A$ and $n$ is a power of $p$ then $\ker(\tilde{R}_{\cU} \to \tilde{S}_{\cU}/\tilde{J}_{\cU,n})=\tilde{I}_{\cU,n}$.
\end{enumerate}
\end{lemma}

\begin{proof}
Let $\Sigma \subset R_{\cU}$ be the multiplicative set generated by $\xi_i-\xi_j$ with $\pi(i) \ne \pi(j)$, so that $\tilde{R}_{\cU}$ is the localization of $R_{\cU}$ at $\Sigma$. Let $\Sigma' \subset S_{\cU}$ be the multiplicative set generated by $t_{\alpha}-t_{\beta}$ with $\alpha \ne \beta$, so that $\tilde{S}_{\cU}$ is by definition the localization of $S_{\cU}$ at $\Sigma'$. Modulo $J_{\cU,n}$, the two sets $\Sigma$ and $\Sigma'$ coincide up to nilpotents, i.e., if $x \in \Sigma$ then there exists $y \in \Sigma'$ such that $x-y$ is nilpotent, and conversely (see the proof of Proposition~\ref{prop:iota}). It follows that $\tilde{S}_{\cU}/\tilde{J}_{\cU,n}$ is the localization of $S_{\cU}/J_{\cU,n}$ at $\Sigma$. Thus the result follows from localizing Lemma~\ref{lem:class-2-2} at $\Sigma$.
\end{proof}

\begin{proof}[Proof of Theorem~\ref{thm:class2}]
We proceed in two cases.

\textit{\textbf{Case 1:} $A$ contains a field.}
If $k$ has characteristic~0, let $n$ be a positive integer and $m=2n-1$; if $k$ has positive characteristic~$p$, let $n$ be a power of $p$ and $m=n$. By Lemma~\ref{lem:class-2-3}, we have an injection of rings
\begin{equation} \label{eq:class2}
\tilde{R}_{\cU}/\tilde{I}_{\cU,m} \to \tilde{S}_{\cU}/\tilde{J}_{\cU,n}.
\end{equation}
For each $\alpha \in \cI$, choose $j(\alpha) \in U_{\alpha}$, and let $\fn$ be the ideal of $\tilde{S}_{\cU}/\tilde{J}_{\cU,n}$ generated by the $\epsilon_{j(\alpha)}$ with $\alpha \in \cI$. Since $\fn$ is generated by finitely many nilpotent elements, it is nilpotent.

We claim that
\begin{displaymath}
\tilde{S}_{\cU}/\tilde{J}_{\cU,n} = \tilde{R}_{\cU}/\tilde{I}_{\cU,m} + \fn.
\end{displaymath}
To see this, note that $T=\tilde{R}_{\cU}/\tilde{I}_{\cU,m} + \fn$ is an $A$-subalgebra of $\tilde{S}_{\cU}/\tilde{J}_{\cU,m}$. It therefore suffices to show that $T$ contains each $t_{\alpha}$ and each $\epsilon_i$. By definition, $T$ contains $\iota(\xi_i)=t_{\pi(i)}+\epsilon_i$ for all $i$ and $\epsilon_{j(\alpha)}$ for all $\alpha$. We thus see that $T$ contains $t_{\alpha}=\iota(\xi_{j(\alpha)})-\epsilon_{j(\alpha)}$ for $\alpha$. It therefore also contains $\epsilon_i=\iota(\xi_i)-t_{\pi(i)}$ for all $i$. This proves the claim.

From the above paragraph, we see that \eqref{eq:class2} is a nilpotent extension of rings in the sense of \S \ref{ss:nilp}. Thus, by Proposition~\ref{prop:nilp2}, we see that contraction gives a bijection
\begin{displaymath}
\{ \text{$\fY_{\cU}$-primes $\fp \subset \tilde{R}_{\cU}$ with $\tilde{I}_{\cU,m} \subset \fp$} \} \leftarrow
\{ \text{$\fY_{\cU}$-primes $\fq \subset \tilde{S}_{\cU}$ with $\tilde{J}_{\cU,n} \subset \fq$} \}.
\end{displaymath}
Taking the union over all $n$ yields the desired result.

\textit{\textbf{Case 2:} $A$ is arbitary.}
Let $X$ and $Y$ be the two sets of ideals in the theorem, so that we must show that extension and contraction give mutually inverse bijections $X \leftrightarrow Y$. For a prime number $p$, let $X_p$ be the set of primes $\fp \in X$ such that $\fp \cap \bZ=(p)$; similarly, let $X_0$ be the set of primes $\fp \in X$ such that $\fp \cap \bZ=(0)$. Similarly define $Y_p$ and $Y_0$. Then $X$ is the disjoint union of $X_0$ and the $X_p$, and similarly for $Y$.

Fix a prime $p>0$. Let $\ol{X}_p$ be and $\ol{Y}_p$ be the two sets of primes appearing in Theorem~\ref{thm:class2} for $A/(p)$. Then we have bijections $X_p \leftrightarrow \ol{X}_p$ and $Y_p \leftrightarrow \ol{Y}_p$ that are compatible with extension and contraction. Since extension and contraction give mutually inverse bijections $\ol{X}_p \leftrightarrow \ol{Y}_p$ by the first case, the same holds for $X_p \leftrightarrow Y_p$.

Now, let $\ol{X}_0$ and $\ol{Y}_0$ be the two sets of primes appearing in Theorem~\ref{thm:class2} for $A \otimes \bQ$. Then we have bijections $X_0 \leftrightarrow \ol{X}_0$ and $Y_0 \leftrightarrow \ol{Y}_0$ (by Proposition~\ref{prop:prime-loc}) that are compatible with extension and contraction. Since extension and contraction give mutually inverse bijections $\ol{X}_0 \leftrightarrow \ol{Y}_0$ by the first case, the same holds for $X_0 \leftrightarrow Y_0$. This completes the proof.
\end{proof}

\begin{remark} \label{rmk:class-2} 
The inverse to the bijection in Theorem~\ref{thm:class2} can be described as follows. Let $\fp$ be a $\fY_{\cU}$-prime of $\tilde{R}_{\cU}$ with $\tilde{I}_{\cU} \subset \rad(\fp)$. Let $m_0$ be such that $\tilde{I}_{\cU,m_0} \subset \fp$. If $\fp$ has residue characteristic~0, choose $m \ge m_0$ odd and put $m=2n-1$; if $\fp$ has residue characteristic $p>0$, choose $m \ge m_0$ a power of $p$ and put $m=n$. Let $\fq=\rad_{\fY_{\cU}}(\fp^e+\tilde{J}_{\cU,n})$. Then $\fq$ is $\fY_{\cU}$-prime, contains $\tilde{J}_{\cU}$ in its radical, and contracts to $\fp$. This follows from Proposition~\ref{prop:nilp2}. Thus $\fq$ corresponds to $\fp$ under the bijection in Theorem~\ref{thm:class2}.
\end{remark}

\subsection{Step 3}
We begin with a few lemmas. The following lemma is essentially a special case of the theorem, and contains the main points of the proof.

\begin{lemma} \label{lem:class-3-1}
Suppose that $A=K$ is a field. Let $\fp$ be an $\fS$-prime of $R$ such that $\rad(\fp)=\langle \xi_i \rangle_{i \ge 1}$. Let $n \ge 1$ be minimal such that $\xi_1^n \in \fp$. Then $\fp=\langle \xi_i^n \rangle_{i \ge 1}$.
\end{lemma}

\begin{proof}
Let $\fq=\langle \xi_i^n \rangle_{i \ge 1}$. We thus have $\fq \subset \fp$, and we must show $\fp=\fq$. If $n=1$ then $\fq$ is maximal, and so the result follows. Assume $n>1$ in what follows, and suppose by way of contradiction that $\fp \ne \fq$.

Let $h_r=(\xi_1 \cdots \xi_r)^{n-1}$. We claim that $\fp$ contains $h_r$ for some $r \ge 1$. Indeed, let $f$ be an element of $\fp$ not contained in $\fq$. We can assume that no monomial in $f$ belongs to $\fq$. Suppose that $f$ uses $\xi_1, \ldots, \xi_r$. Of the monomials occurring in $f$ with non-zero coefficient, let $\xi_1^{k_1} \cdots \xi_r^{k_r}$ be one that is minimal with respect to divisibility. Multiplying $f$ by $\xi_1^{n-k_1-1} \cdots \xi_r^{n-k_r-1}$, we see that $h_r$ appears in $f$ with non-zero coefficients and all other monomials in $f$ have some exponent $\ge n$ and thus belong to $\fq$. Hence $h_r \in \fp$, as claimed.

Suppose $r>1$ and $h_r \in \fp$. We claim that $h_{r-1} \cdot \sigma(h_1)$ belongs to $\fp$ for all $\sigma \in \fS$. Indeed, if $\sigma(1)$ is distinct from $1, \ldots, r-1$ then $h_{r-1} \cdot \sigma(h_1)$ belongs to the $\fS$-oribt of $h_r$, which belongs to $\fp$. Otherwise, $\sigma(1)$ coincides with one of $1, \ldots, r-1$, and $h_{r-1} \cdot \sigma(h_1)$ is then divisible by $\xi_{\sigma(1)}^{2n-2}$, which belongs to $\fp$ as $2n-2 \ge n$; thus $h_{r-1} \cdot \sigma(h_1) \in \fp$. This proves the claim. Since $\fp$ is $\fS$-prime, we conclude that $\fp$ contains $h_{r-1}$ or $h_1$.

Continuing in this manner, we find that $\fp$ contains $h_1=\xi_1^{n-1}$, contradicting the definition of $n$. We conclude that $\fp=\fq$, as desired.
\end{proof}

\begin{lemma} \label{lem:class-3-2}
Let $K$ be a field and let $\fp$ be an $\fS_n$-prime of $K[\xi_1, \ldots, \xi_n]$ such that $\rad(\fp)=\langle \xi_i \rangle_{1 \le i \le n}$. Then $\fp=\langle \xi_i \rangle_{1 \le i \le n}$.
\end{lemma}

\begin{proof}
Since $\fS_n$ is finite, any $\fS_n$-prime is radical (Proposition~\ref{prop:prime-subgroup}), so the result follows.
\end{proof}

\begin{lemma} \label{lem:class-3-3}
Let $A=K$ be a field. Let $\fp$ be a $\fY_{\cU}$-prime of $R_{\cU}$ such that $\rad(\fp)=\langle \xi_i \rangle_{i \ge 1}$. For $\alpha \in \cU$, pick $i \in U_{\alpha}$ and let $e(\alpha) \ge 1$ be minimal such that $\xi_i^{e(\alpha)} \in \fp$; note that this is independent of $i$. Then $e$ is a reduced weighting on $\cI$ and we have $\fp=\langle \xi_i^{e(\pi(i))} \rangle_{i \ge 1}$.
\end{lemma}

\begin{proof}
Let $\fq=\langle \xi_i^{e(\pi(i))} \rangle_{i \ge 1}$. Then $\fq \subset \fp$ and we must prove equality. For $\alpha \in \cI$, let $R_{\alpha}=K[\xi_i]_{i \in U_{\alpha}}$, so that $R=\bigotimes R_{\alpha}$ (as in the proof of Lemma~ \ref{lem:class-2-2}). Let $G=\fY_{\cU}$ and $G_{\alpha}=\fS_{\cU_{\alpha}}$, so that $G=\prod_{\alpha \in \cI} G_{\alpha}$. We can regard $G$ as acting on $R_{\alpha}$ by letting $G_{\beta}$ act trivially for $\beta \ne \alpha$; in this way, the inclusion $R_{\alpha} \to R$ is $G$-equivariant. It follows that the contraction $\fp_{\alpha}$ of $\fp$ to $R_{\alpha}$ is a $G_{\alpha}$-prime. Now, if $U_{\alpha}$ is infinite then $(G_{\alpha}, R_{\alpha})$ is isomorphic to $(\fS, R)$. Thus, by Lemma~\ref{lem:class-3-1}, we find that $\fp_{\alpha}=\langle \xi_i^{e(\alpha)} \rangle_{i \in U_{\alpha}}$. If $U_{\alpha}$ is finite of cardinality $n$ then $(G_{\alpha}, R_{\alpha})$ is isomorphic to $(\fS_n, K[\xi_1, \ldots, \xi_n])$. Thus, by Lemma~\ref{lem:class-3-2}, we find $\fp_{\alpha} = \langle \xi_i \rangle_{i \in U_{\alpha}}$. In particular, we find $e(\alpha)=1$, and so $e$ is reduced. We thus see that for $\alpha \in \cI$, the contractions of $\fp$ and $\fq$ to $R_{\alpha}$ coincide.

Suppose that $m$ is a monomial belonging to $\fp$. Write $m=\prod_{\alpha \in \cI} m_{\alpha}$ where $m_{\alpha}$ is a monomial in $R_{\alpha}$. We claim that $m_{\alpha} \in \fp$ for some $\alpha$. Pick $\beta \in \cI$ and let $m'=\prod_{\alpha \ne \beta} m_{\alpha}$. Given $\sigma=(\sigma_{\alpha})_{\alpha \in \cI}$ in $G$, we have $\sigma m_{\beta}=\sigma_{\beta} m_{\beta}$. Thus $(\sigma m_{\beta}) m'=\sigma_{\beta}(m)$ belongs to $\fp$. Since this holds for all $\sigma$ and $\fp$ is $G$-prime, it follows that $m_{\beta} \in \fp$ or $m' \in \fp$. In the latter case, we apply the same argument to some $\gamma \in \cI \setminus \{\beta\}$ and continue by induction. The claim follows.

Now, suppose $\fq \ne \fp$. We claim that there exists a monomial in $\fp$ that does not belong to $\fq$. The argument is the same as that used in the proof of Lemma~\ref{lem:class-3-1}. Indeed, let $f \in \fp \setminus \fq$. Suppose $f$ used $\xi_1, \ldots, \xi_r$. Of all monomials appearing in $f$ with non-zero coefficient, let $\xi_1^{n_1} \cdots \xi_r^{n_r}$ be one that is minimal with respect to divisibility. Multiplying $f$ by $\xi_1^{e(\pi(1))-n_1-1} \cdots \xi_r^{e(\pi(r))-n_r-1}$, we see that $m=\xi_1^{e(\pi(1))-1} \cdots \xi_r^{e(\pi(r))-1}$ appears with non-zero coefficient, and all other monomials have a factor of the form $\xi_i^k$ with $k \ge e(\pi(i))$ and thus belong to $\fq$. Thus $m \in \fp \setminus \fq$ proving the claim.

Combining the previous two paragraphs, we see that there is some $\alpha \in \cI$ and a monomial $m \in R_{\alpha}$ such that $m \in \fp \setminus \fq$. However, this is a contradiction as $\fp$ and $\fq$ have the same contraction to $R_{\alpha}$. Thus $\fp=\fq$.
\end{proof}

We are now ready to prove the theorem:

\begin{proof}[Proof of Theorem~\ref{thm:class3}]
Let $\fq$ be a $\fY_{\cU}$-prime of $\tilde{S}_{\cU}$ such that $\tilde{J}_{\cU} \subset \rad(\fq)$. Let $\fc$ be the contraction of $\fq$ to $\tilde{S}_{\cU,0}$, which is a prime ideal, and let $K$ be the residue field of $\fc$. Let $S'=\tilde{S}_{\cU} \otimes_{\tilde{S}_{\cU,0}} K \cong K[\epsilon_i]_{i \ge 1}$. The extension $\fq^e$ of $\fq$ to $S'$ is $\fY_{\cU}$-prime (Proposition~\ref{prop:prime-loc}). Since $\tilde{J}_{\cU}^e \subset \rad(\fq^e)$ and $\tilde{J}_{\cU}^e=\langle \epsilon_i \rangle_{i \ge 1}$ is maximal, we have $\rad(\fq^e)=\tilde{J}_{\cU}^e$.

Now, let $R'=K[\xi_i]_{i \ge 1}$ be like $R$ but with coefficient ring $K$. Then the $K$-algebra homomorphism $R' \to S'$ given by $\xi_i \mapsto \epsilon_i$ is an isomorphism and $\fY_{\cU}$-equivariant. It thus follows from Lemma~\ref{lem:class-3-3} that $\fq^e=\langle \epsilon_i^{e(\pi(i))} \rangle_{i \ge 1}$ for some reduced weighting $e$. We thus see that $\fq^e$ coincides with the extension of $\fQ_{\cU}(\lambda, e; Z)$, where $Z \subset \bU^{\cI}_A$ is the variety defined by $\fc$. Hence $\fq=\fQ_{\cU}(\lambda,e;Z)$ by Proposition~\ref{prop:prime-loc}.
\end{proof}

\subsection{Completion of the proof}

We now finish the proof of Theorem~\ref{thm:class}. Let $\fp$ be an finitary $\fS$-prime of $R$ of type $\lambda$. Let $\cI$ be the index set of $\lambda$ and let $\cU$ be a partition of $[\infty]$ of type $\lambda$. Let $\fp_1$ be the extension of $\fp$ to $\tilde{R}_{\cU}$. By Theorem~\ref{thm:class1}, $\fp_1$ is an $\rN \fY_{\cU}$-prime of $\tilde{R}_{\cU}$ with $\tilde{I}_{\cU} \subset \rad(\fp_1)$, and $\fp$ is the $\fS$-contraction of $\fp_1$. Since $\fY_{\cU}$ is a finite index normal subgroup of $\rN \fY_{\cU}$, it follows from Proposition~\ref{prop:prime-subgroup} that there is a $\fY_{\cU}$-prime $\fp_2$ of $\tilde{R}_{\cU}$ such that $\fp_1=\bigcap_{\sigma \in \rN_{\lambda}} \sigma \fp_2$. Since $\fp_1 \subset \fp_2$, we have $\tilde{I}_{\cU} \subset \rad(\fp_2)$. Per Theorem~\ref{thm:class2}, let $\fp_3$ be the unique $\fY_{\cU}$-prime of $\tilde{S}_{\cU}$ that contracts to $\fp_2$ and satisfies $\tilde{J}_{\cU} \subset \rad(\fp_3)$. By Theorem~\ref{thm:class3}, we have $\fp_3=\fQ_{\cU}(\lambda,e;Z)$ for some $e$ and $Z$. We thus see that $\fp$ is the $\fS$-contraction of $\fp_3$ to $R$, and so $\fp=\fP(\lambda,e;Z)$, as required.

The uniqueness aspect of Theorem~\ref{thm:class} follows from the above analysis. Indeed, the type of $\fp$ is unique up to isomorphism by definition. Once this is fixed, the only choice in the above analysis is that of $\fp_2$. Suppose $\fp_2'$ is a second choice, leading to data $(e',Z')$. Then $\fp_2'=\sigma \fp_2$ for some $\sigma \in \rN_{\lambda}$ by Proposition~\ref{prop:prime-subgroup}, which shows that $(e',Z')=\sigma (e,Z)$, and so $(\lambda,e;Z)$ is isomorphic to $(\lambda,e';Z')$.

\section{Containments} \label{s:contain}

In this section we prove the containment theorem (Theorem~\ref{mainthm2}), which is the second principal theorem of this paper.

\subsection{The $\Theta$ construction} \label{ss:theta}

Let $\lambda$ and $\mu$ be $\infty$-compositions with index sets $\cI$ and $\cK$ and let $e$ and $d$ be reduced weightings on $\cI$ and $\cK$. Given a subset $Z$ of $\bU^{\cI}$, we define a closed subset $\Theta^{\lambda,e}_{\mu,d}(Z)$ of $\bU^{\cK}$. When the parameters $\lambda$, $e$, $\mu$, and $d$ are clear from context, as will be the case throughout most of this section, we simply write $\Theta(Z)$. The $\Theta$ construction will be important throughout this section and the remainder of the paper.

Consider a pair $(\cE, \phi)$ where $\cE$ is a subset of $\cI$ and $\phi \colon \cE \to \cK$ is a function. For $\beta \in \cK$, let $\cE_{\beta}$ be the fiber of $\phi$ above $\beta$. We say that the pair $(\cE, \phi)$ is \defi{good} if the following conditions hold:
\begin{itemize}
\item[(G1)] For $\beta \in \cK$ we have $\mu_{\beta} \le \sum_{\alpha \in \cE_{\beta}} \lambda_{\alpha}$.
\item[(G2)] For $\beta \in \cK^{\infty}$ we have $d_{\beta} \le \sum_{\alpha \in \cE^{\infty}_{\beta}} e_{\alpha}$.
\end{itemize}
Condition (G1) implies that $\phi \colon \cE \to \cK$ is surjective. We note that ``good'' depends on $\lambda$, $e$, $\mu$, and $d$. When these quantities are not clear from context, we say that $(\cE, \phi)$ is ``a good pair between $(\mu,d)$ and $(\lambda,e)$.''

Let $(\cE, \phi)$ be a good pair. Let $i \colon \bA^{\cK} \to \bA^{\cE}$ be the map associated to $\phi$, and let $p \colon \bA^{\cI} \to \bA^{\cE}$ be the projection map. We define $\Theta^{\phi}_{\cE}(Z) \subset \bU^{\cK}$ to be $i^{-1}(\ol{p(Z)}) \cap \bU^{\cK}$. Since $\phi$ is surjective, $i$ is a closed immersion, and so $\Theta^{\phi}_{\cE}(Z)$ is a closed subset of $\bU^{\cK}$. Finally, we define $\Theta(Z)$ to be the union of the $\Theta^{\phi}_{\cE}(Z)$ over all good pairs $(\cE,\phi)$. This is a closed subset of $\bU^{\cK}$.

\begin{remark} \label{rmk:theta}
The $\Theta$ construction has  the following geometric interpretation. Regard $\bU^{\cI}$ as the configuration space of distinct points in $\bA^1$ where the points are labeled by $\cI$, have multiplicities specified by $\lambda$, and have weights specified by $e$, as described in \S \ref{ss:intro-spec}. Similarly interpret $\bU^{\cK}$. Let $y=(y_{\beta})_{\beta \in \cK}$ be a point of $\bU^{\cK}$. Consider a 1-parameter family $z(t)$ in $Z$, defined for $t \ne 0$. We regard $z(t)$ as a collection $\{z_{\alpha}(t)\}_{\alpha \in \cI}$ of 1-parameter families in $\bA^1$ defined for $t \ne 0$. Suppose that for each $\beta \in \cK$ there is some subset $\cE_{\beta}$ of $\cI$ such that the following conditions hold:
\begin{enumerate}
\item The $z_{\alpha}(t)$'s with $\alpha \in \cE_{\beta}$ each converge to $y_{\alpha}$ as $t \to 0$.
\item The total multiplicity of the points $z_{\alpha}$ with $\alpha \in \cE_{\beta}$ is at least the multiplicity of $y_{\beta}$.
\item If $y_{\beta}$ has infinite multiplicity, then the total weight of the points $z_{\alpha}$ with $\alpha \in \cE_{\beta}$ and $z_{\alpha}$ of infinite multiplicity, is at least the weight of $y_{\beta}$
\end{enumerate}
Then $y$ belongs to $\Theta(Z)$. To see this, take $\cE=\bigcup_{\beta \in \cK} \cE_{\beta}$, and define $\phi \colon \cE \to \cK$ by mapping $\cE_{\beta}$ to $\beta$. Conditions (b) and (c) show that $(\cE, \phi)$ is good, while condition (a) shows that $i(y)$ belongs to $\ol{p(Z)}$. (Note that the $z_{\alpha}(t)$ with $\alpha \not\in \cE$ do not need to converge as $t \to 0$: this is why it is important to first form the projection $p(Z)$ and only then take the Zariski closure $\ol{p(Z)}$ when defining $\Theta^{\phi}_{\cE}(Z)$.) The converse of the above discussion holds too. We thus see that $\Theta(Z)$ can be described as the set of ``degenerations'' of configurations in $Z$ that satisfy  certain numerical conditions on the multiplicities and weights.
\end{remark}

\subsection{Statement of results}
\label{ss:contain-statement}

Let $(\lambda,e)$ and $(\mu,d)$ be as in the previous section, and let $Z \subset \bU^{\cI}$ and $Y \subset \bU^{\cK}$ be closed subsets. Put $\fp=\fP(\lambda,e;Z)$ and $\fq=\fP(\mu,d;Y)$. The following is the most general version of the containment theorem, and the main result of \S \ref{s:contain}.

\begin{theorem} \label{thm:contain}
We have $\fp \subset \fq$ if and only if $Y \subset \Theta(Z)$.
\end{theorem}

If $Y$ is irreducible then $Y \subset \Theta(Z)$ if and only if $Y \subset \Theta^{\phi}_{\cE}(Z)$ for some good pair $(\cE, \phi)$. This yields the version of the containment theorem formulated in the introduction (Theorem~\ref{mainthm2}).

The proof of Theorem~\ref{thm:contain} is extremely involved, and consumes the rest of this section. We now give an overview. For this, it will be convenient to separate and label the two logical implications of the theorem:
\begin{enumerate}
\item If $\fp \subset \fq$ then $Y \subset \Theta(Z)$.
\item If $Y \subset \Theta(Z)$ then $\fp \subset \fq$.
\end{enumerate}
The proof of (a), carried out in \S \ref{ss:contain-a}, is reasonably straightforward: we prove the contrapositive by simply writing down an element of $\fp$ not in $\fq$. The proof of (b) is much more difficult: the containment $\fp \subset \fq$ is a statement about every element of $\fp$, but it is hard to get a handle on the elements of $\fp$ since it is defined as an infinite intersection of contracted ideals. This second step will take the remainder of \S \ref{s:contain}.

To describe the core idea of the proof of (b), suppose that $A$ is a field, $\mu=(\infty)$, $d=(n)$, and $Y=\{0\}$. In this case, we can write down the ideal $\fq$ explicitly: it is simply $\langle \xi_i^n \rangle_{i \ge 1}$. Suppose that $Y \subset \Theta(Z)$ but $\fp \not\subset \fq$. Let $f$ be an element of $\fp$ that does not belong to $\fq$. Since $f$ does not belong to $\fq$, it contains some monomial with non-zero coefficient where all exponents are $<n$. Multiplying $f$ by certain discriminants and skew-averaging over appropriate symmetric groups, we show that $\fp$ contains an element of the form $f_1 f_2$ where $f_1$ is a disjoint product of certain discriminants and $f_2$ has non-zero constant term. Using the fact that $0 \in \Theta(Z)$, we construct a specific homomorphism $\psi \colon R \to T$ of class $(\lambda,e;Z)$, and show that $\psi(f_1)\ne 0$ and $\psi(f_2)$ is a unit. It follows that $\psi(f_1f_2) \ne 0$, which is a contradiction. We carry out the details of this argument in a simple case in \S \ref{sss:simple}.

The prove (b) in the general, we essentially reduce to the above case. Much of the work in proving (b) is actually devoted to carrying out this reduction, so we attempt to summarize the key ideas. Let $\cU$ be a partition of $[\infty]$ of type $\mu$. In \S \ref{ss:tilde-primes}, we introduce certain $\fY_{\cU}$-ideals of $\tilde{R}_{\cU}$, denoted $\fP^{\phi}_{\cU}(\lambda,e;Z)$. Here $\lambda$ is an $\infty$-composition that refines $\mu$ (and the refinment is specified by $\phi$) and $e$ and $Z$ are as usual. These ideals are analogous to the familiar $\fS$-ideals $\fP(\lambda,e;Z)$. In \S \ref{ss:ext}, we study the extension $\fS$-ideals of $R$ to $\tilde{R}_{\cU}$, and show that the extension $\fP(\lambda,e;Z)$ is a specific finite intersection of ideals of the form $\fP^{\phi}_{\cU}(\kappa,e;Z)$. Using this, in \S \ref{ss:contain2} we reduce Theorem~\ref{thm:contain}(b) to an analogous statement (Theorem~\ref{thm:contain2}) concerning containments $\fp \subset \fq$ where $\fp=\fP^{\phi}_{\cU}(\lambda,e;Z)$ and $\fq=\fP^{\id}_{\cU}(\mu,d;Y)$. The $\id$ in the subscript for $\fq$ indicates that we are considering the trivial refinement of $\mu$; such an ideal is analogous to an $\fS$-ideal of $R$ of the form $\fP(\nu,f;W)$ where $\nu$ has a single part. Thus this reduction nearly places us in the context of the previous paragraph. In \S \ref{ss:punctual}, prove Theorem~\ref{thm:contain2} in the ``punctual case'' where $Y$ is a point: this proceeds much like the previous paragraph. Finally, in \S \ref{ss:contain-gen}, we deduce the general case of Theorem~\ref{thm:contain2} from the punctual case.

\subsection{Proof of Theorem~\ref{thm:contain}(a)} \label{ss:contain-a}

Fix $\fS$-ideals $\fp=\fP(\lambda,e;Z)$, $\fq = \fP(\mu,d;Y)$ as above throughout this subsection. The goal of this subsection is to prove the following result:

\begin{proposition} \label{prop:contain-a}
Assume $Y \not\subset \Theta(Z)$. Then $\fp \not\subset \fq$.
\end{proposition}

For the rest of the subsection, we assume $Y \not\subset \Theta(Z)$. We will show $\fp \not\subset \fq$ by explicitly writing down an element of $\fp$ that does not belong to $\fq$. We fix the following notation:
\begin{itemize}
\item $\cI$ is the index set of $\lambda$ and $\cK$ is the index set of $\mu$.
\item $n=1+\sum_{\alpha \in \cI^{\rf}} \lambda_{\alpha}$.
\item For $\beta \in \cK$ define $\tau_{\beta}=\mu_{\beta}$ if $\mu_{\beta}$ is finite, and $\tau_{\beta}=nd_{\beta}$ if $\mu_{\beta}$ is infinite.
\item $m=\sum_{\beta \in \cK} \tau_{\beta}$.
\item For a subset $S$ of $[\infty]$, we put $\ol{S}=S \cap [m]$.
\item $\cU=\{U_{\beta}\}_{\beta \in \cK}$ is a partition of $[\infty]$ of type $\mu$ such that $\ol{U}_{\beta}$ has size $\tau_{\beta}$ for each $\beta$.
\item For $\beta \in \cK^{\infty}$, let $\ol{U}_{\beta}=\ol{U}_{\beta,1} \sqcup \cdots \sqcup \ol{U}_{\beta,n}$ be a partition of $\ol{U}_{\beta}$ into sets of size $d_{\beta}$.
\item $y$ is a point of $Y$ that does not belong to $\Theta(Z)$. \qedhere
\end{itemize}

\begin{construction} \label{con:g}
Suppose $(\cE, \phi)$ is good.
\begin{itemize}
\item Let $\cV=\{V_{\alpha}\}_{\alpha \in \cI}$ be a partition of $[\infty]$. We say that $\cV$ is \defi{compatible} with $(\cE, \phi)$ if the following conditions hold:
\begin{itemize}[leftmargin=3.5em]
\item[(C1)] We have $\alpha \in \cE$ if and only if $\ol{V}_{\alpha} \ne \emptyset$.
\item[(C2)] We have $\ol{U}_{\beta} = \bigcup_{\alpha \in \cE_{\beta}} \ol{V}_{\alpha}$.
\item[(C3)] We have $\#(\ol{U}_{\beta,i} \cap \ol{V}_{\alpha}) \le e_{\alpha}$ for all $\alpha \in \cI$, $\beta \in \cK^{\infty}$, and $1 \le i \le n$.
\end{itemize}
Let $\Pi_{\cE,\phi}$ be the set of partitions $\cV$ that are compatible with $(\cE, \phi)$. It is easy to see that $\Pi_{\cE,\phi}$ contains a partition of type $\lambda$, and so $\Pi_{\cE,\phi}$ is nonempty.
\item Define two elements $\cV=\{V_{\alpha}\}$ and $\cV'=\{V'_{\alpha}\}$ of $\Pi_{\cE,\phi}$ to be equivalent if $\ol{V}_{\alpha}=\ol{V}'_{\alpha}$ for all $\alpha \in \cI$. Let $\Pi'_{\cE,\phi}$ be a subset of $\Pi_{\cE,\phi}$ containing one representative from each equivalence class. Note that $\Pi'_{\cE,\phi}$ is a finite set.
\item Choose an element  $u_{\cE,\phi} \in A[t_{\alpha}]_{\alpha \in \cE}$ that belongs to the ideal of $\ol{p(Z)} \subset \bA^{\cE}$ but does not vanish at $i(y) \in \bA^{\cE}$. This is possible because $y$ does not belong to $\Theta^{\phi}_{\cE}(Z)=i^{-1}(\ol{p(Z)})$ by assumption.
\item For $\cV \in \Pi'_{\cE,\phi}$, let $u_{\cE,\phi,\cV}$ be an element of $R$ using the variables $\{\xi_i\}_{1 \le i \le m}$ that maps to $u_{\cE,\phi}$ under the homomorphism $R \to A[t_{\alpha}]_{\alpha \in \cI}$ sending $\xi_i$ to $t_{\rho(i)}$, where $\rho \colon [\infty] \to \cI$ is associated to $\cV$. Note that this is possible since $u_{\cE,\phi}$ only uses the $t_{\alpha}$'s with $\alpha \in \cE$ and for every $\alpha \in \cE$ the set $\ol{V}_{\alpha}$ is non-empty.
\item Let $g_{\cE,\phi} = \prod_{\cV \in \Pi'_{\cE,\phi}} u_{\cE,\phi,\cV}^{\# \cE \cdot e_{\max}}$. \qedhere 
\end{itemize}
\end{construction}

\begin{construction} \label{con:h}
Let $N=2e_{\max}-1$. For each good $(\cE, \phi)$, let $g_{\cE,\phi}$ be as in Construction~\ref{con:g}. Now put:
\begin{displaymath}
h_1 = \prod_{\beta \in \cK^{\infty}} \prod_{1 \le k \le n} \Delta_{\ol{U}_{\beta,k}}, \qquad
h_2 = \prod_{\substack{\beta \ne \gamma \in \cK, \\ i \in \ol{U}_{\beta}, j \in \ol{U}_{\gamma}}} (\xi_i-\xi_j)^N, \qquad
h_3 = \prod_{\textrm{$(\cE,\phi)$ good}} g_{\cE,\phi}.
\end{displaymath}
Finally, put $h = h_1 h_2 h_3$.
\end{construction}

An extended example of these constructions is given in \S \ref{ss:circle}.

\begin{remark} \label{rem:optimize}
We note that our arguments below remain unaffected if we replace $h_3 = \prod_{\textrm{$(\cE,\phi)$ good}} g_{\cE,\phi}$ by the product of distinct elements in $\{g_{\cE,\phi} \colon  \textrm{$(\cE,\phi)$ good} \}$. We use this observation to optimize our example in \S \ref{ss:circle}.
\end{remark}	

 Proposition~\ref{prop:contain-a} is a consequence of the following more precise result:

\begin{proposition}
	\label{prop:element-h}
Let $h$ be as in Construction~\ref{con:h}. Then $h \in \fp \setminus \fq$.
\end{proposition}

We break the proof of Proposition~\ref{prop:element-h} into two lemmas.

\begin{lemma}
\label{lem:in-fp}
Let $h$ be as in Construction~\ref{con:h}. Then $h \in \fp$.
\end{lemma}

\begin{proof}
Let $\psi \colon R \to T$ be a homomorphism of class $(\lambda,e;Z)$. Write $\psi(\xi_i)=z_{\rho(i)}+\delta_i$ where $\rho \colon [\infty] \to \cI$ is associated to a partition $\cV=\{V_{\alpha}\}_{\alpha \in \cI}$ of type $\lambda$ of $[\infty]$, $z=(z_{\alpha})_{\alpha \in \cI}$ is a $T$-point of $Z$, and $\delta_i^{e(\rho(i))}=0$. We show that $\psi(h)=0$. To do this, we assume that $\psi(h_1)$ and $\psi(h_2)$ are non-zero and show that $\psi(h_3)=0$.

\textit{\textbf{Step 1:} the pair $(\cE, \phi)$.} Let $\cE \subset \cI$ be the set of indices $\alpha$ such that $\ol{V}_{\alpha}$ is non-empty. Then $\{\ol{V}_{\alpha}\}_{\alpha \in \cE}$ forms a partition of $[m]$. We claim that this partition refines the partition $\{\ol{U}_{\beta}\}_{\beta \in \cK}$, meaning that for each $\alpha \in \cE$ there is a (necessarily unique) $\beta \in \cK$ such that $\ol{V}_{\alpha} \subset \ol{U}_{\beta}$. Suppose not. Then there exists $i,j \in \ol{V}_{\alpha}$ with $i \in \ol{U}_{\beta}$ and $j \in \ol{U}_{\gamma}$ for some $\beta \ne \gamma$ in $\cK$. Now, $\psi(\xi_i)-\psi(\xi_j)=\delta_i-\delta_j$, and $(\delta_i-\delta_j)^N=0$. It follows that $\psi((\xi_i-\xi_j)^N)=0$, and so $\psi(h_2)=0$, a contradiction. Thus the claim follows. We therefore find that there is a unique function $\phi \colon \cE \to \cK$ such that $\phi(\alpha)=\beta$ if $\ol{V}_{\alpha} \subset \ol{U}_{\beta}$.

\textit{\textbf{Step 2:} goodness and compatibility.} We now show that $(\cE, \phi)$ is good and $\cV$ is compatible with $(\cE, \phi)$. We know from Step~1 that (C1) and (C2) hold.

We now verify (G1). Suppose $\beta \in \cK$. We have $\ol{U}_{\beta}=\bigcup_{\alpha \in \cE_{\beta}} \ol{V}_{\alpha}$, and so $\# \ol{U}_{\beta}=\sum_{\alpha \in \cE_{\beta}} \# \ol{V}_{\alpha}$. If $\beta \in \cK^{\rf}$ then $\# U_{\beta}=\mu_{\beta}$; since $\# \ol{V}_{\alpha} \le \lambda_{\alpha}$, we thus find $\mu_{\beta} \le \sum_{\alpha \in \cE_{\beta}} \lambda_{\alpha}$, as required. If $\beta \in \cK^{\infty}$ then $\# U_{\beta}$ is greater than the sum of all finite parts of $\lambda$, and so there must be some $\alpha \in \cE_{\beta}$ with $\lambda_{\alpha}=\infty$; thus $\mu_{\beta} \le \sum_{\alpha \in \cE_{\beta}} \lambda_{\alpha}$.

We now verify (C3), namely, that $\#(\ol{V}_{\alpha} \cap \ol{U}_{\beta,k})\le e_{\alpha}$ for all $\alpha \in \cE$, $\beta \in \cK^{\infty}$, and $1 \le k \le n$. Suppose not. For $i,j \in \ol{V}_{\alpha} \cap \ol{U}_{\beta,k}$, we have $\psi(\xi_i-\xi_j)=\delta_i-\delta_j$. We thus see that $\psi(\Delta_{\ol{U}_{\beta,k}})$ contains as a factor the discriminant of $(\delta_i)_{i \in \ol{V}_{\alpha} \cap \ol{U}_{\beta,k}}$. This is a discriminant on $>e_{\alpha}$ quantities, and so in every term at least one of the $\delta_i$ is raised to the $e_{\alpha}$ power. Since $\delta_i^{e_{\alpha}}=0$, we thus see that this discriminant vanishes. Thus $\psi(\Delta_{\ol{U}_{\beta,k}})=0$, and so $\psi(h_1)=0$, a contradiction. This proves the claim.

Finally, we verify (G2). Suppose, by way of contradiction, there is some $\beta \in \cK^{\infty}$ such that $d_{\beta}>\sum_{\alpha \in \cE^{\infty}_{\beta}} e_{\alpha}$. Since $n$ is greater than the sum of all finite parts of $\lambda$, some there is some $1 \le k \le n$ such that $\ol{U}_{\beta,k}$ does not meet any $V_{\alpha}$ with $\alpha$ finite. Since
\begin{displaymath}
\sum_{\alpha \in \cE^{\infty}_{\beta}} e_{\alpha} < d_{\beta}=\# \ol{U}_{\beta,k}=\sum_{\alpha \in \cE^{\infty}_{\beta}} \# (\ol{V}_{\alpha} \cap \ol{U}_{\beta,k}),
\end{displaymath}
there must be some $\alpha \in \cE^{\infty}_{\beta}$ such that $\#(\ol{V}_{\alpha} \cap \ol{U}_{\beta,k})>e_{\alpha}$, contradicting (C3).

\textit{\textbf{Step 3:} vanishing of $\psi(h_3)$.} We now show that $\psi(h_3)=0$. We do this by showing $\psi(g_{\cE,\phi})=0$. In fact, letting $\cV' \in \Pi'_{\cE,\phi}$ be equivalent to $\cV$, we show that $\psi(u_{\cE, \phi,\cV'})^{\# \cE \cdot e_{\max}} = 0$. Since $u_{\cE,\phi}$ vanishes on $p(Z)$ and $(z_{\alpha})_{\alpha \in \cE}$ is a $T$-point of $p(Z)$, it follows that $u_{\cE,\phi}$ vanishes at $(z_{\alpha})$. Thus if we substitute $z_{\rho(i)}$ for $\xi_i$ in $u_{\cE,\phi, \cV'}$ we get~0 (note that $u_{\cE,\phi,\cV'}$ only uses $\xi_i$ with $i \le m$, so we can change $\cV'$ to $\cV$ here). Therefore when we substitute $\psi(\xi_i)=z_{\rho(i)}+\delta_i$ for $\xi_i$ in $u_{\cE, \phi,\cV'}$, we get something that belongs to the ideal of $T$ generated by the $\delta_i$ with $1 \le i \le m$. Since $\delta_i^{e_{\max}}=0$ and $u_{\cE, \phi,\cV'}$ uses at most $\# \cE$ variables, we see that $\psi(u_{\cE, \phi,\cV'})^{\# \cE \cdot e_{\max}} = 0$. The result follows.
\end{proof}

\begin{lemma}
	\label{lem:not-in-fq}
Let $h$ be as in Construction~\ref{con:h}. Then $h \not\in \fq$.
\end{lemma}

\begin{proof}
Let $\Omega$ be the residue field of the point $y \in Y$, which is naturally an $A$-algebra. We can thus represent $y$ as a tuple $(y_{\beta})_{\beta \in \cK}$ with $y_{\beta} \in \Omega$. let $T=\Omega[\delta_i]/\langle \delta_i^{d(\pi(i))} \rangle$ and consider the $A$-algebra homomorphism $\psi \colon R \to T$ given by $\psi(\xi_i) = y_{\pi(i)} + \delta_i$, where $\pi \colon [\infty] \to \cK$ is associated to $\cU$. Then $\psi$ is a map of class $(\mu, d; Y)$. Thus to show that $h \not\in \fq$, it suffices to show that $\psi(h) \ne 0$. We show that $\psi(h_2)$ and $\psi(h_3)$ are units and $\psi(h_1)$ is non-zero.

\textit{\textbf{Step 1:} $\psi(h_1)$ is non-zero.} The element $\psi(\Delta_{\ol{U}_{\beta,k}})$ is the discriminant of the $d_{\beta}$ elements $\{\delta_i \mid i \in \ol{U}_{\beta,k}\}$. This is a polynomial in which all powers of the variables are $<d_{\beta}$ (see \S \ref{ss:disc}), and thus a non-zero element of $T$. (Note that for $i \in \ol{U}_{\beta,k}$, the defining relation for $\delta_i$ is $\delta_i^{d_{\beta}}=0$.) We have
\begin{displaymath}
\psi(h_1) = \prod_{\beta \in \cK^{\infty}} \prod_{1 \le k \le n} \psi(\Delta_{\ol{U}_{\beta,k}}).
\end{displaymath}
We have just seen that each factor in the product is non-zero. Since the factors use disjoint sets of variables, it follows that their product is non-zero as well.

\textit{\textbf{Step 2:} $\psi(h_2)$ is a unit.} We have
\begin{displaymath}
\psi(h_2) = \prod_{\substack{\beta \ne \gamma \in \cK, \\ i \in \ol{U}_{\beta}, j \in \ol{U}_{\gamma}}} (y_{\beta}-y_{\gamma})^N \pmod{\langle \delta_i \rangle_{i \ge 1}}
\end{displaymath}
Since $y$ is a point of $Y \subset \bU^{\cK}$, we have $y_{\beta} \ne y_{\gamma}$, and so the product is a non-zero element of $\Omega$. Since the $\delta_i$'s are nilpotent, it follows that $\psi(h_2)$ is a unit.

\textit{\textbf{Step 3:} $\psi(h_3)$ is a unit.}
Suppose $(\cE, \phi)$ is good and $\cV \in \Pi'_{\cE, \phi}$. Since $u_{\cE,\phi}$ does not vanish at $i(y)$, we see that by taking $u_{\cE,\phi}$ and performing the substitutions $t_{\alpha} \to y_{\phi(\alpha)}$, we obtain a non-zero element of $\Omega$. As $\psi(u_{\cE,\phi,\cV})$ is equal to this element modulo the $\delta_i$'s, and the $\delta_i$'s are nilpotent, we see that $\psi(u_{\cE,\phi,\cV})$ is a unit of $T$. Thus $\psi(h_3)$ is a unit of $T$.
\end{proof}

\subsection{Ideals of $\tilde{R}_{\cU}$} \label{ss:tilde-primes}

Let $\mu$ be an $\infty$-composition with index set $\cK$, let $\cU=\{U_{\beta}\}_{\beta \in \cK}$ be a partition of $[\infty]$ of type $\mu$, and let $\pi \colon [\infty] \to \cK$ be the associated function. Let $\lambda$ be an $\infty$-composition with index set $\cI$, and let $\phi \colon \cI \to \cK$ be a function such that $\mu_{\beta}=\sum_{\phi(\alpha)=\beta} \lambda_{\alpha}$; we call $(\lambda, \phi)$ a \defi{refinement} of $\mu$. Let $\cV=\{V_{\alpha}\}_{\alpha \in \cI}$ be a partition of $[\infty]$ of type $\lambda$. We say that $\cV$ is a \defi{refinement} of $\cU$ if $V_{\alpha} \subset U_{\phi(\alpha)}$ for all $\alpha \in \cI$. In this case, $\{V_{\alpha}\}_{\alpha \in \cI_{\beta}}$ forms a partition of the set $U_{\beta}$ for each $\beta \in \cK$.

Let $(\lambda,\phi)$ be a refinement of $\mu$, let $e$ be a weighting on $\cI$, and let $Z$ be a subset of $\bU^{\cI}$. Choose a partition $\cV$ of $[\infty]$ of type $\lambda$ refining $\cU$. We then have a natural homomorphism $\tilde{R}_{\cU} \to \tilde{R}_{\cV}$. Recall that $\fQ_{\cV}(\lambda,e;Z)$ is the prime of $\tilde{S}_{\cV}$ associated to $(\lambda,e;Z)$ (see Definition~\ref{defn:Q-ideal}). We define $\fP^{\phi}_{\cU}(\lambda,e;Z)$ to be the $\fY_{\cU}$-contraction of $\fQ_{\cV}(\lambda,e;Z)$ along the maps $\tilde{R}_{\cU} \to \tilde{R}_{\cV} \to \tilde{S}_{\cV}/\tilde{J}_{\cV,n}$ for $n$ sufficiently large. This is a $\fY_{\cU}$-ideal that is independent of the choice of $\cV$. These ideals are similar to the $\fP$ ideals defined in \S \ref{ss:P}. We note a few properties in particular.

\begin{proposition} \label{prop:Pphi-prop}
We have the following:
\begin{enumerate}
\item If $e$ is reduced and $Z$ is irreducible then $\fP^{\phi}_{\cU}(\lambda,e;Z)$ is $\fY_{\cU}$-prime.
\item We have $\rad_{\fY_{\cU}}(\fP^{\phi}_{\cU}(\lambda,e;Z))=\fP^{\phi}_{\cU}(\lambda, \red_{\lambda}(e); Z)$.
\item If $Z=\bigcup_{i \in I} Z_i$ then $\fP^{\phi}_{\cU}(\lambda, e; Z)=\bigcap_{i \in I} \fP^{\phi}_{\cU}(\lambda, e; Z_i)$.
\end{enumerate}
\end{proposition}

\begin{proof}
The proof is similar to the proofs of Propositions~\ref{prop:Pprime} and~\ref{prop:Pprop}.
\end{proof}

\begin{proposition} \label{prop:flat-bc}
Let $(\lambda, \phi)$ be a refinement of $\mu$ with index set $\cI$, let $e$ be a $\lambda$-reduced weighting on $\cI$, and let $Z$ be a closed subset of $\bU^{\cI}_A$. Let $A \to A'$ be a flat ring map, let $\tilde{R}_{\cU}'=A' \otimes_A \tilde{R}_{\cU}$, and let $Z'$ be the inverse image of $Z$ in $\bU^{\cI}_{A'}$. Put $\fp=\fP^{\phi}_{\cU}(\lambda, e; Z)$ and $\fp'=\fP^{\phi}_{\cU}(\lambda, e; Z')$. Then
\begin{enumerate}
\item We have $\rad_{\fY_{\cU}}(\fp^e)=\fp'$.
\item If $\fI(Z)^e \subset \tilde{S}'_{\cU,0}$ is radical then $\fp^e=\fp'$.
\end{enumerate}
We note that if $A \to A'$ is a localization then (b) holds.
\end{proposition}

\begin{proof}
The proof is similar to that of Proposition~\ref{prop:P-ext}.
\end{proof}

We will require an alternative description of $\fP^{\phi}_{\cU}(\lambda,e;Z)$. We say that a homomorphism $\psi \colon \tilde{R}_{\cU} \to T$ has \defi{class} ${}^{\phi}_{\cU}(\lambda, e; Z)$ if  there exists a partition $\cV$ of $[\infty]$ of type $\lambda$ refining $\cU$ and a ring homomorphism $\psi_0 \colon \tilde{S}_{\cV,0} \to T$ such that the following conditions hold:
\begin{enumerate}
\item We have $\psi \vert_A = \psi_0 \vert_A$.
\item We have $\fI_0(Z) \subset \ker(\psi_0)$.
\item We have $\psi(\xi_i)=\psi_0(t_{\rho(i)})+\delta_i$ where $\delta_i \in T$ satisfies $\delta_i^{e(\rho(i))}=0$ and $\rho \colon [\infty] \to \cI$ is the function associated with $\cV$.
\end{enumerate}
Then $\fP^{\phi}_{\cU}(\lambda, e; Z)$ is the intersection of $\ker(\psi)$ over all homomorphisms $\psi \colon \tilde{R}_{\cU} \to T$ of class ${}^{\phi}_{\cU}(\lambda, e; Z)$; in fact, it suffices to consider rings $T$ whose nilradical is maximal. The proof is analogous to that of Proposition~\ref{prop:ker-psi}.

\subsection{Extending to $\tilde{R}_{\cU}$} \label{ss:ext}

Let $\cU$ and $\mu$ be as in the previous section. Let $\lambda$ be an $\infty$-composition on the index set $\cI$. Fix an $\fS$-ideal $\fp=\fP(\lambda,e;Z)$ of $R$, where $e$ is an arbitrary weighting on $\cI$ and $Z$ is an arbitrary subset of $\bU^{\cI}$. Our goal is to determine the extension $\fp^e$ of $\fp$ to $\tilde{R}_{\cU}$. Let $S$ be the set of pairs $(\phi,\kappa)$ where $\kappa$ is an $\infty$-weighting on $\cI$ and $\phi \colon \cI \to \cK$ is a function such that the following three conditions hold:
\begin{enumerate}
\item $\mu_{\beta}=\sum_{\phi(\alpha)=\beta} \kappa_{\alpha}$ for all $\beta \in \cK$.
\item $\kappa_{\alpha} \le \lambda_{\alpha}$ for all $\alpha \in \cI$
\item If $\phi(\alpha) \in \cK^{\infty}$ then $\kappa_{\alpha}=\lambda_{\alpha}$.
\end{enumerate}
We note that $S$ is finite, and can be empty. The following is our main result:

\begin{proposition} \label{prop:ext}
Notation as above, we have
\begin{displaymath}
\fp^e = \bigcap_{(\phi, \kappa) \in S} \fP^{\phi}_{\cU}(\kappa, e; Z).
\end{displaymath}
If $S=\emptyset$ then this says that $\fp^e$ is the unit ideal.
\end{proposition}

\begin{proof}
Let $\fq=\bigcap_{(\phi, \kappa) \in S} \fP^{\phi}_{\cU}(\kappa, e; Z)$. We first show $\fp^e \subset \fq$. Let $f \in \fp$ be given. Suppose that $\psi \colon \tilde{R}_{\cU} \to T$ is a homomorphism of class ${}^{\phi}_{\cU}(\kappa, e; Z)$ for some $(\phi, \kappa) \in S$. Let $m$ be such that $f$ uses only $\xi_1, \ldots, \xi_m$. Then there exists a homomorphism $\psi' \colon R \to T$ of class $(\lambda, e; Z)$ such that $\psi(\xi_i)=\psi'(\xi_i)$ for $1 \le i \le m$. It follows that $\psi(f)=\psi'(f)=0$. Thus $f \in \fP^{\phi}_{\cU}(\kappa, e; Z)$. Since this holds for all $(\phi, \kappa)$, it follows that $f \in \fq$. Thus $\fp \tilde{R}_{\cU} \subset \fq$.

We now prove the reverse inclusion. Let $f \in \fq$ be given. Multiplying by a unit of $\tilde{R}_{\cU}$, we can assume that $f$ belongs to $R$. Let $m$ be such that the $f$ uses only $\xi_1, \ldots, \xi_m$, all finite parts of $\cU$ are contained in $[m]$, and $\#(U_{\alpha} \cap [m])$ exceeds the sum of the finite parts of $\lambda$ for each $\alpha \in \cK^{\infty}$. For a subset $S$ of $[\infty]$, put $\ol{S}=S \cap [m]$. Let $s$ be the product of $\xi_i-\xi_j$ over all $1 \le i,j \le m$ for which $i$ and $j$ belong to different parts of $\cU$. Note that $s$ is a unit of $\tilde{R}_{\cU}$. Let $N=2e_{\max}-1$. We claim that $g=s^N f$ belongs to $\fp$. This will complete the proof, as then $f=s^{-N} g$ belongs to $\fp^e$.

To see this, let $\psi \colon R \to T$ be a homomorphism of class $(\lambda, e;  Z)$. Write $\psi(\xi_i)=a_{\rho(i)}+\delta_i$ where $\rho \colon [\infty] \to \cI$ is associated to a partition $\cV=\{V_{\alpha}\}_{\alpha \in \cI}$ of type $\lambda$. If there exists $1 \le i,j \le m$ such that $i$ and $j$ belong to the same part of $\cV$ but different parts of $\cU$ then $\psi(\xi_i-\xi_j)=\delta_i-\delta_j$, and so $\psi(s^N)=0$, and so $\psi(g)=0$. Thus suppose this is not the case. We claim that  there exists a homomorphism $\psi' \colon \tilde{R}_{\cU} \to T$ of class ${}^{\phi}_{\cU}(\kappa,e;Z)$ for some $(\phi,\kappa) \in S$ such that $\psi'(\xi_i)=\psi(\xi_i)$ for $1 \le i \le m$. Given this, we have $\psi(f)=\psi'(f)=0$, and so $\psi(g)=0$, proving $g \in \fp$.

Fix $\beta_0 \in \cK^{\infty}$. Let $\cE$ be the set of indices $\alpha \in \cI$ for which $\ol{V}_{\alpha} \ne \emptyset$. We have shown that for $\alpha \in \cE$ there exists a unique $\beta \in \cK$ such that $\ol{V}_{\alpha} \subset U_{\beta}$. Define $\phi_0 \colon \cE \to \cK$ by taking $\alpha \in \cE$ to this unique $\beta$. Let $\phi \colon \cI \to \cK$ be the extension of $\phi_0$ taking each element of $\cI \setminus \cE$ to $\beta_0$. We note that for each $\beta \in \cK^{\infty}$ there is some $\alpha \in \cE^{\infty}$ with $\phi(\alpha)=\beta$; this follows from how we chose $m$.  Fix some $\alpha_0 \in \cE^{\infty}$ with $\phi(\alpha_0)=\beta_0$. We define a new partition $\cV'=\{V'_{\alpha}\}_{\alpha \in \cI}$ of $[\infty]$ as follows. If $\phi(\alpha) \in \cK^{\rf}$, we take $V'_{\alpha}=\ol{V}_{\alpha}$. For the remaining $\alpha \ne \alpha_0$, we take $V'_{\alpha}=V_{\alpha}$. Finally, we take $V'_{\alpha_0}$ to consist of all elements $[\infty]$ not yet used; note that $V'_{\alpha_0}$ contains $V_{\alpha_0}$. Let $\kappa_{\alpha}=\# V'_{\alpha}$. Then $(\phi, \kappa) \in S$. We now define $\psi' \colon R \to T$ by $\psi' \vert_A = \psi \vert_A$ and $\psi'(\xi_i)=a_{\rho'(i)}+\delta_i$, where $\rho' \colon [\infty] \to \cI$ is associated to $\cV'$. One easily sees that $\psi'$ is of class ${}^{\phi}_{\cU}(\kappa,e;Z)$, which completes the proof.
\end{proof}

Recall that for an $\infty$-composition $\kappa$ on the index set $\cI$ and a weight function $e$ on $\cI$, we write $\red_{\kappa}(e)$ for the weighting on $\cI$ that is $e$ on $\kappa^{-1}(\infty)$ and~1 off of this subset. 

\begin{corollary} \label{cor:ext}
Notation as above, we have
\begin{displaymath}
\rad_{\fY_{\cU}}(\fp^e) = \bigcap_{(\phi, \kappa) \in S} \fP^{\phi}_{\cU}(\kappa, \red_{\kappa}(e); Z).
\end{displaymath}
\end{corollary}

\begin{proof}
This follows from the proposition and Proposition~\ref{prop:Pphi-prop}(b).
\end{proof}

\begin{corollary} \label{cor:contract-ext}
Notation as above, assume $\lambda=\mu$. Then $\fp$ is the $\fS$-contraction of $\fp^e$.
\end{corollary}

\begin{proof}
Clearly $\fp \subset (\fp^e)^{\fS c}$. For the reverse, note that $(\id,\mu) \in S$, and so by the proposition we have $\fp^e \subset \fP^{\id}_{\cU}(\mu,e;Z)$. The ideal $\fP^{\id}_{\cU}(\mu,e;Z)$ is simply the contraction of $\fQ_{\cU}(\mu,e;Z)$ to $\tilde{R}_{\cU}$. The $\fS$-contraction of this ideal to $R$ is $\fp$, by definition. Thus the result follows.
\end{proof}

\subsection{A reduction of Theorem~\ref{thm:contain}(b)} \label{ss:contain2}

Let $\cU$ and $\mu$ be as above. Let $\fp=\fP^{\phi}_{\cU}(\lambda, e; Z)$ and $\fq=\fP^{\id}_{\cU}(\mu, d; Y)$ be $\fY_{\cU}$-ideals of $\tilde{R}_{\cU}$, with $e$ and $d$ reduced; here the $\id$ superscript means that the $\phi$ map used to define $\fq$ is the identity $\cK \to \cK$, where $\cK$ is the index set for $\mu$. The goal of \S \ref{ss:contain2} is to formulate a statement concerning the containment $\fp \subset \fq$, and reduce Theorem~\ref{thm:contain}(b) to this statement.

Let $\cI$ be the index set of $\lambda$. We say that a subset $\cE$ of $\cI$ is \defi{good} if $(\cE, \phi \vert_{\cE})$ is good as defined in \S \ref{ss:theta}. Explicitly, this amounts to the following
\begin{itemize}
\item For $\beta \in \cK^{\rf}$ we have $\cI_{\beta} \subset \cE$.
\item For $\beta \in \cK^{\infty}$ we have $d_{\beta} \le \sum_{\alpha \in \cE^{\infty}_{\beta}} e_{\alpha}$.
\end{itemize}
To see the first condition, note that $\mu_{\beta}=\sum_{\alpha \in \cI_{\beta}} \lambda_{\alpha}$ in the present setting, and so the condition $\mu_{\beta} \le \sum_{\alpha \in \cE_{\beta}} \lambda_{\alpha}$ forces $\cE$ to contain $\cI_{\beta}$ if $\beta \in \cK^{\rf}$. We let $\Theta^{\phi}(Z)$ be the union of the $\Theta^{\phi}_{\cE}(Z)$ over all good $\cE$. We can now state the main result:

\begin{theorem} \label{thm:contain2}
If $Y \subset \Theta^{\phi}(Z)$ then $\fp \subset \fq$.
\end{theorem}

In fact, the converse statement holds as well, but we do not prove it since it is not needed for the proof of Theorem~\ref{thm:contain}. The proof of this theorem will take the remainder of \S \ref{s:contain}. We now show how the above theorem implies the main result of this section.

\begin{proposition}
Assume Theorem~\ref{thm:contain2} holds. Then Theorem~\ref{thm:contain}(b) holds.
\end{proposition}

\begin{proof}
Let $\fp=\fP(\lambda,e;Z)$ and $\fq=\fP(\mu,d;Y)$ be $\fS$-ideals of $R$ with $e$ and $d$ reduced such that $Y \subset \Theta(Z)$. We will show that $\fp \subset \fq$.

Let $\cU$ be as above (i.e., a partition of $[\infty]$ of type $\mu$), and let $(-)^e$ denote extension to $\tilde{R}_{\cU}$. By Corollary~\ref{cor:contract-ext}, $\fq$ is the $\fS$-contraction of $\fq^e$, and so it suffices to show $\fp^e \subset \fq^e$. Recall that $\rN_{\mu} \subset \Aut(\cK)$ is the group of permutations fixing $\mu$. By Proposition~\ref{prop:ext}, we find
\begin{displaymath}
\fq^e=\bigcap_{\sigma \in \rN_{\mu}} \fP^{\sigma}_{\cU}(\mu,d;Y)
\end{displaymath}
Let $\fq_0=\fP^{\id}_{\cU}(\mu,d;Y)$. Write
\begin{displaymath}
\rad_{\fY_{\cU}}(\fp^e)=\bigcap_{(\phi,\kappa) \in S} \fP^{\phi}_{\cU}(\kappa,\red_{\kappa}(e);Z)
\end{displaymath}
per Corollary~\ref{cor:ext}. Since $\fp^e$ is an $\rN\fY_{\cU}$-ideal, it suffices to show $\fp^e \subset \fq_0$, for then $\fp^e$ will necessarily be contained in $\bigcap_{\sigma \in \rN_{\mu}} \sigma \fq_0=\fq^e$. We now establish this containment.

Recall that $(\cE, \phi)$ is good (with respect to $\lambda$, $e$, $\mu$, and $d$) if the following conditions hold:
\begin{itemize}
\item For $\beta \in \cK$ we have $\mu_{\beta} \le \sum_{\alpha \in \cE_{\beta}} \lambda_{\alpha}$
\item For $\beta \in \cK^{\infty}$ we have $d_{\beta} \le \sum_{\alpha \in \cE_{\beta}, \lambda_{\alpha}=\infty} e_{\alpha}$.
\end{itemize}
We say that $(\cE, \phi)$ is \defi{very good} if it is good and $\# \cE_{\beta} \le \mu_{\beta}$ for $\beta \in \cK^{\rf}$. We note that if $(\cE, \phi)$ is good then there is a subset $\cE'$ of $\cE$ such that $(\cE', \phi)$ is very good; necessarily, $\Theta^{\phi}_{\cE}(Z) \subset \Theta^{\phi}_{\cE'}(Z)$. Thus $\Theta(Z)$ is the union of the $\Theta^{\phi}_{\cE}(Z)$ with $(\cE, \phi)$ very good.

Let $(\cE, \phi)$ be very good, and put $Y^{\phi}_{\cE}=Y \cap \Theta^{\phi}_{\cE}(Z)$. We extend $\phi$ to a function $\cI \to \cK$ by defining $\phi(\alpha)=\beta_0$ for $\alpha \not\in \cE$, where $\beta_0 \in \cK^{\infty}$ is chosen arbitrarily. We define a weight $\kappa$ on $\cI$ as follows. For $\beta \in \cK^{\rf}$ and $\alpha \in \phi^{-1}(\beta)$, we we pick $1 \le \kappa_{\alpha} \le \lambda_{\alpha}$ arbitrarily so that $\mu_{\beta}=\sum_{\phi(\alpha)=\beta} \kappa_{\alpha}$ (this is possible since $\# \cE_{\beta} \le \mu_{\beta} \le \sum_{\alpha \in \cE_{\beta}} \lambda_{\alpha}$). For $\alpha \in \phi^{-1}(\cK^{\infty})$, we take $\kappa_{\alpha}=\lambda_{\alpha}$. Thus $(\phi,\kappa) \in S$. It is clear that $(\cE,\phi)$ is good with respect to $\kappa$, $\red_{\kappa}(e)$, $\mu$, and $d$, and so $\fP^{\phi}_{\cU}(\kappa, \red_{\kappa}(e); Z) \subset \fP^{\id}_{\cU}(\mu,d; Y^{\phi}_{\cE})$ by Theorem~\ref{thm:contain2}. We thus find $\fp^e \subset \fP^{\id}_{\cU}(\mu,d; Y^{\phi}_{\cE})$. Now intersect over all very good pairs $(\cE,\phi)$. Since $Y \subset \Theta(Z)$, we have $Y = \bigcup_{(\cE,\phi)} Y^{\phi}_{\cE}$, and so Proposition~\ref{prop:Pphi-prop}(c) gives $\fp^e \subset \fq_0$, which completes the proof.
\end{proof}

\subsection{Set-up for remainder of \S \ref{s:contain}} \label{ss:contain2-overview}

The remainder of \S \ref{s:contain} is devoted to the proof of Theorem~\ref{thm:contain2}. The following notation will be in effect from here on:
\begin{itemize}
\item $\mu$ is an $\infty$-composition with index set $\cK$, and $d$ is a weighting on $\cK$.
\item $\lambda$ is an $\infty$-composition with index set $\cI$, and $e$ is a weighting on $\cI$.
\item $\phi \colon \cI \to \cK$ is a function with $\mu_{\beta}=\sum_{\phi(\alpha)=\beta} \lambda_{\alpha}$ for all $\beta \in \cK$.
\item $\cU=\{U_{\beta}\}_{\beta \in \cK}$ is a partition of $[\infty]$ of type $\mu$, and $\pi \colon [\infty] \to \cK$ is the associated map.
\item $\fq=\fP^{\id}_{\cU}(\mu,d;Y)$ and $\fp=\fP^{\phi}_{\cU}(\lambda,e;Z)$ are $\fY_{\cU}$-ideals of $\tilde{R}_{\cU}$ such that $e$ and $d$ are reduced, $Y$ and $Z$ are closed, and $Y \subset \Theta^{\phi}(Z)$.
\item When needed, $\cV=\{V_{\alpha}\}_{\alpha \in \cI}$ will denote a partition of $[\infty]$ of type $\lambda$ that refines $\cU$, and $\rho \colon [\infty] \to \cI$ will denote the associated function.
\end{itemize}

\subsection{Proof of Theorem~\ref{thm:contain2}: the punctual case} \label{ss:punctual}

Work in the setting of \S \ref{ss:contain2-overview}. We say that we are in the \defi{punctual case} if the coefficient ring $A$ is a local ring with algebraically closed residue field $A/\fm_A$ and $Y$ is a point of $\bU^{\cK}$ lying over the closed point of $\Spec(A)$. We do not require $A$ to be noetherian. In \S \ref{ss:punctual}, we prove Theorem~\ref{thm:contain2} in the punctual case.

\subsubsection{A simple case} \label{sss:simple}

The proof of the punctual case in general involves a lot of bookkeeping, so we first illustrate the essential idea of the argument in a special case. We assume that $A=K$ is an algebraically closed field. Suppose $\lambda=(\infty,\infty)$ and $e=(e_1,e_2)$ where $e_1$ and $e_2$ are positive integers. Let $Z$ be an irreducible curve in $\bU^2$, and let $\fp=\fP(\lambda, e; Z)$. Let $\mu=(\infty)$ and $d=e_1+e_2$, and let $\fq=\fP(\mu, d;\{0\})$.  We prove:

\begin{proposition}
Suppose that the closure of $Z$ in $\bA^2$ contains the origin. Then $\fp \subset \fq$.
\end{proposition}

Before giving the proof, we note that this does indeed fit into our set-up, by taking $\cU$ the partition of $[\infty]$ with just one part (as $\mu = (\infty)$); in this case, $\tilde{R}_{\cU}=R$ and the $\fP_{\cU}$ primes revert to our original primes $\fP$. Furthermore, in this case $\Theta^{\phi}(Z)$ consists of those points $a \in \bA^1$ such that $(a,a) \in \bA^2$ belongs to the closure of $Z$.

\begin{proof}
Assume $\fp \not\subset \fq$; we will derive a contraction. Let $f$ be an element of $\fp$ not contained in $\fq$. For notational simplicity, we suppose that $f$ uses only $\xi_1$ and $\xi_2$. As $\fq=\langle \xi_i^d \rangle$ we see that $f$ contains a monomial in which all exponents are $<d$. Multiplying by a scalar and monomial, we can assume that the coefficient of $(\xi_1 \xi_2)^{d-1}$ in $f$ is~1. Write $f=\sum_{i,j} c_{i,j} \xi_1^i \xi_2^j$, where $c_{i,j} \in K$; thus $c_{d-1,d-1}=1$.

Let $I^{\circ}$ and $J^{\circ}$ be disjoint subsets of $[\infty] \setminus \{1,2\}$ of cardinality $d-1$, and put $I=\{1\} \cup I^{\circ}$ and $J=\{2\} \cup J^{\circ}$. Let $f'$ be the polynomial obtained by multiplying $f$ by the discriminants $\Delta_I$ and $\Delta_J$ and then skew-averaging over $\Aut(I)/\Aut(I^{\circ})$ and $\Aut(J)/\Aut(J^{\circ})$. Of course, $f'$ belongs to $\fp$ since $\fp$ is an $\fS$-ideal. The skew-average of $\xi_d^k \cdot \Delta_{d-1}$ over $\fS_d/\fS_{d-1}$ vanishes if $k<d-1$, and otherwise is equal to $\Delta_d \cdot s_{k-d+1}$, where $s_k$ is a homogeneous symmetric polynomial of degree $k$. We thus find
\begin{displaymath}
f' = \Delta_I \Delta_J \sum_{i,j \ge 0} c_{d-1+i, d-1+j} s_{i,I} s_{j,J},
\end{displaymath}
where $s_{i,I}$ denotes the polynomial $s_i$ in the variables indexed by $I$. Since $s_{i,I}$ is homogeneous of degree $i$ (and $s_{0,I}=1$), we can thus write
\begin{displaymath}
f' = \Delta_I \Delta_J (1+g)
\end{displaymath}
where $g$ belongs to the ideal generated by the $\xi_i$.

Let $\cV=\{V_1,V_2\}$ be a partition of $[\infty]$ with $V_1$ and $V_2$ infinite such that for each $i \in \{1,2\}$ the sets $I \cap V_i$ and $J \cap V_i$ have cardinality $e_i$. Thus $\{I \cap V_1, I \cap V_2\}$ is a partition of the $d$-element set $I$ into subsets of size $e_1$ and $e_2$, and similarly for $J$. Such a partition $\cV$ exists since $d=e_1+e_2$. Let $\rho \colon [\infty] \to \{1,2\}$ be the projection associated to $\cV$. Since the closure of $Z$ contains the origin, we can find a $K\lpp t \rpp$-point $z$ of $Z$ that limits on the origin. Write $z=(z_1,z_2)$ where $z_i \in K\lpp t \rpp$ is the $i$th coordinate of $z$. Since $z$ limits on the origin, we have $z_i \in t K \lbb t \rbb$.

Let $T_0=K\lbb t \rbb[\epsilon_i]/\langle \epsilon_i^{e(\rho(i))} \rangle$ and let $T=K\lpp t \rpp \otimes_{K\lbb t \rbb} T_0$. Let $\psi \colon R \to T$ be the map given by $\psi(\xi_i)=z_{\rho(i)}+\epsilon_i$. This is a  map of class $(\lambda, e ;Z)$, and so $\fp \subset \ker(\psi)$. In particular, $\psi(f')=0$. We will compute $\psi(f')$ and find that it is non-zero, reaching our contradiction.

First, note that $T_0$ is a local ring and $\psi(\xi_i)$ belongs to the maximal ideal for all $i$. It follows that $\psi(g)$ belongs to the maximal ideal, and so $\psi(1+g)$ is a unit of $T_0$, and thus $T$.

Next, we examine $\psi(\Delta_I)$. We have
\begin{displaymath}
\Delta_I = \prod_{i<j} (\xi_i-\xi_j) = \Delta_{I \cap V_1} \cdot \Delta_{I \cap V_2} \cdot  \prod_{i<j, \rho(i) \ne \rho(j)} (\xi_i-\xi_j)
\end{displaymath}
where $i,j \in I$ in the two products. Call the rightmost product $\Pi$. For $(i,j)$ as in this product, we have
\begin{displaymath}
\psi(\xi_i-\xi_j)=(z_{\rho(i)}-z_{\rho(j)})+(\epsilon_i-\epsilon_j).
\end{displaymath}
Since $z_1-z_2 \ne 0$ is a unit of $K\lpp t \rpp$ and $\epsilon_i-\epsilon_j$ is nilpotent in $T$, we see that $\psi(\xi_i-\xi_j)$ is a unit of $T$. Thus $\psi(\Pi)$ is a unit of $T$. Now, $\psi(\Delta_{I \cap V_1})$ is the discriminant of $\{z_1+\epsilon_i\}_{i \in I \cap V_1}$. Since the discriminant only depends on the differences of the quantities, this coincides with the discriminant of $\{\epsilon_i\}_{i \in I \cap V_1}$. This discriminant is a sum of monomials in the $\epsilon_i$'s where all exponents are $<e_1$ (see \S \ref{ss:disc}). It is thus a non-zero element of $T$, since the index of nilpotency of $\epsilon_i$ is $e_1$ for $i \in V_1$. Similarly, $\psi(\Delta_{I \cap V_2})$ is non-zero. Since $I \cap V_1$ and $I \cap V_2$ are disjoint, it follows that the $\psi(\Delta_{I \cap V_1})$ and $\psi(\Delta_{I \cap V_2})$ have no $\epsilon_i$ in common, and so their product is non-zero. We thus see that $\psi(\Delta_I) \ne 0$.

Similarly, we see that $\psi(\Delta_J) \ne 0$. Again, since $I$ and $J$ are disjoint, it follows that $\psi(\Delta_I) \psi(\Delta_J) \ne 0$. We thus find that $\psi(f') \ne 0$, which completes the proof.
\end{proof}

\subsubsection{The general case}

We now prove Theorem~\ref{thm:contain2} in the punctual case.

\begin{proposition} \label{prop:contain2-1}
Assume we are in the punctual case. If $Y \subset \Theta^{\phi}(Z)$ then $\fp \subset \fq$.
\end{proposition}

Let $y=(y_{\beta})_{\beta \in \cK}$ be the unique point of $Y$, where $y_{\beta} \in A/\fm_A$. Let $\tilde{y}_{\beta} \in A$ be a lift of $y_{\beta}$. We let $\eta_i=\xi_i-\tilde{y}_{\pi(i)}$. The ideal $\fq$ is then generated by $\fm_A$ and $\eta_i^{d(\pi(i))}$ for $i \ge 1$. Recall that for $I \subset [\infty]$ we let $\Delta_I$ be the discriminant of the $\{\xi_i\}_{i \in I}$. We note that if $I$ is contained in some $U_{\beta}$ then $\Delta_I$ coincides with the discriminant of the $\{\eta_i\}_{i \in I}$, since $\eta_i-\eta_j=\xi_i-\xi_j$ for $i,j \in U_{\beta}$.

\begin{lemma}
Suppose $\fa$ is an $\fY_{\cU}$-ideal of $\tilde{R}_{\cU}$ that is not contained in $\fq$. Then $\fa$ contains an element of the form
\begin{displaymath}
(1+g) \cdot \prod_{\beta \in \cK^{\infty}} \prod_{1 \le \gamma \le n(\beta)} \Delta_{I_{\beta,\gamma}}
\end{displaymath}
where $I_{\beta,1}, \ldots, I_{\beta,n(\beta)}$ are disjoint subsets of $U_{\beta}$ of cardinality $d(\beta)$ and $g$ belongs to the ideal generated by the $\eta_i$.
\end{lemma}

\begin{proof}
Let $f$ be an element of $\fa$ that does not belong to $\fq$. Express $f$ as a polynomial in the $\eta_i$'s, and let $m$ be such that $f$ only uses $\eta_1, \ldots, \eta_m$. Since $f \not\in \fa$, there must be some monomial in which each $\eta_i$ appears with exponent $<d(\pi(i))$ and with coefficient not in $\fm$. Scaling by a unit of $A$, we can assume that the coefficient is~1. Multiplying by a monomial if necessary, we can assume that $\prod_{i=1}^m \eta_i^{d(\pi(i))-1}$ appears with coefficient~1 in $f$. Write $f=\sum_k c_k \eta^k$ where the sum is over multi-indices $k \in \bN^m$. Thus $c_{\ell}=1$ where $\ell_i=d(\pi(i))-1$.

For each $1 \le i \le m$ let $I_i^{\circ}$ be a subset of $U_{\pi(i)}$ of size $d(\pi(i))-1$ such that the $I_i^{\circ}$ are disjoint from each other and from $[m]$. Let $I_i=\{i\} \cup I_i^{\circ}$. Let $f'$ be the result of skew-averaging $\Delta_{I_1^{\circ}} \cdots \Delta_{I_m^{\circ}} \cdot f$ over $\Aut(I_i)/\Aut(I_i^{\circ})$ for each $1 \le i \le m$. Since $\fa$ is an ideal and stable by $\fY_{\cU}$, it follows that $f'$ belongs to $\fa$. (Note that $\Aut(I_i)$ is a subgroup of $\fY_{\cU}$.)

We have
\begin{align*}
& \sum_{\sigma \in \fS_n/\fS_{n-1}} \sgn(\sigma) \cdot \sigma (X_n^k \cdot \Delta(X_1, \ldots, X_{n-1})) \\
=& \Delta(X_1, \ldots, X_n) \times \begin{cases}
s_{k-n+1}(X_1, \ldots, X_n) & \text{if $k \ge n-1$} \\
0 & \text{if $k<n-1$} \end{cases}
\end{align*}
where $s_k$ is a homogeneous symmetric polynomial of degree $k$. We thus see that any monomial in $f$ in which some $\eta_i$ has exponent $<d(\pi(i))-1$ will be killed by the averaging procedure. We find
\begin{displaymath}
f'=\Delta_{I_1} \cdots \Delta_{I_m} \sum_{k \in \bN^m} c_{k+\ell} s_{k_1,I_1} \cdots s_{k_m,I_m}
\end{displaymath}
where $s_{k,I}$ denotes the polynomial $s_k$ in the $\eta$ variables indexed by $I$. Since $c_{\ell}=1$ and $s_k$ is homogeneous of degree $k$ (and $s_0=1$), we see that the sum has the form $1+g$ where $g$ belongs to the ideal generated by the $\eta_i$'s. The product of discriminants also has the required form. The number $n(\beta)$ is the cardinality of $\{1,\ldots,m\} \cap U_{\beta}$. We also note that if $\pi(i)$ is in a finite part of $\cU$ then $d(\pi(i))=1$, and so $\Delta_{I_i}=1$ and can be omitted. This completes the proof.
\end{proof}

\begin{proof}[Proof of Proposition~\ref{prop:contain2-1}]
Suppose that $\fp \not\subset \fq$. We will derive a contradiction. Let
\begin{displaymath}
f= (1+g) \cdot \prod_{\beta \in \cK^{\infty}} \prod_{1 \le \gamma \le n(\beta)} \Delta_{I_{\beta,\gamma}}
\end{displaymath}
be an element as in the lemma that belongs to $\fp$. Let $\cE \subset \cI$ be a good subset such that $Y \subset \Theta^{\phi}_{\cE}(Z)$. Let $p \colon \bA^{\cI} \to \bA^{\cE}$ and $i \colon \bA^{\cK} \to \bA^{\cE}$ be the usual maps, so that $\Theta^{\phi}_{\cE}(Z)=i^{-1}(\ol{p(Z)})$. Since $i(y)$ belongs to the closure of $p(Z)$, we can find a valuation ring $B$ with fraction field $K$ and a $K$-point $z=(z_{\alpha})_{\alpha \in \cI}$ of $Z$ such that $p(z)$ extends to a $B$-point of $\bA^{\cE}$ that limits on $i(y)$. Explicitly, this means the following:
\begin{itemize}
\item The induced ring homomorphism $\theta \colon A \to B$ is local, i.e., $\theta(\fm_A) \subset \fm_B$.
\item For $\alpha \in \cE$ we have $z_{\alpha} \in B$.
\item For $\alpha \in \cE$ we have $\ol{z}_{\alpha}=\ol{\theta}(y_{\phi(\alpha)})$, where $\ol{\theta} \colon A/\fm_A \to B/\fm_B$ is the induced homomorphism and $\ol{z}_{\alpha}$ is the image of $z_{\alpha}$ in $B/\fm_B$.
\end{itemize}
For $\beta \in \cK^{\infty}$ let $\{V_{\alpha}\}_{\alpha \in \cI_{\beta}}$ be a partition of $U_{\beta}$ with the following properties:
\begin{itemize}
\item $\# V_{\alpha}=\lambda_{\alpha}$.
\item If $\eta_i$ appears in $g$, where $i \in U_{\beta}$, then $i \in \bigcup_{\alpha \in \cE_{\beta}} V_{\alpha}$.
\item We have $I_{\beta,\gamma} \subset \bigcup_{\alpha \in \cE_{\beta}} V_{\alpha}$ for each $1 \le \gamma \le n(\beta)$.
\item We have $\#(V_{\alpha} \cap I_{\beta,\gamma}) \le e_{\alpha}$ for each $1 \le \gamma \le n(\beta)$ and $\alpha \in \cE_{\beta}$.
\end{itemize}
We can pick such a partition since $\# I_{\beta,i}=d(\beta)$ and $d(\beta) \le \sum_{\alpha \in \cE_{\beta}} e_{\alpha}$ since $\cE$ is good. For $\beta \in \cK \setminus \cK^{\infty}$, let $\{V_{\alpha}\}_{\alpha \in \cE_{\beta}}$ be a partition of $U_{\beta}$ with $\# V_{\alpha}=\lambda_{\alpha}$. Thus $\cV=\{\cV_{\alpha}\}_{\alpha \in \cI}$ is a partition of $[\infty]$ of type $\lambda$ that refines $\cU$. We let $\rho \colon [\infty] \to \cI$ be associated to $\cV$. We note that if $\eta_i$ appears in $g$ then $\rho(i) \in \cE$.

Let $T_0=B[\delta_i]_{i \ge 1}/\langle \delta_i^{e(\rho(i))} \rangle$ and let $T=K \otimes_B T_0$. Let $\psi \colon R \to T$ be the ring homomorphism defined by $\psi(\xi_i)=z_{\rho(i)}+\delta_i$ and $\psi \vert_A=\theta$. This is a map of class  ${}^{\phi}_{\cU}(\lambda, e; Z)$, and so $\fp \subset \ker(\psi)$. We must therefore have $\psi(f)=0$.

The ring $T_0$ is local, with maximal ideal generated by $\fm_B$ and the $\delta_i$'s. If $\rho(i) \in \cE$ then $\psi(\eta_i)=z_{\rho(i)}-\theta(\tilde{y}_{\pi(i)})+\delta_i$. We have $\ol{z}_{\rho(i)}=\ol{\theta}(y_{\phi(\rho(i))})=\ol{\theta}(y_{\pi(i)})$ and $\ol{\theta(\tilde{y}_{\pi(i)})}=\ol{\theta}(y_{\pi(i)})$. It follows that $z_{\rho(i)}-\theta(\tilde{y}_{\pi(i)}) \in \fm_B$. Thus $\psi(\eta_i)$ belongs to the maximal ideal of $T_0$. Since every $\eta_i$ appearing in $g$ has $\rho(i) \in \cE$, it follows that $\psi(g)$ belongs to the maximal ideal of $T_0$. Thus $\psi(1+g)$ is a unit of $T_0$, and thus of $T$ as well.

Suppose $i,j \in I_{\beta,\gamma}$ belong to different parts of $\cV$. Then $\rho(i) \ne \rho(j)$ and so $z_{\rho(i)} \ne z_{\rho(j)}$ (this is where it is important that $Z$ is a subset of $\bU^{\cI}$, i.e., every point of $Z$ has distinct coordinates). We thus see that $\psi(\xi_i-\xi_j)$ is a unit of $T$. We therefore have
\begin{displaymath}
\psi(\Delta_{I_{\beta,\gamma}})=u \cdot \prod_{\alpha \in \cE_{\beta}} \psi(\Delta_{V_{\alpha} \cap I_{\beta,\gamma}})
\end{displaymath}
for some unit $u$ of $T$, and so (using the previous paragraph)
\begin{displaymath}
\psi(f) = v \cdot \prod_{\beta \in \cK^{\infty}} \prod_{1 \le \gamma \le n(\beta)} \prod_{\alpha \in \cE_{\beta}} \psi(\Delta_{V_{\alpha} \cap I_{\beta,\gamma}})
\end{displaymath}
for some unit $v$ of $T$.
Now, for $i,j \in V_{\alpha} \cap I_{\beta,\gamma}$, we have $z_{\rho(i)}=z_{\rho(j)}$, and so $\psi(\xi_i-\xi_j)=\delta_i-\delta_j$. We thus see that $\psi(\Delta_{V_{\alpha} \cap I_{\beta,\gamma}})$ is the discriminant of $\delta_i$ with $i \in V_{\alpha} \cap I_{\beta,\gamma}$. Since $\#(V_{\alpha} \cap I_{\beta,\gamma})\le e_{\alpha}$ by construction, this discriminant is a sum of monomials in the $\delta_i$'s in which every exponent is $<e_{\alpha}$. It is therefore non-zero, since these $\delta_i$ have nilpotency index $e_{\alpha}$. Since the sets $V_{\alpha} \cap I_{\beta,\gamma}$ are pairwise disjoint, it follows that the product of the $\psi(\Delta_{V_{\alpha} \cap I_{\beta,\gamma}})$ is non-zero. (Here we are simply using the fact that two non-zero elements of $T$ that use no common $\delta_i$ have non-zero product.) We thus see that $\psi(f) \ne 0$, which is a contradiction.
\end{proof}

\subsection{Proof of Theorem~\ref{thm:contain2}: the general case} \label{ss:contain-gen}

We now deduce Theorem~\ref{thm:contain2} in general from the punctual case. We use notation as in \S \ref{ss:contain2-overview}.

\begin{lemma}
Consider a cartesian diagram of affine schemes
\begin{displaymath}
\xymatrix@C=4em{ X' \ar[r]^{g'} \ar[d]_{f'} & X \ar[d]^f \\
Y' \ar[r]^g & Y }
\end{displaymath}
with $g$ flat. Let $Z$ be a closed subset of $X$. Then $g^{-1}(\ol{f(Z)})=\ol{f'((g')^{-1}(Z))}$.
\end{lemma}

\begin{proof}
Consider the corresponding diagram of rings
\begin{displaymath}
\xymatrix@C=4em{
R' & R \ar[l]_-{(g')^*} \\
S' \ar[u]^{(f')^*} & S \ar[u]_{f^*} \ar[l]^-{g^*} }
\end{displaymath}
Write $Z=V(I)$, where $I$ is an ideal of $R$. Then $g^{-1}(\ol{f(Z)})=V((I^c)^e)$ while $\ol{f'((g')^{-1}(Z))}=V((I^e)^c)$. Since $S \to S'$ is flat, we have $(I^c)^e=(I^e)^c$ (Lemma~\ref{lemma:flat-ec}), and so the result follows.
\end{proof}

\begin{proposition}
Formation of $\Theta^{\phi}$ commutes with flat base change. Precisely, let $A \to A'$ be a flat ring map, put $R'=A' \otimes_A R$, and let $Z'$ be the inverse image of $Z$ in $\bU^{\cI}_{A'}$. Then $\Theta^{\phi}(Z')$ is the inverse image of $\Theta^{\phi}(Z)$ in $\bA^{\cK}_{A'}$. (We assume that $Z$ is closed in $\bU^{\cI}_A$.)
\end{proposition}

\begin{proof}
It suffices to prove the analogous statement for $\Theta^{\phi}_{\cE}$ for each good subset $\cE$ of $\cI$. Thus let $\cE$ be given, and let $p \colon \bU_A^{\cI} \to \bA^{\cE}_A$ and $i \colon \bA^{\cK}_A \to \bA^{\cE}_A$ be the usual maps. Write $(-)'$ for base change to $A$. By the previous lemma, we have $(\ol{p(Z)})'=\ol{p'(Z')}$. Since formation of $i^{-1}$ commutes with base change, the result now follows.
\end{proof}

\begin{proposition} \label{prop:contain2-9-1}
Suppose that $A$ is a local ring with $k=A/\fm_A$ algebraically closed and $Y$ lies above the closed point of $\Spec(A)$. Then Theorem~\ref{thm:contain2} holds.
\end{proposition}

\begin{proof}
For $y \in Y(k)$, put $\fq_y= \fP^{\id}_{\cU}(\mu,d;\{y\})$. Since $Y$ is an algebraic variety over an algebraically closed field, the set $Y(k)$ is dense in $Y$. Thus, by Proposition~\ref{prop:Pphi-prop}(c), we have $\fq=\bigcap_{y \in Y(k)} \fq_y$. Hence it suffices to show that $\fp \subset \fq_y$ for all $y \in Y(k)$. This follows from the punctual case.
\end{proof}

\begin{proposition} \label{prop:contain2-9-2}
Suppose that $A$ is a local ring and $Y$ lies above the closed point of $\Spec(A)$. Then Theorem~\ref{thm:contain2} holds.
\end{proposition}

\begin{proof}
Let $A \to A'$ be a flat ring homomorphism where $A'$ is a local ring with algebraically closed residue field such that $\rad(\fm_A^e)=\fm_{A'}$. If $A$ has residue characteristic~0, one can take $A'$ to be the strict henselization of $A$ (see \stacks{0BSK}). If $A$ has residue characteristic $p>0$, one can obtain $A'$ by first taking the strict henselization and then adjoining all $p$-power roots of all elements.

Let $Y' \subset \bU^{\cK}_{A'}$ and $Z' \subset \bU^{\cI}_{A'}$ be the inverse images of $Y$ and $Z$. By Proposition~\ref{prop:flat-bc}, we have $\rad_{\fY_{\cU}}(\fp^e)=\fP^{\phi}_{\cU}(\lambda, e; Z')$ and $\rad_{\fY_{\cU}}(\fq^e)=\fP^{\id}_{\cU}(\mu, d; Y')$. Since formation of $\Theta^{\phi}$ commutes with flat base change, we have $Y' \subset \Theta^{\phi}(Z')$. Thus, by Proposition~\ref{prop:contain2-9-1}, we have $\rad_{\fY_{\cU}}(\fp^e) \subset \rad_{\fY_{\cU}}(\fq^e)$. Since contractions commute with $\fY_{\cU}$-radicals (Proposition~\ref{prop:rad-contract}), we have $\rad_{\fY_{\cU}}(\fp^{ec}) \subset \rad_{\fY_{\cU}}(\fq^{ec})$. Since $A \to A'$ is faithfully flat, we have $\fp^{ec}=\fp$ and $\fq^{ec}=\fq$. Since $\fp$ and $\fq$ are $\fY_{\cU}$-radical, we find $\fp \subset \fq$, as required.
\end{proof}

\begin{proposition}
Theorem~\ref{thm:contain2} holds for all $A$.
\end{proposition}

\begin{proof}
Let $\fc=A \cap \fq$, which is a prime ideal of $A$ by Proposition~\ref{prop:prime-contract} (since $\fS$ acts trivially on $A$). Let $A'$ be the localization of $A$ at $\fc$, let $R'=A' \otimes_A R$, and write $(-)'$ for base change to $R'$. By Proposition~\ref{prop:flat-bc}, we have $\fp'=\fP^{\phi}_{\cU}(\lambda,e;Z')$ and $\fq'=\fP^{\id}_{\cU}(\mu,d;Y')$. Since formation of $\Theta^{\phi}$ commutes with flat base change, we have $Y' \subset \Theta^{\phi}(Z')$, and so by Proposition~\ref{prop:contain2-9-2}, we have $\fp' \subset \fq'$. Proposition~\ref{prop:prime-loc} now implies that $\fp \subset \fq$. (Note that $\fq'$ is obviously not the unit ideal, and so the containment $\fp' \subset \fq'$ shows that $\fp'$ is not either; this is why Proposition~\ref{prop:prime-loc} applies.)
\end{proof}

\section{Generators} \label{s:gen}

In this section, we give finite generating sets for the ideals $\fP(\lambda,e;Z)$ up to $\fS$-radical, as discussed in \S \ref{ss:intro-gen}. After some preperatory material, we first treat the ideals $\fP(\lambda,e)$ in \S \ref{ss:gen1} before proceeding to the general case in \S \ref{ss:gen2}. Finally, in \S \ref{ss:circle}, we give a detailed example.

\subsection{The poset $\Xi$}

Let $\Xi$ be the set of isomorphism classes of pairs $(\lambda, e)$ where $\lambda$ is an $\infty$-composition (on some index set) and $e$ is a reduced weighting. Define a partial order $\preceq$ on $\Xi$ by $(\mu, d) \preceq (\lambda, e)$ if there exists a good pair $(\cE, \phi)$ between $(\mu, d)$ and $(\lambda, e)$ as in \S \ref{ss:theta}. We note that $(\mu,d) \preceq (\lambda,e)$ implies that $\mu \preceq \lambda$ as defined in \S \ref{ss:radical}. The motivation for this definition comes from the $\Theta$ construction. In particular, we have the following:

\begin{proposition} \label{prop:order-contain}
We have $(\mu,d) \preceq (\lambda,e)$ if and only if $\fP(\lambda,e) \subset \fP(\mu,d)$.
\end{proposition}

\begin{proof}
Suppose $(\mu,d) \preceq (\lambda,e)$ and let $(\cE,\phi)$ be a good pair between them. Then $\Theta^{\phi}_{\cE}(\bU^{\cI})=\bU^{\cK}$, where $\cI$ is the index set of $\lambda$ and $\cK$ of $\mu$. We thus see that $\bU^{\cK} \subset \Theta^{\lambda,e}_{\mu,d}(\bU^{\cI})$, and so $\fP(\lambda,e) \subset \fP(\mu,d)$ by Theorem~\ref{thm:contain}.

Now suppose that $\fP(\lambda,e) \subset \fP(\mu,d)$. Then by Theorem~\ref{thm:contain} we have $\bU^{\cK} \subset \Theta^{\lambda,e}_{\mu,d}(\bU^{\cI})$. In particular, $\Theta^{\lambda,e}_{\mu,d}(\bU^{\cI})$ is non-empty, which implies there is a good pair between $(\mu,d)$ and $(\lambda,e)$, and so $(\mu,d) \preceq (\lambda,e)$.
\end{proof}

Recall that a partially ordered set is a \defi{well-quasi-order} (wqo) if in any sequence of elements $x_1, x_2, \ldots$ there exists $i<j$ such that $x_i \le x_j$.

\begin{proposition} \label{prop:wqo}
$\Xi$ is a wqo.
\end{proposition}

\begin{proof} 
Let $\Sigma$ be the set of pairs $(x,y)$ where $x \in \bZ_{\ge 1} \cup \{\infty\}$ and $y \in \bZ_{\ge 1}$ such that $y=1$ if $x \ne \infty$, and order $\Sigma$ lexicographically (comparing the first coordinate first). One easily sees that $\Sigma$ is a wqo. Let $\Sigma^{\star}$ be the set of words in the alphabet $\Sigma$. Give $\Sigma^{\star}$ the Higman order: $a_1 \cdots a_n \le b_1 \cdots b_m$ if there exists an order-preserving injection $f \colon [n] \to [m]$ such that $a_i \le b_{f(i)}$ for all $i \in [n]$. It follows from Higman's lemma \cite[Theorem~4.3]{higman} that the set $\Sigma^{\star}$ is a wqo. Now, suppose that $(\lambda^i, e^i)$ is a sequence in $\Xi$, for $i \ge 1$. We may assume that the index set of $\lambda^i$ has the form $[n_i]$ for some $i$, and thereby regard $(\lambda^i, e^i)$ as a word of length $n_i$ on $\Sigma$. Since $\Sigma^{\star}$ is a wqo, there exists $i<j$ such that $(\lambda^i,e^i) \le (\lambda^j,e^j)$ in the Higman order. One easily sees that this implies $(\lambda^i,e^i) \preceq (\lambda^j,e^j)$, which completes the proof.
\end{proof}	

\subsection{A preliminary result}
Fix an $\infty$-composition $\lambda$ on index set $\cI$, a weighting $e$ on $\cI$, and an irreducible closed subset $Z$ of $\bU^{\cI}$. We will require the following result below.

\begin{proposition} \label{prop:h-factors}
Let $h \in R$ be a product of factors of the form $\xi_i-\xi_j$, and suppose that $h$ belongs to $\fP(\lambda,e;Z)$. Let $p$ be the residue characteristic of the generic point of $Z$.
\begin{enumerate}
\item If $p=0$ then $h$ belongs to $\fP(\lambda,e)$.
\item If $p>0$ then $h$ belongs to $\fP(\lambda,e;\bU^{\cI}_{A \otimes \bF_p})$.
\end{enumerate}
\end{proposition}

\begin{proof}
(a) Let $\cU$ be partition of $[\infty]$ of type $\lambda$. Let $\fe$ be the ideal of $\bZ[\epsilon_i]$ generated by $\epsilon_i^{e(\alpha)}$ for $i \in U_{\alpha}$ and $\alpha \in \cI$, and let $B=\bZ[\epsilon_i]/\fe$, which is free as a $\bZ$-module. Then, as in the proof of Proposition~\ref{prop:Qunion}, we have
\begin{displaymath}
\tilde{S}_{\cU}/\fQ_{\cU}(\lambda,e) = \tilde{S}_{\cU,0} \otimes_{\bZ} B, \qquad
\tilde{S}_{\cU}/\fQ_{\cU}(\lambda,e;Z) = \tilde{S}_{\cU,0}/\fI_0(Z) \otimes_{\bZ} B.
\end{displaymath}
Since $\fI_0(Z)$ has residue characteristic~0, it follows that the natural map $B \to \tilde{S}_{\cU,0}/\fI_0(Z) \otimes B$ is injective. Now, write $h=h_1 h_2$, where $h_1$ (resp.\ $h_2$) is a product of factors $\xi_i-\xi_j$ where $i$ and $j$ belong to the same (resp.\ different) parts. Then $\iota(h_1)$ is a product of factors of the form $\epsilon_i-\epsilon_j$, and can thus be regarded as an element of the ring $\bZ[\epsilon_i]$, while $\iota(h_2)$ is a unit modulo $\tilde{J}_{\cU,n}$ for any $n \ge 0$. Since $h$ belongs to $\fP(\lambda,e;Z)$, it follows that the image of $\iota(h)$ in $\tilde{S}_{\cU,0}/\fI_0(Z) \otimes B$ vanishes. Since $\iota(h_2)$ is a unit of this ring, it follows that $\iota(h_1)$ maps to~0 in $B$, and so $\iota(h_1)$ maps to~0 in $\tilde{S}_{\cU}/\fQ_{\cU}(\lambda,e)$. We thus see that $h$ belongs to the contraction of $\fQ_{\cU}(\lambda,e)$. Since this holds for all $\cU$, it follows that $h$ belongs to the $\fS$-contraction of $\fQ_{\cU}(\lambda,e)$, which is $\fP(\lambda,e)$.

(b) The argument is similar: the main difference is that we can only conclude that $\iota(h_1)$ maps to~0 in $B/pB$.
\end{proof}

\subsection{The ideals $\fP(\lambda,e)$} \label{ss:gen1}

We now describe generators for the ideals $\fP(\lambda,e)$ up to $\fS$-radical. The following is the key construction:

\begin{construction} \label{con:gen1}
Fix $(\lambda,e) \in \Xi$.
\begin{itemize}
\item Let $\Psi \subset \Xi$ be the set of elements $(\mu,d) \in \Xi$ satisfying $(\mu,d) \npreceq (\lambda,e)$, and let $\Psi_0$ be the set of minimal elements in $\Psi$. The set $\Psi_0$ is finite since $\Xi$ is a wqo (Proposition~\ref{prop:wqo}).
\item For each $(\mu,d) \in \Psi_0$, let $h_{\mu,d}$ be the element $h$ produced by Construction~\ref{con:h} with $\fp=\fP(\lambda,e)$ and $\fq=\fP(\mu,d)$. We note that the result of the construction is independent of the point $y$, since $h_3=1$ as there are no good pairs $(\cE,\phi)$.
\item Let $\sG = \{h_{\mu, d} \colon (\mu, d) \in \Psi_0 \}$.
\end{itemize}
The set $\sG$ is the result of the construction. It is a finite subset of $R$.
\end{construction}

\begin{theorem} 
\label{thm:generators1}
Let $\sG$ be the set produced by Construction~\ref{con:gen1}. Then the $\fS$-radical of the $\fS$-ideal generated by $\sG$ is $\fP(\lambda,e)$. 
\end{theorem}	

Before proving the theorem, we need a few lemmas.

\begin{lemma}
Let $(\mu,d) \in \Psi_0$. Then $h_{\mu,d} \not\in \fP(\mu,d;Y)$ for any non-empty $Y$.
\end{lemma}

\begin{proof}
Since $(\mu,d) \npreceq (\lambda,e)$, there are no good pairs $(\cE,\phi)$ between $(\mu,d)$ and $(\lambda,e)$, and so the $h_3=1$ in Construction~\ref{con:h}. We thus see that $h_{\mu,d}$ is a possible output of Construction~\ref{con:h} with $\fp=\fP(\lambda,e)$ and $\fq=\fP(\mu,d;Y)$, and so $h_{\mu,d} \not\in \fP(\mu,d;Y)$ by Lemma~\ref{lem:not-in-fq}.
\end{proof}

\begin{lemma} \label{lem:nonzero-minimal}
Let $(\mu,d) \in \Psi_0$ and let $(\nu,f) \in \Xi$ satisfy $(\mu,d) \preceq (\nu,f)$. Let $\cL$ be the index set of $\nu$ and let $W \subset \bU^{\cL}$ be non-empty. Then $h_{\mu,d} \not\in \fP(\nu,f;W)$.
\end{lemma}

\begin{proof}
Suppose by way of contradiction that $h$ belongs to $\fP(\nu,f;W)$; replacing $W$ with an irreducible component of its closure, we can assume that $W$ is irreducible and closed. First suppose that the generic point of $W$ has residue characteristic~0. Since $h_{\mu,d}$ is a product of factors of the form $\xi_i-\xi_j$, Proposition~\ref{prop:h-factors} shows that $h_{\mu,d}$ belongs to $\fP(\nu,f)$. Since $(\mu,d) \preceq (\nu,f)$, Proposition~\ref{prop:order-contain} now shows that $h_{\mu,d}$ belongs to $\fP(\mu,d)$, contradicting the previous lemma.

Now suppose the generic point of $W$ has characteristic $p>0$. The argument is similar. Proposition~\ref{prop:h-factors} shows that $h_{\mu,d}$ belongs to $\fP(\nu,f;\bU^{\cL}_{A \otimes \bF_p})$. Since $(\mu,d) \preceq (\nu,f)$, we see that $\Theta^{\nu,f}_{\mu,d}(\bU^{\cL}_{A \otimes \bF_p})=\bU^{\cK}_{A \otimes \bF_p}$, where $\cK$ is the index set of $\mu$. Theorem~\ref{thm:contain} thus shows $\fP(\nu,f;\bU^{\cL}_{A \otimes \bF_p}) \subset \fP(\mu,d;\bU^{\cK}_{A \otimes \bF_p})$, again contradicting the previous lemma. (Note that $A \otimes \bF_p$ is not the zero ring due to the existence of $W$.)
\end{proof}

\begin{lemma} \label{lem:G-disc}
The set $\sG$ contains a product of discriminants in disjoint variables. In particular, if $\fq$ is an $\fS$-prime containing $\sG$ then $\fq$ is finitary.
\end{lemma}

\begin{proof}
Let $c = \sum_{\alpha \in \cI^{\infty}} e_{\alpha} $. Then, by the definition of a good pair, we have $((\infty), (c)) \preceq (\lambda, e)$ but $((\infty),(c+1))  \npreceq (\lambda, e)$. It follows that $((\infty),(c+1))  \in \Psi_0$. Examining Construction~\ref{con:h}, we see that $h_{(\infty),(c+1)}$ is a disjoint product of discriminants (we have $h_2 = h_3 = 1$).

Suppose $\fq$ contains $\sG$. Then $\fq$, by the above,  contains a disjoint product of discriminants. Since a disjoint product of discriminants divides some larger discriminant, we see that $\fq$ contains $\Delta_n$ for some $n$. Thus $V(\fq)$ consists of finitary points, and so $\fq$ is finitary by Proposition~\ref{prop:finitary-char}.
\end{proof}

\begin{proof}[Proof of Theorem~\ref{thm:generators1}]
Let $\fa$ be the $\fS$-radical of the $\fS$-ideal generated by $\sG$ and let $\fp=\fP(\lambda,e)$. We must show $\fa=\fp$. By Lemma~\ref{lem:in-fp}, we have $h_{\mu,d} \in \fp$ for all $(\mu,d) \in \Psi_0$, and so $\sG \subset \fp$, and so $\fa \subset \fp$. We now prove the reverse inclusion. By Proposition~\ref{prop:rad-primes}, $\fa$ is the intersection of all $\fS$-primes containing $\fa$. Let $\fq$ be an $\fS$-prime containing $\sG$. It suffices to show that $\fp \subset \fq$; suppose by way of contradiction that this does not hold. By Lemma~\ref{lem:G-disc}, $\fq$ is a finitary prime, say $\fq = \fP(\nu, f; W)$. Since $\fp \not\subset \fq$ and $\fP(\nu,f) \subset \fq$, we have $\fp \not\subset \fP(\nu,f)$, and so $(\nu, f) \npreceq (\lambda, e)$ by Proposition~\ref{prop:order-contain}; that is $(\nu,f) \in \Psi$. Let $(\mu, d) \in \Psi_0$ be such that $(\mu, d) \preceq (\nu, f)$. By Lemma~\ref{lem:nonzero-minimal}, $h_{\mu, d} \notin \fq$. This is a contradiction as $h_{\mu,d} \in \sG \subset \fp$.
\end{proof}	

\begin{example} \label{ex:e2} 
Let $\fp=\fP((\infty),(2))$. Then $\Psi_0$ consists of the two pairs $((\infty), (3))$ and $((\infty,1),(1,1))$. We thus find that $\sG$ consists of $\Delta_3$ and $h=(\xi_1-\xi_2)^3$. Let $\fa$ be the $\fS$-ideal generated by these elements. Thus Theorem~\ref{thm:generators1} asserts that $\fp=\rad_{\fS}(\fa)$. In characteristic~0, we know from Theorem~\ref{thm:contract} that $\fp$ is in fact generated by the $\fS$-orbit of $h$, and so $\fp=\fa$ in this case.

Suppose now that $A$ is a field of characteristic~2. It follows from Theorem~\ref{thm:contract} that $\fp$ is generated by the $\fS$-orbit of $g=(\xi_1-\xi_2)^2$. Thus Theorem~\ref{thm:generators1} implies that $g$ belongs to the $\fS$-radical of $\fa$. In fact, we claim $g \cdot \sigma g \in \fa$ any $\sigma \in \fS$. There are three cases to check:
\begin{itemize}
\item \textit{Case 1: $\sigma$ preserves $\{1,2\}$.} We find that $g \cdot \sigma g$ is a multiple of $h$.
\item \textit{Case 2: $\sigma( \{1,2\}) = \{i , j\}$ where $i \in \{1,2\}$ and $j \notin \{1,2\}$.} We see
\begin{displaymath}
g \cdot \sigma g = (\xi_1 + \xi_2)\Delta_{\{1,2, j\}} + (\xi_i + \xi_j)h.
\end{displaymath}
\item \textit{Case 3: $\sigma( \{1,2\})$ does not intersect $\{1,2\}$.} In this case, we have
\begin{displaymath}
g \cdot \sigma g = (\xi_1 + \xi_2)(\Delta_{\{1,2, \sigma(1)\}} + \Delta_{\{1,2, \sigma(2)\}}) + (\xi_{\sigma(1)} + \xi_{\sigma(2)})h.
\end{displaymath}
\end{itemize}
This establishes the claim.
\end{example}

\begin{example} \label{ex:2-2} 
Let $\fp=\fP((\infty, \infty),(2,2))$. Then $\Psi_0$ consists of the three pairs
\begin{displaymath}
((\infty),(5)), \quad ((\infty,1), (3,1)), \quad ((\infty,1,1),(1,1,1)).
\end{displaymath}
In this case, $\sG$ consists of the following three elements:
\begin{displaymath}
h_1 = \Delta_5, \qquad
h_2 = \Delta_3 \prod_{i=1}^3(\xi_i - \xi_4)^3, \qquad
h_3 = (\Delta_3)^3.
\end{displaymath}
We thus see that $\fp$ is the $\fS$-radical of the $\fS$-ideal generated by $h_1$, $h_2$, and $h_3$.
\end{example}

\subsection{The ideals $\fP(\lambda,e;Z)$} \label{ss:gen2}

Assume $A$ is noetherian. Let $\lambda$ be an $\infty$-composition on the index set $\cI$, let $e$ be a reduced weighting on $\cI$, and let $Z$ be a closed subset of $\bU^{\cI}$. Our goal is to give a generating set of $\fp=\fP(\lambda,e;Z)$ up to $\fS$-radical.

\begin{construction} \label{con:gen2-1}
Let $(\mu,d) \in \Xi$ satisfy $(\mu,d) \preceq (\lambda,e)$.
\begin{itemize}
\item For each subset $\cE$ of $\cI$, fix a finite generating set $\sD_{\cE} \subset A[t_{\alpha}]_{\alpha \in \cE}$ of the ideal of $\ol{p(Z)} \subset \bA^{\cE}$. This is possible since $A$ is noetherian.
\item Let $\Sigma$ be the finite set of all functions $\ul{u}$ that assign to a good pair $(\cE,\phi)$ between $(\mu,d)$ and $(\lambda,e)$ an element $\ul{u}_{\cE,\phi}$ of $\sD_{\cE}$.	
\item For a good pair $(\cE, \phi)$ between $(\mu, d)$ and $(\lambda, e)$ and $u_{\cE,\phi} \in \sD_{\cE}$, we let $g_{\cE,\phi}(u_{\cE,\phi})$ be the result of Construction~\ref{con:g} (we ignore the condition in that construction about non-vanishing at $i(y)$).
\item For $\ul{u} \in \Sigma$, we let $h_{\ul{u}}$ be the result of Construction~\ref{con:h}, where the choices of $u$'s in the construction are specified by $\ul{u}$; precisely, we take $h_3=\prod g_{\cE,\phi}(\ul{u}_{\cE,\phi})$ where the product is over all good pairs $(\cE, \phi)$.
\item Finally, we let $\sH^{\mu,d}=\{ h_{\ul{u}} : \ul{u} \in \Sigma \}$.
\end{itemize}
The set $\sH^{\mu,d}$ is the result of the construction. It is a finite subset of $R$.
\end{construction}

\begin{construction} \label{con:gen2}
Define the following quantities:
\begin{itemize}
\item Set $n=1+\sum_{\alpha \in \cI^{\rf}} \lambda_{\alpha}$.
\item Let $\Phi$ be the finite subset of $\Xi$ consisting of pairs $(\mu, d)$ such that
\begin{itemize}
\item $(\mu,d) \preceq (\lambda,e)$
\item any finite part of $\mu$ is at most $n$, and 
\item $\Theta^{\lambda, e}_{\mu, d}(Z)$ is a proper closed subset of $\bU^{\cK}$ where $\cK$ is the index set of $\mu$.
\end{itemize}
\item Let $\sH = \bigcup_{(\mu, d) \in \Phi} \sH^{\mu, d}$.
\end{itemize}
The set $\sH$ is the result of the construction. It is a finite subset of $R$.
\end{construction}
	
\begin{theorem} \label{thm:generators2}
Let $\sG$ be as in Construction~\ref{con:gen1} and let $\sH$ be as in Construction~\ref{con:gen2}. Then the $\fS$-radical of the $\fS$-ideal generated by $\sG \cup \sH$ is $\fp$. 
\end{theorem}	

\begin{proof}
Let $\fa$ be the $\fS$-radical of the $\fS$-ideal generated by $\sG \cup \sH$. We must show $\fa=\fp$. We have shown in Theorem~\ref{thm:generators1} that $\sG \subset \fP(\lambda,e)$, and so $\sG \subset \fp$. It follows from Lemma~\ref{lem:in-fp} that $\sH \subset \fp$ (note that in that lemma, the non-vanishing of $u_{\cE,\phi}$ at $i(y)$ is not used). We thus see that $\fa \subset \fp$. We now prove the reverse inclusion. By Proposition~\ref{prop:rad-primes}, $\fa$ is the intersection of all $\fS$-primes containing $\fa$. Let $\fq$ be an $\fS$-prime containing $\sG \cup \sH$. It suffices to show that $\fp \subset \fq$.

Since $\sG \subset \fq$, we have $\fP(\lambda,e) \subset \fq$ by Theorem~\ref{thm:contain}. In particular, $\fq$ is finitary, say $\fq = \fP(\nu, d; Y)$. By Theorem~\ref{thm:contain}, it suffices to show that $Y \subset \Theta^{\lambda,e}_{\nu,d}(Z)$. From the containment $\fP(\lambda,e) \subset \fq$, we find that $(\mu,d) \preceq (\lambda,e)$. Let $\mu$ be the $\infty$-composition obtained from $\nu$ by reducing each finite part greater than $n$ to $n$. Clearly, $(\mu,d) \preceq (\nu,d)$. By construction, if $(\cE, \phi)$ defines a good pair between $(\nu, d)$ and $(\lambda, e)$ then it also defines a good pair between $(\mu, d)$ and $(\lambda, e)$, and vice versa. Therefore, we have $\Theta^{\lambda, e}_{\nu, d}(Z)  = \Theta^{\lambda, e}_{\mu, d}(Z)$, and so it is enough to show $Y \subset \Theta^{\lambda,e}_{\mu,d}(Z)$.
 
Suppose by way of contradiction that $Y \not\subset \Theta^{\lambda,e}_{\mu,d}(Z)$. Pick $y \in Y$ with $y \not\in \Theta^{\lambda,e}_{\mu,d}(Z)$, and use notation as in Construction~\ref{con:gen2-1}. For each good pair $(\cE, \phi)$ between $(\mu, d)$ and $(\lambda, e)$, pick $\ul{u}_{\cE, \phi} \in \sD_{\cE}$ that does not vanish on $i(y) \in \bA^{\cE}$. By Lemma~\ref{lem:not-in-fq}, we see that $h_{\ul{u}} \notin \fP(\mu,d;Y)$.  As $(\mu, d) \in \Phi$, we have $h_{\ul{u}} \in \sH$. But this contradicts $\fP(\mu,d;Y) \supset \fq \supset \sH$, finishing the proof.
\end{proof}	

\subsection{An example} \label{ss:circle}

Let $\lambda=(\infty,\infty)$ and $e=(2,2)$. Identify the index set of $\lambda$ with $\{1,2\}$, and  let $Z \subset \bU^{2}$ be the variety defined by $t_1^2 + t_2^2=1$. We now construct an explicit finite set that generates $\fp = \fP(\lambda, e; Z)$ up to $\fS$-radical. For this, it suffices to find $\sG$ and $\sH$ as in Theorem~\ref{thm:generators2}. By Example~\ref{ex:2-2}, we may take $\sG$ to be the set consisting of:
\begin{displaymath}
\Delta_5, \qquad
\Delta_3 \prod_{i=1}^3(\xi_i - \xi_4)^3, \qquad
\Delta_3^3. 
\end{displaymath}
Now we find $\sH$. For this, note that $n=1$, and so $(\mu, d) \in \Phi$ if and only if $(\mu, d)$ is of one of the following forms:
\begin{enumerate}
\item $((\infty),(a))$ with $3 \le a \le 4$. (Note that if $a \le 2$ then $\Theta^{\lambda, e}_{\mu, d}(Z)$ is not proper.)
\item $((\infty, 1),(a,1))$ with $ 1 \le a \le 2$.
\item $((\infty, \infty), (a,b))$ with $1 \le a,b \le 2$.
\end{enumerate}
Now suppose $(\mu, d) \in \Phi$ is as above.  If $(\cE, \phi)$ is a good pair between $(\mu, d)$ and $(\lambda, e)$, then we must have $\cE = \cI$. We take $\sD_{\cE}$ to be the singleton consisting of $t_1^2 + t_2^2 - 1$. Also, we have  $e_{\max} = 2$ and $N = 2e_{\max} - 1 = 3$. We now construct $\sH^{\mu, d}$ in each of the three cases above as follows:

\textit{Case (a).}  In this case, we have $\tau = (a)$, and so $m = a$. Set $\ol{U}_1 = [a]$. There is a unique good pair given by $(\cE, \phi) = (\{1,2\}, \phi \colon \{1,2\} \to \{1\})$,  and $\Pi'_{\cE, \phi}$ can be identified with partitions $\ol{U}_1 = \ol{V}_1 \cup \ol{V}_2$ where each of $\ol{V}_1$ and $\ol{V}_2$ has at most two elements. In particular, there are six such partition for each $a \in \{3,4\}$, and one can easily construct a bijection \[ \zeta \colon \{(i,j) \colon 1 \le i, j \le 3, i \ne j  \} \to \Pi'_{\cE, \phi}  \] such that $\zeta(i,j) = (\ol{V}_1, \ol{V}_2)$ implies $i \in \ol{V}_1$ and $j \in \ol{V}_2$.  We have $u_{\cE, \phi}= t_1^2 + t_2^2 - 1$. And for $\cV = (V_1, V_2) \in \Pi'_{\cE, \phi}$, we take $u_{\cE, \phi, \cV}$ to be $\xi_i^2 + \xi_j^2 - 1$ where $\zeta(i,j) = (\ol{V}_1, \ol{V}_2)$. By Remark~\ref{rem:optimize},  we may take $\sH^{(\infty),(a)}$ to be the set consisting of the one element
\begin{displaymath}
\Delta_a \prod_{1 \le i < j \le 3} (\xi_i^2 + \xi_j^2 - 1 )^{4}. 
\end{displaymath}
	
\textit{Case (b).}  In this case, we have $\tau = (a, 1)$, and so $m = a+1$. Set $\ol{U}_1 = [a]$ and $\ol{U}_2 = \{a+1\}$. A pair $(\cE, \phi)
	= (\{1,2\}, \phi \colon \{1,2\} \to \{1, 2\})$  is good if $\phi$ is surjective. So $\phi$ could be either the identity map $\id$ or the the map $\sigma$ that switches $1$ and $2$. In each case, we have $u_{\cE, \phi} = t_1^2 + t_2^2 - 1$. There is a unique element in $\Pi'_{\cE, \id}$ which corresponds to $(\ol{V}_1, \ol{V}_2) = (\ol{U}_1, \ol{U}_2)$, and so $g_{\cE, \id}$ can be taken to be $(\xi_1^2 + \xi_{a+1}^2 - 1)^{4}$. There is a unique element in $\Pi'_{\cE, \id}$ as well which corresponds to $(\ol{V}_1, \ol{V}_2) = (\ol{U}_2, \ol{U}_1)$, and so $g_{\cE, \sigma}$ can again be taken to be $(\xi_1^2 + \xi_{a+1}^2 - 1)^{4}$.  By Remark~\ref{rem:optimize},  we may take $\sH^{(\infty, 1),(a,1)}$ to be the set consisting of the one element
\begin{displaymath}
\Delta_a \left(\prod_{i \in [a]} (\xi_i - \xi_{a+1})^3\right) (\xi_1^2 + \xi_{a+1}^2 - 1)^{4}.
\end{displaymath}

\textit{Case (c).} In this case, we have $\tau = (a, b)$, and so $m = a+b$. Set $\ol{U}_1 = \{1, \ldots, a\}$, and  $\ol{U}_2 = \{a+1, \ldots, b\}$. A pair $(\cE, \phi)
	= (\{1,2\}, \phi \colon \{1,2\} \to \{1, 2\})$  is good if $\phi$ is surjective. So $\phi$ could be either the identity map $\id$ or the the map $\sigma$ that switches $1$ and $2$. We see that $\Pi'_{\cE, \phi}$ are exactly as in Case (b) above, and we may take $g_{\cE, \phi}$ to be $(\xi_1^2 + \xi_{a+b}^2 - 1)^{4}$ for each $\phi$. By Remark~\ref{rem:optimize}, we may take $\sH^{(\infty, \infty),(a,b)}$ to be the set  consisting of the one element
\begin{displaymath}
\Delta_a \Delta_{\{a+1, \ldots, a+b \}} \left(\prod_{1 \le i \le a < j \le a+b} (\xi_i - \xi_j)^3\right) (\xi_1^2 + \xi_{a+b}^2 - 1)^4.
\end{displaymath}

By Theorem~\ref{thm:generators2}, the collection of all the elements constructed above generate $\fp$ up to $\fS$-radicals. Ignoring redundancies, we conclude that the following five elements
\begin{displaymath}
\Delta_5, \quad
\Delta_3 \prod_{i=1}^3(\xi_i - \xi_4)^3, \quad
\Delta_3^3, \quad
\Delta_3 \prod_{1 \le i < j \le 3} (\xi_i^2 + \xi_j^2 - 1 )^{4}, \quad
(\xi_1 - \xi_{2})^3 (\xi_1^2 + \xi_{2}^2 - 1)^{4}
\end{displaymath}
generate $\fp$ up to $\fS$-radicals.

\section{The equivariant spectrum} \label{s:spec}

In this section, we examine the $\fS$-spectrum of $R$. We assume $A$ is noetherian throughout.

\subsection{The $\fS$-spectrum}

Let $\Spec_{\fS}(R)$ be the $\fS$-spectrum of $R$, as defined in \S \ref{ss:Gspec}. Recall that this is the set of all $\fS$-prime ideals of $R$, endowed with the Zariski topology. The goal of \S \ref{s:spec} is to give a direct description of $\Spec_{\fS}(R)$ in terms of finite dimensional algebraic varieties.

Our description of $\Spec_{\fS}(R)$ is essentially a repackaging of the main theorems of this paper. Thus the results here are not substantially new. For this reason, we omit some technical details from the proofs, especially when similar technicalities have already been treated thoroughly in \cite{svar}.

\subsection{The set $\sX$}

Let $\lambda$ be an $\infty$-composition on an index set $\cI$ and let $e$ be a reduced weighting on $\cI$. We put
\begin{displaymath}
\sX_{\lambda,e}=\bU^{\cI}, \qquad
\sX_{[\lambda,e]}=\sX_{\lambda,e}/\Aut(\lambda,e).
\end{displaymath}
We simply regard $\sX_{[\lambda,e]}$ as a set: it is the quotient of the set of scheme-theoretic points of $\bU^{\cI}$ by the finite group $\Aut(\lambda,e)$. We note that if $(\lambda,e)$ and $(\mu,d)$ are isomorphic then $\sX_{[\lambda,e]}$ is canonically identified with $\sX_{[\mu,d]}$. We put
\begin{displaymath}
\sX_{\fin} = \coprod_{(\lambda,e)} \sX_{[\lambda,e]}, \qquad
\sX_{\infty} = \Spec(A), \qquad
\sX = \sX_{\fin} \amalg \sX_{\infty},
\end{displaymath}
where the first disjoint union is taken over all isomorphism classes of pairs $(\lambda,e)$. We let $\pi_{\lambda,e} \colon \sX_{\lambda,e} \to \sX_{[\lambda,e]} \subset \sX$ be the quotient map, and we let $\pi_{\infty} \colon \sX_{\infty} \to \sX$ be the inclusion. For a subset $\sZ$ of $\sX$, we put $\sZ_{\lambda,e}=\pi_{\lambda,e}^{-1}(\sZ)$ and $\sZ_{\infty}=\pi_{\infty}^{-1}(\sZ)$.

\begin{remark}
As discussed in the introduction \S \ref{ss:intro-spec}, one should picture $\sX_{\lambda,e}$ as the configuration space of $\# \cI$ distinct points on the affine line $\bA^1$, where the points have been labeled by $\cI$ and the point labeled by $\alpha \in \cI$ has been assigned multiplicity $\lambda_{\alpha}$ and weight $e_{\alpha}$. The space $\sX_{[\lambda,e]}$ is similar except the labeling has been discarded.
\end{remark}

\subsection{The map $\rho$}

We now define a natural map
\begin{displaymath}
\rho \colon \sX \to \Spec_{\fS}(R).
\end{displaymath}
Let $x \in \sX_{[\lambda,e]}$ be given. Let $y \in \sX_{\lambda,e}$ be a lift of $x$ let $Z \subset \bU^{\cI}$ be its Zariski closure. If $y'$ is a second lift with closure $Z'$ then $y'=\sigma y$ for some $\sigma \in \Aut(\lambda,e)$, and so $Z'=\sigma Z$. It follows that the data $(\lambda,e;Z)$ is isomorphic to the data $(\lambda,e;Z')$ via $\sigma$, and so $\fP(\lambda,e;Z)=\fP(\lambda,e;Z')$. Hence this $\fS$-prime depends only on $x$. We define $\rho(x)=\fP(\lambda,e;Z)$. This defines $\rho$ on $\sX_{\fin}$. For $x \in \sX_{\infty}$ corresponding to a prime $\fc \subset A$, we let $\rho(x)$ be the $\fS$-prime $\fc R$. This completes the definition of $\rho$.

\begin{proposition}
The function $\rho$ is a bijection.
\end{proposition}

\begin{proof}
This is exactly the classification theorem of $\fS$-primes (Theorem~\ref{thm:class} in the finitary case and Proposition~\ref{prop:non-finitary} in the non-finitary case).
\end{proof}

\subsection{The $\Theta$-topology} \label{ss:Theta-top}

Our goal is to now describe an explicit topology on $\sX$ that will make $\rho$ a homeomorphism. Let $\sZ$ be a subset of $\sX$. We say that $\sZ$ is \defi{$\Theta$-closed} if the following conditions hold:
\begin{enumerate}
\item[(D1)] $\sZ_{\lambda,e}$ is Zariski closed in $\sX_{\lambda,e}$ for all $(\lambda,e)$, and $\sZ_{\infty}$ is Zariski closed in $\sX_{\infty}$.
\item[(D2)] We have $\Theta^{\lambda,e}_{\mu,d}(\sZ_{\lambda,e}) \subset \sZ_{\mu,d}$ for all $(\lambda,e)$ and $(\mu,d)$.
\item[(D3)] Let $(\mu,d)$ be given on the index set $\cK$. Suppose that $\cK^{\infty}$ has more than one element, and let $\alpha \in \cK^{\infty}$ satisfy $d_{\alpha}=1$. Let $\mu(n)$ be the $\infty$-composition on $\cK$ with $\mu(n)_{\alpha}=n$ and $\mu(n)_{\beta}=\mu_{\beta}$ for $\beta \ne \alpha$. Then $\sZ_{\mu,d}=\sZ_{\mu(n),d}$ for $n \gg 0$.
\item[(D4)] Let $Y_n \subset \Spec(A)$ be the image of $\sZ_{(\infty^n,1^n)}$ under the structure map $\bU^n_A \to \Spec(A)$. Then $\sZ_{\infty}=\bigcap_{n \ge 1} Y_n$.
\end{enumerate}
We make two comments concerning these conditions. First, suppose $\sZ$ satisfies (D1) and (D2). With notation as in (D3), we have $\sZ_{\mu,d} \subset \sZ_{\mu(n+1),d} \subset \sZ_{\mu(n),d}$ for all $n$; this follows from (D2). Thus the $\sZ_{\mu(n),d}$ form a descending chain in $\bU^{\cK}$, which must stabilize since $\bU^{\cK}$ is noetherian. The (D3) condition says that $\sZ_{\mu,d}$ is the stable value. We could equivalently phrase (D3) as $\sZ_{\mu,d}=\bigcap_{n \ge 1} \sZ_{\mu(n),d}$. Second, one can show that if $\sZ$ is $\Theta$-closed then $\sZ_{\lambda,e}$ contains the inverse image of $\sZ_{\infty}$ under the structure map $\sX_{\lambda,e} \to \sX_{\infty}=\Spec(A)$ for all $(\lambda, e)$.

\begin{proposition} \label{prop:Theta-top}
We have the following:
\begin{enumerate}
\item A finite union of $\Theta$-closed sets is $\Theta$-closed.
\item An arbitrary intersection of $\Theta$-closed sets is $\Theta$-closed.
\item The sets $\emptyset$ and $\sX$ are $\Theta$-closed.
\end{enumerate}
\end{proposition}

\begin{proof}
(a) and (c) are clear, as is (b) for finite intersections. For infinite intersections, use the comments preceding the proposition and the fact that $\bU^{\cK}$ is noetherian.
\end{proof}

We thus see that the $\Theta$-closed sets of $\sX$ form the closed sets in a topology, which we call the \defi{$\Theta$-topology} on $\sX$.

\begin{remark}
The $\Theta$-topology admits a geometric interpretation in terms of configuration spaces, as discussed in the \S \ref{ss:intro-spec}.
\end{remark}

\begin{remark}
In \cite{svar}, we described the radical $\fS$-ideals of $R$ in terms of ``C3 subvarieties'' of a space $\cX$. This is very closely related to the idea of $\Theta$-closed subsets of $\sX$. The biggest difference between the two notions is that the pieces of $\cX$ are indexed by just $\infty$-compositions $\lambda$: there are no weightings $e$ in the picture. There are a number of minor differences too, e.g., we defined $\cX_{\lambda}$ to be $\bA^{\cI}$ whereas here we define $\sX_{\lambda,e}$ to be $\bU^{\cI}$. These changes were made here to make the description of $\Spec_{\fS}(R)$ as clear as possible.
\end{remark}

\subsection{The main theorem}

We can now state our main result on the $\fS$-spectrum of $R$:

\begin{theorem} \label{thm:rho-homeo}
The map $\rho \colon \sX \to \Spec_{\fS}(R)$ is a homeomorphism when the domain is equipped with the $\Theta$-topology and the target with the Zariski topology.
\end{theorem}

We require a few lemmas before proving the theorem. For a pair $(\lambda,e)$ on index set $\cI$ and a subset $Z$ of $\bU^{\cI}$, we let $\Gamma_{\lambda,e}(Z)$ be the $\Theta$-closure of $\pi_{\lambda,e}(Z)$. For a subset $Z$ of $\Spec(A)$, we let $\Gamma_{\infty}(Z)$ be the $\Theta$-closure of $\pi_{\infty}(Z)$. The following lemma is essentially a reformulation of the containment theorem (Theorem~\ref{thm:contain}), and is the most substantial part of the proof of Theorem~\ref{thm:rho-homeo}.

\begin{lemma} \label{lem:rho-1}
Let $Z$ be a closed subset of $\bU^{\cI}$ and let $\fp=\fP(\lambda,e;Z)$. Then $\Gamma_{\lambda,e}(Z)=\rho^{-1}(V_{\fS}(\fp))$.
\end{lemma}

\begin{proof}
Let $\sZ=\rho^{-1}(V_{\fS}(\fp))$. Then Theorem~\ref{thm:contain} shows that $\sZ_{\mu,e}=\Theta^{\lambda,e}_{\mu,d}(Z)$. We now verify that $\sZ$ is $\Theta$-closed:
\begin{itemize}
\item[(D1)] This follows since $\Theta$ always produces a closed subset of $\bU$.
\item[(D2)] We have
\begin{displaymath}
\Theta^{\mu,d}_{\nu,f}(\sZ_{\mu,d}) = \Theta^{\mu,d}_{\nu,f}(\Theta^{\lambda,e}_{\mu,d}(Z)) \subset \Theta^{\lambda,e}_{\nu,f}(Z) = \sZ_{\nu,f}.
\end{displaymath}
The containment above can  be proved directly, but is also a consequence of Theorem~\ref{thm:contain}.
\item[(D3)] Use notation as in the formulation of (D3). Let $N$ be an integer larger than the sum of finite parts of $\lambda$ and all finite parts of $\mu$, and let $n \ge N$. We claim that $\sZ_{\mu(n),d}=\sZ_{\mu,d}$. We have already explained the containment $\sZ_{\mu,d} \subset \sZ_{\mu(n),d}$, so we prove the reverse. Let $x \in \sZ_{\mu(n),d}$ be given. Returning to the definition of $\Theta$, we see that there exists a good pair $(\cE,\phi)$ relative to $\lambda$, $e$, $\mu(n)$, and $d$ such that $x \in i^{-1}(\ol{p(Z)})$, where $i \colon \bA^{\cK} \to \bA^{\cE}$ is the multi-diagonal map associated to $\phi$ and $p \colon \bA^{\cI} \to \bA^{\cE}$ is the projection associated to $\cE \subset \cI$. By the choice of $n$, we see that $(\cE, \phi)$ is also good relative to $\lambda$, $e$, $\mu$, and $d$, and so $x \in \sZ_{\mu, d}$, as required.
\item[(D4)] We have $\sZ_{(\infty^n,1^n)}=\emptyset$ for $n \gg 0$ (this is clear from the definiton of $\Theta$) and $\sZ_{\infty}=\emptyset$.
\end{itemize}
This completes the proof that $\sZ$ is closed.

Now, let $\sW$ be a $\Theta$-closed subset of $\sX$ containing $\pi_{\lambda,e}(Z)$. We must show that $\sZ \subset \sW$. We have $Z \subset \sW_{\lambda,e}$ by assumption. We therefore have
\begin{displaymath}
\sZ_{\mu,d}=\Theta^{\lambda,e}_{\mu,d}(Z) \subset \Theta^{\lambda,e}_{\mu,d}(\sW_{\lambda,e}) \subset \sW_{\mu,d},
\end{displaymath}
where in the last step we used (D2) for $\sW$. Since $\sZ_{\infty}=\emptyset$ we also have $\sZ_{\infty} \subset \sW_{\infty}$. Thus $\sZ \subset \sW$, and the result follows.
\end{proof}

\begin{lemma} \label{lem:rho-2}
Let $Z$ be a closed subset of $\Spec(A)$ corresponding to $\fc \subset A$. Then $\Gamma_{\infty}(Z)=\rho^{-1}(V_{\fS}(\fc R))$.
\end{lemma}

\begin{proof}
Let $\sZ=\rho^{-1}(V_{\fS}(\fc R))$. We first show that $\sZ$ is $\Theta$-closed. We have $\fc \subset \fP(\lambda,e;Y)$ if and only if $Y \subset \bU^{\cI}_A$ lies above $Z \subset \Spec(A)$. We thus see that $\sZ_{\lambda,e}$ is the inverse image of $Z$ in $\bU^{\cI}$. One easily verifies conditions (D1)--(D4) from this description.

Now suppose that $\sW$ is a $\Theta$-closed subset of $\sX$ containing $\pi_{\infty}(Z)$. Thus $Z \subset \sW_{\infty}$. As mentioned before Proposition~\ref{prop:Theta-top}, $\sW_{\lambda,e}$ contains the inverse image of $\sW_{\infty}$, which is exactly $\sZ_{\lambda,e}$. Thus $\sZ \subset \sW$, which completes the proof.
\end{proof}

\begin{lemma} \label{lem:rho-3}
Every $\Theta$-closed subset of $\sX$ is a finite union of sets of the form $\Gamma_{\lambda,e}(Z)$ and $\Gamma_{\infty}(Z)$ with $Z$ irreducible and closed.
\end{lemma}

\begin{proof}
Say that a $\Theta$-closed subset $\sZ$ is \defi{bounded} if $\sZ_{\infty}=\emptyset$. Using an argument similar to the proof of \cite[Proposition~5.1]{svar}, one can show that every $\Theta$-closed set $\sZ$ has the form $\sZ_1 \cup \sZ_2$ where $\sZ_1$ is $\Theta$-closed and bounded, and $\sZ_2=\Gamma_{\infty}(Z)$ for some $Z \subset \Spec(A)$. It thus suffices to prove the lemma for bounded $\Theta$-closed sets.

Using an argument similar to the proof of \cite[Proposition~5.6]{svar} (but with $\Xi$ in place of $\Lambda$), one sees that $\sX$ is a noetherian topological space. Arguing as in \cite[Theorem~5.8]{svar}, one sees that the irreducible bounded $\Theta$-closed subsets of $\sX$ have the form $\Gamma_{\lambda,e}(Z)$. The result now follows. (We note that (D3) and (D4), which have not really been used yet, are crucial in these arguments.)
\end{proof}	

\begin{proof}[Proof of Theorem~\ref{thm:rho-homeo}]
Every closed subset of $\Spec_{\fS}(R)$ is a finite union of sets of the form $V_{\fS}(\fp)$ where $\fp$ is an $\fS$-prime of $R$. By Lemmas~\ref{lem:rho-2} and~\ref{lem:rho-3}, we see that the inverse image of such sets under $\rho$ is closed. Thus $\rho$ is continuous.

By Lemma~\ref{lem:rho-3}, every closed subset of $\sX$ is a finite union of sets of the form $\Gamma_{\lambda,e}(Z)$ and $\Gamma_{\infty}(Z)$. By Lemmas~\ref{lem:rho-2} and~\ref{lem:rho-3}, we see that the images of such sets under $\rho$ are closed. Thus $\rho^{-1}$ is continuous.
\end{proof}

\end{document}